\documentclass{amsart}

\usepackage[english]{babel}
\usepackage{csquotes}
\usepackage[style=alphabetic,
            isbn=false,
            doi=false,
            url=false,
            backend=bibtex,
            maxnames=5,
            maxalphanames = 5,
            firstinits=true
           ]{biblatex}
\bibliography{sources}
\usepackage{amsmath,amsfonts,amsthm,mathtools, amssymb}
\usepackage{xcolor}
\usepackage{tikz-cd}
\usepackage{hyperref}
\usepackage{aliascnt}
\usepackage{xparse}
\usepackage{tensor}
\usepackage{caption}
\usepackage{ mathrsfs }
\usepackage{tabto}
\usepackage{imakeidx}
\usepackage[normalem]{ulem}
\usepackage{array}

\usepackage{tikz}
\usetikzlibrary{matrix,arrows,calc}
\usetikzlibrary{positioning}
\usetikzlibrary{decorations.pathreplacing,decorations.markings,decorations.pathmorphing}
\usetikzlibrary{positioning,arrows,patterns}
\usetikzlibrary{cd}
\usetikzlibrary{intersections}

\def \HoG{.6cm} % height of graphs 
\def \PotlI{.85} % position of top left indices 
\def \PotrI{.85} % position of top right indices 

\tikzset{
	every loop/.style={very thick},
	comp/.style={circle,fill,black,,inner sep=0pt,minimum size=5pt},
	order bottom left/.style={pos=.05,left,font=\tiny},
	order top left/.style={pos=.9,left,font=\tiny},
	order bottom right/.style={pos=.05,right,font=\tiny},
	order top right/.style={pos=.9,right,font=\tiny},
	order node dis/.style={text width=.75cm},
	circled number/.style={circle, draw, inner sep=0pt, minimum size=12pt},
	below left with distance/.style={below left,text height=10pt},
    below right with distance/.style={below right,text height=10pt}
	}

\newcommand{\referee}[1]{\textcolor{black}{ #1}}

\newcolumntype{C}[1]{>{\centering\arraybackslash$}p{#1}<{$}}

\makeatletter
\ifcsname phantomsection\endcsname
    \newcommand*{\@gobblenexttocentry}[9]{}
\else
    \newcommand*{\@gobblenexttocentry}[4]{}
\fi
\newcommand*{\addsubsection}{%
    \addtocontents{toc}{\protect\@gobblenexttocentry}%
    \subsection*}
\makeatother

\begin{document}

\def\subsectionautorefname{Section}
\def\subsubsectionautorefname{Section}
\def\sectionautorefname{Section}
\def\equationautorefname~#1\null{(#1)\null}

\newcommand{\mynewtheorem}[4]{
  % #1=env name
  % #2=displayed name
  % #3=counter to share
  % #4=counter to obey
  \if\relax\detokenize{#3}\relax %#3 empty
    \if\relax\detokenize{#4}\relax %#4 empty
      \newtheorem{#1}{#2}
    \else
      \newtheorem{#1}{#2}[#4]
    \fi
  \else
    \newaliascnt{#1}{#3}
    \newtheorem{#1}[#1]{#2}
    \aliascntresetthe{#1}
  \fi
  \expandafter\def\csname #1autorefname\endcsname{#2}
}

%%%%%%%%%%%%%%%%%%%%%%%%%%%%
%%%% Benutzerteil: %%%%%%%%%%%%%%%%%
%%%%%%%%%%%%%%%%%%%%%%%%%%%%
\mynewtheorem{theorem}{Theorem}{}{section}
\mynewtheorem{lemma}{Lemma}{theorem}{}
\mynewtheorem{rem}{Remark}{lemma}{}
\mynewtheorem{prop}{Proposition}{lemma}{}
\mynewtheorem{cor}{Corollary}{lemma}{}
\mynewtheorem{definition}{Definition}{lemma}{}
\mynewtheorem{question}{Question}{lemma}{}
\mynewtheorem{assumption}{Assumption}{lemma}{}
\mynewtheorem{example}{Example}{lemma}{}
\mynewtheorem{conj}{Conjecture}{}{section}

%%%%%%%%%%%%%%%%%%%%%%%%%%%%
%%%%%%%%%%%%%%%%%%%%%%%%%%%%

\def\defbb#1{\expandafter\def\csname b#1\endcsname{\mathbb{#1}}}
\def\defcal#1{\expandafter\def\csname c#1\endcsname{\mathcal{#1}}}
\def\deffrak#1{\expandafter\def\csname frak#1\endcsname{\mathfrak{#1}}}
\def\defop#1{\expandafter\def\csname#1\endcsname{\operatorname{#1}}}
\def\defbf#1{\expandafter\def\csname b#1\endcsname{\mathbf{#1}}}

\makeatletter
\def\defcals#1{\@defcals#1\@nil}
\def\@defcals#1{\ifx#1\@nil\else\defcal{#1}\expandafter\@defcals\fi}
\def\deffraks#1{\@deffraks#1\@nil}
\def\@deffraks#1{\ifx#1\@nil\else\deffrak{#1}\expandafter\@deffraks\fi}
\def\defbbs#1{\@defbbs#1\@nil}
\def\@defbbs#1{\ifx#1\@nil\else\defbb{#1}\expandafter\@defbbs\fi}
\def\defbfs#1{\@defbfs#1\@nil}
\def\@defbfs#1{\ifx#1\@nil\else\defbf{#1}\expandafter\@defbfs\fi}
\def\defops#1{\@defops#1,\@nil}
\def\@defops#1,#2\@nil{\if\relax#1\relax\else\defop{#1}\fi\if\relax#2\relax\else\expandafter\@defops#2\@nil\fi}
\makeatother

%%%%%%%%%%%%%%%%%%%%%%%%%%%%
%%%% Benutzerteil: %%%%%%%%%%%%%%%%%
%%%%%%%%%%%%%%%%%%%%%%%%%%%%
\defbbs{BZHQCNPALRVW}
\defcals{DOPQMNXYLTRAEHZKCFIG}
\deffraks{apijklmnopqueRB}
\defops{Gh, PGL,SL,Sp,mod,Spec,Re,Gal,Tr,End,GL,Hom,PSL,H,div,Aut,rk,Mod,R,T,Tr,Mat,Vol,MV,Res,vol,Z,diag,Hyp,ord,Im,ev,U,dev,c,CH,fin,pr,Pic,lcm,ch,td,LG,id,Sym,Aut,hor,lev,rel,stab, SU,PU,tor}
\defbfs{kiuvzwpds} % only use this for small letters, otherwise you go directly to hell!!!
%%%%%%%%%%%%%%%%%%%%%%%%%%%%
%%%%%%%%%%%%%%%%%%%%%%%%%%%%

\def\ep{\varepsilon}
\def\abs#1{\lvert#1\rvert}
\def\dd{\mathrm{d}}
\def\inj{\hookrightarrow}
\def\eq{=}
\newcommand{\hyp}{{\rm hyp}}
\newcommand{\odd}{{\rm odd}}

\def\i{\mathrm{i}}
\def\e{\mathrm{e}}
\def\st{\mathrm{st}}
\def\ct{\mathrm{ct}}

\def\uC{\underline{\bC}}
\def\ol{\overline}
  
\def\Vrel{\bV^{\mathrm{rel}}}
\def\Wrel{\bW^{\mathrm{rel}}}
\def\twolev{\mathrm{LG_1(B)}}

%%%%%%%%%%%%%% equations
\def\be{\begin{equation}}   \def\ee{\end{equation}}     \def\bes{\begin{equation*}}    \def\ees{\end{equation*}}
\def\ba{\be\begin{aligned}} \def\ea{\end{aligned}\ee}   \def\bas{\bes\begin{aligned}}  \def\eas{\end{aligned}\ees}
\def\={\;=\;}  \def\+{\,+\,} \def\m{\,-\,}

%%%%%%%%%%%%% moduli spaces
\newcommand*{\proj}{\mathbb{P}}
\newcommand{\barmoduli}[1][g]{{\overline{\mathcal M}}_{#1}}
\newcommand{\moduli}[1][g]{{\mathcal M}_{#1}}
\newcommand{\omoduli}[1][g]{{\Omega\mathcal M}_{#1}}
\newcommand{\komoduli}[1][g]{{\Omega^k\mathcal M}_{#1}}
\newcommand{\modulin}[1][g,n]{{\mathcal M}_{#1}}
\newcommand{\omodulin}[1][g,n]{{\Omega\mathcal M}_{#1}}
\newcommand{\zomoduli}[1][]{{\mathcal H}_{#1}}
\newcommand{\barzomoduli}[1][]{{\overline{\mathcal H}_{#1}}}
\newcommand{\pomoduli}[1][g]{{\proj\Omega\mathcal M}_{#1}}
\newcommand{\pomodulin}[1][g,n]{{\proj\Omega\mathcal M}_{#1}}
\newcommand{\pobarmoduli}[1][g]{{\proj\Omega\overline{\mathcal M}}_{#1}}
\newcommand{\pobarmodulin}[1][g,n]{{\proj\Omega\overline{\mathcal M}}_{#1}}
\newcommand{\potmoduli}[1][g]{\proj\Omega\tilde{\mathcal{M}}_{#1}}
\newcommand{\obarmoduli}[1][g]{{\Omega\overline{\mathcal M}}_{#1}}
\newcommand{\obarmodulio}[1][g]{{\Omega\overline{\mathcal M}}_{#1}^{0}}
\newcommand{\otmoduli}[1][g]{\Omega\tilde{\mathcal{M}}_{#1}}
\newcommand{\pom}[1][g]{\proj\Omega{\mathcal M}_{#1}}
\newcommand{\pobarm}[1][g]{\proj\Omega\overline{\mathcal M}_{#1}}
\newcommand{\pobarmn}[1][g,n]{\proj\Omega\overline{\mathcal M}_{#1}}
\newcommand{\princbound}{\partial\mathcal{H}}
\newcommand{\omoduliinc}[2][g,n]{{\Omega\mathcal M}_{#1}^{{\rm inc}}(#2)}
\newcommand{\obarmoduliinc}[2][g,n]{{\Omega\overline{\mathcal M}}_{#1}^{{\rm inc}}(#2)}
\newcommand{\pobarmoduliinc}[2][g,n]{{\proj\Omega\overline{\mathcal M}}_{#1}^{{\rm inc}}(#2)}
\newcommand{\otildemoduliinc}[2][g,n]{{\Omega\widetilde{\mathcal M}}_{#1}^{{\rm inc}}(#2)}
\newcommand{\potildemoduliinc}[2][g,n]{{\proj\Omega\widetilde{\mathcal M}}_{#1}^{{\rm inc}}(#2)}
\newcommand{\omoduliincp}[2][g,\lbrace n \rbrace]{{\Omega\mathcal M}_{#1}^{{\rm inc}}(#2)}
\newcommand{\obarmoduliincp}[2][g,\lbrace n \rbrace]{{\Omega\overline{\mathcal M}}_{#1}^{{\rm inc}}(#2)}
\newcommand{\obarmodulin}[1][g,n]{{\Omega\overline{\mathcal M}}_{#1}}
\newcommand{\LTH}[1][g,n]{{K \overline{\mathcal M}}_{#1}}
\newcommand{\PLS}[1][g,n]{{\bP\Xi \mathcal M}_{#1}}

\DeclareDocumentCommand{\okmoduli}{ O{g} O{k}}{{\Omega^{#2}\mathcal M}_{#1}}
\DeclareDocumentCommand{\LMS}{ O{\mu} O{g,n} O{}}{\Xi\overline{\mathcal{M}}^{#3}_{#2}(#1)}
\DeclareDocumentCommand{\kLMS}{ O{\mu} O{g,n} O{k} O{}}{\Xi^{#3}\overline{\mathcal{M}}^{#4}_{#2}(#1)}
\DeclareDocumentCommand{\Romod}{ O{\mu} O{g,n} O{}}{\Omega\mathcal{M}^{#3}_{#2}(#1)}
\newcommand*{\Eq}{\mathrm{Eq}}
\newcommand*{\Tw}[1][\Lambda]{\mathrm{Tw}_{#1}}  %twist group
\newcommand*{\sTw}[1][\Lambda]{\mathrm{Tw}_{#1}^s}  %simple twist group

% Boldface letters
\newcommand{\bfa}{{\bf a}}
\newcommand{\bfb}{{\bf b}}
\newcommand{\bfd}{{\bf d}}
\newcommand{\bfe}{{\bf e}}
\newcommand{\bff}{{\bf f}}
\newcommand{\bfg}{{\bf g}}
\newcommand{\bfh}{{\bf h}}
\newcommand{\bfm}{{\bf m}}
\newcommand{\bfn}{{\bf n}}
\newcommand{\bfp}{{\bf p}}
\newcommand{\bfq}{{\bf q}}
\newcommand{\bfP}{{\bf P}}
\newcommand{\bfR}{{\bf R}}
\newcommand{\bfU}{{\bf U}}
\newcommand{\bfu}{{\bf u}}
\newcommand{\bfz}{{\bf z}}
\newcommand{\bfG}{{\bf G}}

\newcommand{\bfl}{{\boldsymbol{\ell}}}
\newcommand{\bfmu}{{\boldsymbol{\mu}}}
\newcommand{\bfeta}{{\boldsymbol{\eta}}}
\newcommand{\bfomega}{{\boldsymbol{\omega}}}
\newcommand{\bfsigma}{{\boldsymbol{\sigma}}}
\newcommand{\bftau}{{\boldsymbol{\tau}}}
\newcommand{\Stab}{\operatorname{Stab}}
\newcommand{\cl}{\operatorname{cl}}

\newcommand{\wh}{\widehat}
\newcommand{\wt}{\widetilde}
\newcommand{\whmu}{\widehat{\mu}}
\newcommand{\whLa}{\widehat{\Lambda}}

\newcommand{\ps}{\mathrm{ps}}  
\newcommand{\pmarked}{\mathrm{mp}}

\newcommand{\tdpm}[1][{\Gamma}]{\mathfrak{W}_{\operatorname{pm}}(#1)}
\newcommand{\tdps}[1][{\Gamma}]{\mathfrak{W}_{\operatorname{ps}}(#1)}

\newlength{\halfbls}\setlength{\halfbls}{.5\baselineskip}
\newlength{\halbls}\setlength{\halfbls}{.5\baselineskip}

\newcommand*{\Hrel}{\cH_{\text{rel}}^1}
\newcommand*{\Hrelbar}{\overline{\cH}^1_{\text{rel}}}
\newcommand*{\HrelB}{\overline{\cH}^1_{\text{rel},B}}

\newcommand*\interior[1]{\mathring{#1}}

\newcommand{\prodt}[1][j]{ t_{\lceil #1 \rceil}}
\newcommand{\prodtL}[1][\lceil L \rceil]{t_{#1}}

\DeclareDocumentCommand{\MSgrp}{ O{\mu} }{\mathcal{MS}_{#1}}
\NewDocumentCommand{\cherry}{O{i,j} O{p,q}}{{}_{#1} \Lambda_{#2}}

\newcommand{\whGtau}{\widehat \Gamma_{\!\tau}}
\newcommand{\area}{{\rm area}}
\newcommand{\PPer}{{\rm PPer}}
\newcommand{\tui}[1][i]{t_{\lceil #1 \rceil}}

%%%%%%%%%%%%%%%%%%%%%%%%%
%%%%%%%%%%%%%%%%%%%%%%%%%

\title[Chern classes of linear submanifolds]
{Chern classes of linear submanifolds \\
with application to spaces of $k$-differentials \\
and ball quotients }

\author{Matteo Costantini}
\email{matteo.costantini@uni-due.de}
\address{Institut f\"ur Mathematik, Universit\"at Duisburg-Essen,
	45117 Essen, Germany}
\author{Martin M\"oller}
\email{moeller@math.uni-frankfurt.de}
\address{Institut f\"ur Mathematik, Goethe-Universit\"at Frankfurt,
Robert-Mayer-Str. 6-8,
60325 Frankfurt am Main, Germany}
\author{Johannes Schwab}
\email{schwab@math.uni-frankfurt.de}
\address{Institut f\"ur Mathematik, Goethe-Universit\"at Frankfurt,
Robert-Mayer-Str. 6-8,
60325 Frankfurt am Main, Germany}

\thanks{Research of J.S and M.M.\ is supported
by the DFG-project MO 1884/2-1 and the Collaborative Research Centre
TRR 326 ``Geometry and Arithmetic of Uniformized Structures''.}
\thanks{Research of M.C. has been supported by the DFG
  Research Training Group 2553.}

\begin{abstract}
We provide formulas for the Chern classes of linear submanifolds of the
moduli spaces of Abelian differentials and hence for their Euler characteristic.
This includes as special case the moduli spaces of $k$-differentials, for
which we set up the full intersection theory package and implement it
in the SageMath package \texttt{diffstrata}.
\par
As an application, we give an algebraic proof of the theorems
of Deligne-Mostow and Thurston that suitable compactifications of
moduli spaces of $k$-differentials on the $5$-punctured projective line
with weights satisfying the INT-condition are quotients of the
complex two-ball.
\end{abstract}
\maketitle
\tableofcontents

%%%%%%%%%%%%%%%%%%%%%%%%%%%%%%%%%%%%%%%%%%%%%%%%%%%%%%%%%%%%
% !TEX root = EClinsub.tex

%%%%%%%%%%%%%%%%%%%%%%%%%%%%%%%%%%%%%%%%%%%%%%%%%%%%%%%%%%
\section{Introduction}
%%%%%%%%%%%%%%%%%%%%%%%%%%%%%%%%%%%%%%%%%%%%%%%%%%%%%%%%%%

Linear submanifolds  are the most interesting and well-studied
subvarieties of moduli spaces of Abelian differentials $\omoduli[g,n](\mu)$
and their classification seems far from complete at present. They
are defined as the normalization  of algebraic substacks of~$\omoduli[g,n](\mu)$
that are locally a union of linear subspaces in period coordinates. In the holomorphic
case, linear submanifolds defined by real linear equations are precisely the
closures of $\GL_2^+(\bR)$-orbits by the fundamental theorems of
Eskin-Mirzakhani-Mohammadi (\cite{EsMi}, \cite{EsMiMo}).
These orbit closures are automatically algebraic subvarieties
by Filip's theorem (\cite{Filip}). Our results require algebraicity, but
they work as well for meromorphic differentials and for subvarieties whose
equations are only $\bC$-linear.
\par
Linear submanifolds include
\begin{itemize}
\item spaces of quadratic differentials,
\item Teichmüller curves, 
\item eigenform loci and Prym loci,
\item the recent sporadic examples from \cite{MMW} and \cite{emmw}, but also
\item spaces defined by covering constructions, and
\item in the meromorphic case, spaces defined by residue conditions.
\end{itemize}
These examples are $\bR$-linear. Spaces of $k$-differentials for
$k \geq 2$ and in particular the ball quotients in Section~\ref{sec:BQ} are
prominent examples that are only $\bC$-linear.
\par
Our primary goal is a formula for the Chern classes of the cotangent
bundle of any linear submanifold or rather of its compactification.
The Euler characteristic is an intrinsic compactification-independent
application. Knowing the Chern classes is a prerequisite for understanding the
birational geometry of linear submanifolds, such as computations of the
Kodaira dimension, see \cite{CCM}.
\par
This goal was achieved in \cite{CMZeuler} for the full projectivized strata
of Abelian differentials $\bP\omoduli[g,n](\mu)$ themselves, taking the modular
smooth normal crossing compactification $\bP\LMS$ of multi-scale differentials
from \cite{LMS} as point of departure. 
In the inextricable zoo of linear manifolds we are not aware of any
intrinsic way to construct a smooth compactification with modular properties.
Working with the normalization of the closure in some ambient compactification
is usually unsuitable for intersection theory computations. Here, however,
thanks to the work of Benirschke-Dozier-Grushevsky (\cite{BDG}) and some
minor upgrades we are able to work with this closure.
\par
We now introduce  more notation to state the general results and then apply
them to specific linear submanifolds. Let $\Omega\cH \to  \omoduli[g,n](\mu)$ be
a linear submanifold. Let moreover $\cH \to \bP\omoduli[g,n](\mu)$ be its
projectivization and let $\ol \cH \to \bP\LMS$ denote the normalization of its
closure into the space of multi-scale differentials. The boundary strata~$D_\Gamma$
of $\bP\LMS$ are indexed by level graphs~$\Gamma$ as we recall in
\autoref{sec:boundarycomb}. By \cite[Theorem~1.5]{BDG} the boundary of~$\ol \cH$
is divisorial and consists two types of divisors: First there are the
divisors $D^\cH_{\text{h}}$ of
curves whose level graphs have only horizontal edges (i.e.\ joining vertices of
the same level). Second there are the divisors $D^\cH_\Gamma$ parameterized by level
graphs~$\Gamma \in \LG_1(\cH)$ that have one level below the zero level and no
horizontal edges and such that the intersection of $\ol{\cH}$ with the interior
of the boundary divisor~$D_\Gamma$  is non-empty.
Those boundary divisors~$D^\cH_\Gamma$ come with the integer~$\ell_\Gamma$,
the least common multiple of the prongs~$\kappa_e$ along the edges. The
interior of $D_\Gamma$ can intersect the linear submanifold $\ol \cH$ in finitely
many irreducible components, whose number we denote by $n_\Gamma(\ol \cH)$. We
denote by $\LG^+_1(\cH) $ the set of pairs $\Gamma^+=(\Gamma,i)$ where
$\Gamma\in \LG_1(\cH)$ and $i\in \{1,\dots,n_\Gamma(\ol \cH)\}$ is the index
set of irreducible components of  $D_\Gamma^\circ\cap \ol \cH$. We refer to
$\Gamma^+=(\Gamma,i)$ as a \emph{refined level graph} and
set $\ell_{(\Gamma,i)}:=\ell_\Gamma$. This extra notational complexity is necessary since a priori it is possible that a linear submanifold intersects a boundary component in irreducible components with different level dimensions. This is clearly possible since for example we do not require our linear manifolds to be irreducible, which is
convenient in order to include entire strata of $k$-differentials, but it could be possible also in the case of irreducible linear submanifolds.
We let $\xi = c_1(\cO(-1))$ be the first Chern class of the tautological bundle
on~$\ol \cH$.
\par
\begin{theorem} \label{thm:c1cor}
The first Chern class of the logarithmic cotangent bundle of a
projectivized compactified linear submanifold $\ol \cH$ is
\be \label{eq:firstChern}
\c_1(\Omega^1_{\ol{\cH}}(\log \partial\cH)) \= N \cdot \xi + \sum_{\Gamma^+ \in 
\LG_1^+(\cH)} (N-N_{\Gamma^+}^\top) \ell_{\Gamma^+}  [D^\cH_{\Gamma^+}] \qquad \in \CH^1(\ol{\cH})\,,
\ee
where $N := \dim(\Omega \cH)$ and where $N_{\Gamma^+}^\top:=\dim(D^{\cH,\top}_{\Gamma^+})+1$
is the dimension of the unprojectivized top level stratum in~$D^\cH_{\Gamma^+}$.
\end{theorem}
\par
To state a formula for the full
Chern character we need to recall a procedure that also determines adjacency
of boundary strata. It is given by undegeneration maps~$\delta_i$ that
contract all the edges except those that cross from level~$-i+1$ to level~$-i$, 
see \autoref{sec:boundarycomb}.
This construction can obviously be generalized so that a larger
subset of levels remains. For example the  undegeneration map~$\delta_i^\complement$
contracts only the edges crossing from level~$-i+1$ to level~$-i$. For any element $\Gamma$ of the set $ \LG_L(\cH)$ of graphs with  $L$ levels below zero and without 
horizontal edges, we can now
define  the boundary component $D^\cH_\Gamma$ of codimension $L$ and the
quantity $\ell_\Gamma = \prod_{i=1}^L \ell_{\delta_{i}(\Gamma)}$. We  also extend the undegeneration maps at the level of refined level graphs, i.e., for
elements in $\LG^+_L(\cH)$, which we define analogously to $\LG^+_1(\cH) $,
and we still denote them by the same letter.
\par
\begin{theorem} \label{intro:Chern}
The Chern character of the logarithmic cotangent
bundle is 
\bes
\ch(\Omega^1_{\overline{\cH}}(\log \partial\cH)) \= e^{\xi} \cdot \sum_{L=0}^{N-1}
\sum_{ \Gamma^+ \in \LG^+_L(\cH)} \!\!\!\!\!\!
\ell_{\Gamma^+}\left(N-N_{\delta_{L}(\Gamma^+)}^\top\right) \fraki_{\Gamma^+ *}
\prod_{i=1}^{L} \td \left(\cN_{\Gamma^+/\delta_{i}^\complement(\Gamma^+)}^{\otimes -\ell_{\delta_i(\Gamma^+)}} \right)^{-1} \!\!\!\!, 
\ees
where $\cN_{\Gamma^+/\delta_{i}^\complement(\Gamma^+)}$ denotes the normal bundle 
of $D^\cH_{\Gamma^+}$ in $D^\cH_{\delta_{i}^\complement(\Gamma^+)}$, where $\td$ is the
Todd class  and $\fraki_{\Gamma^+}: D^\cH_{\Gamma^+}\hookrightarrow \overline{\cH}$
is the inclusion map.
\end{theorem}
\par
So far the results have been stated to parallel exactly those in \cite{CMZeuler}. The ambient spaces can be mildly singular (see \autoref{sec:closurelinsection}), but the maps $i_{\Gamma^+}$ are regular embeddings (see \autoref{subsec:normalbundles}) which allows us to work tacitly with operational Chow groups just as in the case of the smooth DM stack in  \cite{CMZeuler}.
We start explaining the difference in evaluating this along with the next
result, a closed formula for the Euler characteristic.
\par
\begin{theorem} \label{intro:ECformula}
Let $\cH \to \bP\omoduli[g,n](\mu)$ be a projectivized linear submanifold.
The orbifold Euler characteristic of $\cH$ is given by
\bes
\chi(\cH) \= (-1)^d \sum_{L=0}^d \sum_{\Gamma^+ \in \LG^+_L(\cH)}
\frac{K_{\Gamma^+}^{\cH} \cdot N_{\Gamma^+}^\top }{|\Aut_{\cH}(\Gamma^+)| }
\cdot \prod_{i=0}^{-L} \int_{\cH_{\Gamma^+}^{[i]}}
\xi^{d_{\Gamma^+}^{[i]}}_{\cH_{\Gamma^+}^{[i]}},
\ees
where the  integrals are over the normalization of the closure $\ol \cH\to \bP\LMS$
inside the moduli space of multi-scale differentials and similar
integrals over boundary strata, where
\begin{itemize}
\item $\cH_{\Gamma^+}^{[i]}$ are the linear submanifolds at level $i$ of $\Gamma^+$ as
defined in \autoref{sec:closurelin},
\item $d_{\Gamma^+}^{[i]} := \dim(\cH_{\Gamma^+}^{[i]})$ is the projectivized dimension,
\item $K_{\Gamma^+}^{\cH}$ is the product of the number of prong-matchings
on each edge of $\Gamma$ that are actually contained in $D_{\Gamma^+}^\cH$,
\item $\Aut_{\cH}(\Gamma^+)$ is the set of automorphism of the
graph $\Gamma$ whose induced action on a neighborhood of $D_{\Gamma^+}^\cH$
preserves $\ol \cH$,
\item $d := \dim(\cH)$ is the projectivized dimension.
\end{itemize}
\end{theorem}
\par
The number of \emph{reachable prong matchings} $K_{\Gamma^+}^{\cH}$ and the
number $|\Aut_{\cH}(\Gamma^+)|$ as
defined in the theorem are in general non-trivial to determine. Also
the description of $\cH_{\Gamma^+}^{[i]}$ requires specific investigation.
For example, for strata of $k$-differentials, these $\cH_{\Gamma^+}^{[i]}$
are again some strata of $k$-differentials, but the markings of the
edges have to be counted correctly.
\par
The most important obstacle to evaluate this formula however is to compute
the fundamental classes of linear submanifolds, or
to use tricks to avoid this. For strata of Abelian differentials, this
step was provided by the recent advances in relating fundamental classes
to Pixton's formula (\cite{HolmesSchmitt}, \cite{BHPSS}). Whenever we
have the fundamental classes at our disposal, we can evaluate expressions
in the tautological ring, as we briefly summarize in \autoref{sec:nb}.
\par
\medskip
\paragraph{\textbf{Applications: Teichm\"uller curves in genus two}}
As an example where fundamental class considerations can be avoided, we give in \autoref{sec:examples} an alternative quick proof of one of the first computations of Euler characteristics of Teichm\"uller curves, initially proven in \cite{Bai07}, see also \cite{MoeZag} for a proof via theta derivatives.
\par
\begin{theorem}[Bainbridge] \label{thm:ECWDintro}
  The orbifold Euler characteristic of the Teichm\"uller
curve~$W_D\subset  \bP\omoduli[2,1](2)$
in the eigenform locus for real multiplication by a non-square discriminant~$D$
is $\chi(W_D) = -9 \zeta(-1)$ where
$\zeta = \zeta_{\bQ(\sqrt D)}$ is the Dedekind zeta function.
\end{theorem}

\paragraph{\textbf{Strata of $k$-differentials}} The space of quadratic
differentials is the cotangent space to moduli space of curves and thus fundamental
in Teichm\"uller dynamics. We give formulas for Chern classes, Euler
characteristics and for the intersection theory in these spaces. In fact,
our formulas work uniformly for spaces of $k$-differentials for all~$k \geq 1$.
Having the quadratic case in mind, we write $\ol{\cQ} = \bP \Xi^k \ol \cM_{g,n}(\mu)$ for the space of  multi-scale $k$-differentials defined in \cite{CoMoZa}.
The space $\ol{\cQ}$ is the disjoint union over all divisors~$d$ of~$k$ of the subspaces parametrizing powers of $k/d$-differentials. We write
$\ol{\cQ}_{\mathrm{pr}} = \bP \Xi_{\mathrm{pr}} ^k \ol \cM_{g,n}(\mu)$ for the (still possibly disconnected) subspace of \emph{primitive multi-scale $k$-differentials}, the closure of the components corresponding to $d=1$. The space  $\ol{\cQ}$
coincides (up to explicit isotropy groups, see \autoref{lem:deg_d}) with the
compactification as above of the linear submanifolds associated to its connected components obtained via the canonical covering construction.
\par
The formulas in \autoref{intro:Chern} apply to the connected components of $\ol{\cQ}$ viewed as linear
submanifolds in some higher genus stratum~$\cM_{\wh g,\wh n}(\wh \mu)$. However
the fundamental class of these submanifolds is not known, conceivably it is not
even a tautological class. The main challenge here is to convert these formulas
into formulas that can be evaluated on $\ol{\cQ}$ viewed as a submanifold
in $\barmoduli[g,n]$ where the fundamental class is given by Pixton's formula.
\par
While the boundary strata of the moduli space $\bP\LMS$ are indexed by level
graphs, the boundary strata of the moduli space of multi-scale
$k$-differentials $\ol{\cQ}$ are indexed by $k$-coverings of level graphs
$\pi : \widehat \Gamma_\pmarked \to \Gamma$, where the legs of
$\widehat \Gamma_\pmarked$ are marked only partially, see \autoref{sec:kdiff}
or also \cite[Section~2]{CoMoZa} for the definitions of these objects
and the labeling conventions of those covers. The $k$-coverings appearing in the boundary of $\ol{\cQ}_{\mathrm{pr}}$ are precisely those with $\widehat{\Gamma}$ connected. Each edge $e \in \Gamma$ has an associated $k$-enhancement $\kappa_e$ given by $|\ord_{e} \omega + k|$, where $\omega$ is the $k$-differential on a generic point of the associated boundary stratum $D_\pi$.   We let $\zeta = c_1(\cO(-1))$ be the first Chern class of the tautological bundle on $\ol \cQ$.
Via the canonical cover construction, \autoref{intro:ECformula}
implies the following formula for the Euler characteristic of strata
of $k$-differentials.  Note that if $\cH_k$ is the linear submanifold associated to a connected component of a stratum of $k$-differentials,  the information of $\pi$ is enough to uniquely determine the relevant information of the irreducible components $D_{\Gamma^+}$ of $D_\pi$, i.e. the  level strata dimensions  $d_{\Gamma^+}^{[i]}$, the number of reachable prong-matchings $K_{\Gamma^+}^{\cH_k}$ and $|\Aut_{\cH_k}(\Gamma^+)|$. So in the applications of the formulas of  \autoref{intro:Chern} and \autoref{intro:ECformula} to strata of $k$-differentials, we can group together all irreducible components of $D_\pi$. 
\par
\begin{cor} \label{cor:kstrata}
The orbifold Euler characteristic of a projectivized stratum of
$k$-differentials $\bP\okmoduli[g,n](\mu)$ is given by
\begin{multline*}
\chi(\bP \Omega^k_{\mathrm{pr}}\moduli[g,n](\mu)) = \\
\left(\frac{-1}{k}\right)^d \sum_{L=0}^d
\sum_{(\pi : \widehat \Gamma_\pmarked \to \Gamma) \in \LG_L(\cQ_{\mathrm{pr}})} S(\pi) \cdot
\frac{N_\pi^\top \cdot \prod_{e \in E(\Gamma)} \kappa_e,}
{|\Aut(\Gamma)|} \cdot \prod_{i=0}^{-L}
\int_{\cQ_\pi^{[i]}} \zeta^{d_\pi^{[i]}}_{\cQ_\pi^{[i]}},
\end{multline*}
where $S(\pi)$ is the normalized size of a stabilizer of a totally labeled version of the graph $\wh \Gamma_\pmarked$ and $\cQ_\pi^{[i]}$ are the strata of $k$-differentials of $D_\pi$ at level $i$.
\end{cor}
The full definition of $S(\pi)$ is presented in~\eqref{eq:defSpi}. It equals one for
many~$\pi$, e.g. if all vertices in~$\Gamma$ have only one preimage
in~$\wh\Gamma_\pmarked$. See \autoref{rem:Spi1} for values of this combinatorial
constant.
\par
\newlength{\largeswidth}
\settowidth{\largeswidth}{$-\frac{1}{40}$}
\begin{figure}[h]
$$ \begin{array}{|c|*{9}{C{\largeswidth}|}}
\hline  &&&&&&&&& \\ [-\halfbls] 
k  & 1 & 2 & 3 & 4 & 5 & 6 & 7 & 8 & 9 \\
[-\halfbls] &&&&&&&&& \\ 
\hline &&&&&&&&& \\ [-\halfbls]
\chi(\bP \Omega^k_{\mathrm{pr}}\cM_{2,1}(2k)) %\,(\oamoduli(1^{2g-2})) 
& -\frac{1}{40} & 0& \frac{1}{3}  & \frac{3}{2} & \frac{21}{5} & 9 & 18 & 30 & 51 \\
[-\halfbls] &&&&&&&&& \\
\hline
\end{array}
$$
\captionof{table}[foo2]{Euler characteristics of some minimal strata of primitive $k$-differentials. Note that in the case of $k=2$ the stratum of primitive minimal $k$-differentials is empty, see \cite[Theorem~2(c)]{MaSmillie}.}
\label{cap:EulerMinimal}
\end{figure}
Table~\ref{cap:EulerMinimal} %% DO NOT USE AUTOREF ->"Figure 1"
gives the Euler characteristics of some strata of primitive quadratic differentials, for more examples and cross-checks see \autoref{sec:values}.
\par
All the formulas for evaluations in the tautological ring of strata
of $k$-differentials have been coded in an extension of the SageMath package
\texttt{diffstrata} (an extension of \texttt{admcycles} by \cite{DSvZ})
that initially had this functionality for Abelian differentials only
(see \cite{CMZeuler}, \cite{CoMoZadiffstrata}).
See \autoref{sec:nb} for generalities on tautological ring computations and in
particular \autoref{sec:kdiff} for the application to $k$-differentials.
The program \texttt{diffstrata} has been used to verify the Hodge-DR-conjecture
from \cite{CGHMS} in low genus. Moreover, \texttt{diffstrata} confirms
that the values of the tables in \cite{Goujard} can be obtained via intersection theory computations:
\par
\begin{prop}
The Conjecture~1.1 in \cite{CMS} expressing Masur-Veech volumes for strata
of quadratic differentials as intersection numbers holds true for strata
of projectivized dimension up to six, e.g.\ $\cQ(12) =5614/6075 \cdot \pi^6$.
\end{prop}
\par
\medskip
\paragraph{\textbf{Ball quotients}}
Deligne-Mostow (\cite{DeligneMostow86}) and Thurston \cite{thurstonshapes}
constructed compactifications of strata of $k$-differentials on $\moduli[0,n]$ for
very specific choices of~$\mu$ and showed that these compactified strata are
quotients of the complex $(n-3)$-ball. These results were celebrated as they
give a list of non-arithmetic ball quotients, of which there are still only finitely many sporadic examples today, see  \cite{DPP16} and \cite{DerauxPU3} for recent
progress. The compactifications are given as GIT quotients (in \cite{DeligneMostow86})
or in the language of cone manifolds (in \cite{thurstonshapes}) and the proof of
the discreteness of the monodromy representation requires delicate arguments for
extension of the period at the boundary, resp.\ surgeries for the cone manifold
completion. 
\par
As application of our Chern class formulas we give a purely algebraic proof
that these compactifications are ball quotients, based on the fact that
the equality case in the Bogomolov-Miyaoka-Yau inequality implies a ball quotient
structure, see \autoref{prop:BQcrit}. Since this is a proof of concept, we restrict
to the case $n=5$, i.e.\ to quotients of the complex two-ball, and to the
condition $\text{INT}$ in~\eqref{eq:INT}, leaving the analog for Mostow's generalized
$\Sigma\text{INT}$-condition \cite{Mostow} for the reader.
\par
The computation of the hyperbolic volume of these ball quotients had been open
for a long time. A solution has been given by McMullen \cite{McMullenGaussBonnet}
and Koziarz-Nguyen \cite{KozNgu}, see also \cite{KM}. Since computing the hyperbolic
volume is equivalent to computing the Euler characteristic by Gauss-Bonnet, our
results provide alternative approach to this question, too.
\par
For spaces $\ol{\cQ}$ of multi-scale $k$-differentials in $g=0, n=5$ with these conditions, there are only four kinds of boundary divisors:
\begin{itemize}
\item The divisors $\Gamma_{ij}$ where two points with $a_i + a_j < k$ collide.
\item The divisors $L_{ij}$ where two points with $a_i + a_j > k$ collide.
\item The 'horizontal' boundary divisor $D_{\hor}$ consisting of all components
  where two points with $a_i + a_j = k$ collide.
\item The 'cherry' boundary divisors $ _{ij}\Lambda_{kl}$.
  \end{itemize}
\par
\begin{theorem} \label{intro:BQcertificate} Suppose that $\mu=(-a_1,\ldots,-a_5)$
is a tuple with $a_i \geq 0$ and with the condition
\be \label{eq:INT}
\Bigl(1-\frac{a_i}k - \frac{a_j}k\Bigr)^{-1} \in \bZ \qquad  \qquad 
\text{if $a_i + a_k < k$} \qquad\qquad
\hfill \text{(INT)}
\ee
for all $i \neq j$. Then there exists a birational contraction morphism $\ol{\cQ} \to
\ol\frakB$ onto a smooth proper DM-stack $\ol\frakB$ that contracts precisely all the
divisors $L_{ij}$ and $ _{ij}\Lambda_{kl}$. The target~$\ol\frakB$ satisfies
the Bogomolov-Miyaoka-Yau equality for $\Omega^1_{\ol\frakB }(\log D_{\hor})$.
\par
As a consequence $\frakB=\ol\frakB\setminus D_{\hor}$ is a ball quotient. 
\end{theorem}
\par
The signature of the intersection form on the eigenspace that $k$-differentials
are modeled on has been computed by Veech \cite{VeechFlat}.  The only other case where
the signature is $(1,2)$ are strata in $\moduli[1,3]$. As observed by
Ghazouani-Pirio in \cite{GP1}, (see also \cite{GP2}) there are only few cases
where the metric completion of the strata can be a ball quotient. However
they also find additional cases where the monodromy of the stratum is discrete.
This implies that the period map descends to a map from the compactified
stratum to a ball quotient. It would be interesting to investigate if there
are more such cases, possibly with non-arithmetic monodromy.
\par
\subsection*{Acknowledgments} We thank Selim Ghazouani, Daniel Greb and
Vincent Koziarz for helpful remarks  on ball quotients and the
Bogomolov-Miyaoka-Yau equality and Frederik Benirschke und Johannes Schmitt
for their comments. We thank the referee for carefully reading the paper and for examples that provided further cross-checks of the Euler characteristic formula.

%%%%%%%%%%%%%%%%%%%%%%%%%%%%%%%%%%%%%%%%%%%%%%%%%%%%%%%%%%%%

%%%%%%%%%%%%%%%%%%%%%%%%%%%%%%%%%%%%%%%%%%%%%%%%%%%%%%%%%%%%
%%%%%%%%%%%%%%%%%%%%%%%%%%%%%%%%%%%%%%%%%%%%%%%%%%%%%%%%%%
\section{Logarithmic differential forms and toric varieties}
\label{sec:backEuler}
%%%%%%%%%%%%%%%%%%%%%%%%%%%%%%%%%%%%%%%%%%%%%%%%%%%%%%%%%%

This section connects the Euler characteristic to integrals of
characteristic classes of the sheaf of logarithmic differential forms.
We work on a possibly singular but normal, proper and irreducible
variety $\ol{\cH}$ of dimension~$d$, whose singularities are toric and
contained in some boundary divisor~$\partial\cH$. We
are interested in the Euler characteristic of the (Zariski) open subvariety
$ \cH=\ol{\cH}\setminus \partial\cH$ given by the complement of $\partial\cH$, in the situation where the inclusion
$ \cH \hookrightarrow \ol{\cH}$ is a toroidal embedding. In particular, in this case,
the boundary divisor $\partial\cH$ is
locally on open subsets $U_\alpha$ a torus-invariant divisor.
\par
In this situation we define locally $\Omega^1_{U_\alpha}(\log)$ to be the
sheaf generated by  $(\bC^*)^d$-invariant meromorphic differential forms. These
glue to sheaf $\Omega^1_{\ol{\cH}} (\log \partial\cH)$, that is called
\emph{logarithmic differential sheaf}. This terminology is
justified by the following idea from \cite[Section~4]{mumford}, the
details and definitions being given in \cite{KKMS}.
For any 'allowable' smooth modification $p: \ol{W} \to
\ol{\cH}$ that maps a normal crossing boundary divisor $\partial W
\subset \ol{W}$ onto $\partial\cH$ we have $p^*\Omega^1_{\ol{\cH}}
(\log \partial\cH) = \Omega^1_{\ol{W}}(\log \partial W)$ for the
usual definition of the logarithmic sheaf on $\ol{W}$. Moreover, such an
'allowable' smooth modification always exists. The previous situation can be generalized verbatim to the case where $\ol{\cH}$ is a Deligne-Mumford stack and this is the setup we are interested in.
\par
\begin{prop} \label{prop:chiviaTlog}
Let $\ol{\cH}$ be a proper irreducible Deligne-Mumford stack of dimension $d$ with toric singularities. Assume moreover that the coarse moduli space of  $\ol{\cH}$  is projective. Let $ \cH \hookrightarrow \ol{\cH}$ be a toroidal embedding and $\partial \cH=\ol{\cH}\setminus\cH$ . Then  the Euler characteristic
of~$ \cH$ can be computed as the integral
\be \label{eq:chiviaint}
\chi( \cH) \= (-1)^d \int_{\overline{\cH}} \,c_d(\Omega^1_{\ol{\cH}}
(\log \partial\cH )) 
\ee
over the top Chern class of the logarithmic differential sheaf.
\end{prop}
\par
\begin{proof} If $\ol{\cH}$ is a smooth Deligne-Mumford stack and $\partial\cH =\emptyset$, this is well known (see e.g. \cite[Corollary~4.16]{toen}). In the case where $\ol{\cH}$ is still smooth but $\partial\cH$ is not empty, a self-contained proof of the statement was given in \cite[Proposition~2.1]{CMZeuler} (the proof was given in the case where $\ol\cH$ is a smooth variety, but it works verbatim for the more general case of smooth DM stack).
\par
In general we use an allowable modification.
By definition this restricts to an isomorphism $W \to \cH$,
hence does not change the left hand side. The right hand side also
stays the same by push-pull and the pullback formula along an
allowable smooth modification.
\end{proof}
\par

%%%%%%%%%%%%%%%%%%%%%%%%%%%%%%%%%%%%%%%%%%%%%%%%%%%%%%%%%%%%

%%%%%%%%%%%%%%%%%%%%%%%%%%%%%%%%%%%%%%%%%%%%%%%%%%%%%%%%%%%%
%%%%%%%%%%%%%%%%%%%%%%%%%%%%%%%%%%%%%%%%%%%%%%%%%%%%%%%%%%%%%%%%%%
\section{The closure of linear submanifolds}
\label{sec:closurelinsection}
%%%%%%%%%%%%%%%%%%%%%%%%%%%%%%%%%%%%%%%%%%%%%%%%%%%%%%%%%%%%%%%%%%

The compactification of a linear submanifold we work with has (currently)
no intrinsic definition.
Rather we consider the normalization of the closure of a linear
submanifold inside the moduli space of multi-scale differentials $\LMS$.
We recall from \cite{BDG} the basic properties of such closures.
The goal of this section is to make precise and to explain the
following two slogans:
\begin{itemize}
\item Near boundary points without horizontal edges, the closure is
determined as for the ambient Abelian stratum by the combinatorics of
the level graph and it is smooth. The \emph{ghost automorphisms}, the
stack structure at the boundary that stems from twist groups, agrees with
the ghost automorphisms of the ambient stratum and the intersection pattern
is essentially determined by the \emph{profiles} of the level graph,
a subset of the profiles of the ambient stratum.
\item In the presence of horizontal edges there are toric singularities.
Working with the appropriate definition of the logarithmic cotangent sheaf
these singularities don't matter. This sheaf decomposes into summands
from horizontal nodes, from the level structure, and the deformation of
the differentials at the various levels, just as in the ambient stratum.
\end{itemize}

%%%%%%%%%%%%%%%%%%%%%%%%%%%%%%%%%%%%%%%%%%%%%%%%%%%%%%%%%%%%%%%%%%
\subsection{Linear submanifolds in generalized strata}
\label{sec:linsubmfd}
%%%%%%%%%%%%%%%%%%%%%%%%%%%%%%%%%%%%%%%%%%%%%%%%%%%%%%%%%%%%%%%%%%
\label{sec:generalizedstrata}

Let $\omoduli[g,n](\mu)$ denote the moduli space of Abelian differentials
of possibly meromorphic signature~$\mu$. Despite calling them 'moduli space'
or 'strata' we always think of them as quotient stacks or orbifolds and
intersection numbers etc.\ are always understood in that sense. These
strata come with a linear structure given by period coordinates (e.g.\
\cite{zorich06} for an introduction).
A \emph{linear submanifold} $\Omega \cH$ of $\omoduli[g,n](\mu)$ is an
algebraic stack with a map $\Omega \cH \to \omoduli[g,n](\mu)$ which is
the normalization of its image and whose image is locally given as a
finite union of linear subspaces in period coordinate charts.
See \cite[Example 4.1.10]{FiSurvey} for an example that illustrates why
we need to pass to the normalization for  $\Omega \cH$ to be a smooth
stack. 
In the context of holomorphic signatures and $\GL_2(\bR)$-orbit
closures, the linear manifolds obtained in this way can locally be defined by
equations with $\bR$-coefficients (\cite{EsMi}, \cite{EsMiMo}).
We refer to them as \emph{$\bR$-linear submanifolds}. In this context,
the algebraicity follows from being closed by the result of Filip (\cite{Filip}),
but in general algebraicity is an extra hypothesis.
\par
To set up for clutching morphisms and a recursive description of the boundary
of compactified linear submanifolds, we now define \emph{generalized strata}, compare
\cite[Section~4]{CMZeuler}. For a tuple $\bfg = (g_1,\ldots,g_k)$
of genera and a tuple $\bfn = (n_1,\ldots,n_k)$ together with a collection of
types $\bfmu = (\mu_1,\ldots, \mu_k)$ with $|\mu_i| = n_i$ we first define the
disconnected stratum
$%\be
\omoduli[\bfg,\bfn](\bfmu) \= \prod_{i=1}^k \omoduli[g_i,n_i](\mu_i)\,.
$ %\ee
Then, for a linear subspace~$\frakR$ inside the space of the residues at all poles
of $\bfmu$ we define the generalized stratum $\Romod[\bfmu][\bfg,\bfn][\frakR]$
to be the subvariety with residues lying in~$\frakR$. Generalized strata obviously
come with period coordinates and we thus define a \emph{generalized linear
submanifold} $\Omega \cH$ to be an  algebraic stack together with a map to
$\Romod[\bfmu][\bfg,\bfn][\frakR]$ whose image is locally linear in period
coordinates and where $\Omega \cH$ is the normalization of its image.
\par
Rescaling the differential gives an action of $\bC^*$ on strata and the
quotient are projectivized strata $\bP \omoduli[g,n](\mu)$. The image of a linear
submanifold in $\bP \omoduli[g,n](\mu)$ is called \emph{projectivized linear manifold} $\cH$, but
we usually omit the 'projectivized'.
\par
We refer with an index~$B$ to quantities of the ambient projectivized stratum,
such as its  dimension $d_B$  and the unprojectivized
dimension $N_B = d_B +1$. The same letters without additional index are used
for the linear submanifold, e.g.\ $N = d+1$, and we write $d_\cH$ and $N_\cH$ only if
ambiguities may arise.

%%%%%%%%%%%%%%%%%%%%%%%%%%%%%%%%%%%%%%%%%%%%%%%%%%%%%%%%%%%%%%%%%%
\subsection{Multi-scale differentials: boundary combinatorics}
\label{sec:boundarycomb}
%%%%%%%%%%%%%%%%%%%%%%%%%%%%%%%%%%%%%%%%%%%%%%%%%%%%%%%%%%%%%%%%%%

We will work inside the moduli stack of multi-scale differentials, that is
the compactification $\ol{B}:=\bP\LMS$ of a stratum $B:=\bP\omoduli[g,n](\mu)$
defined in \cite{LMS} and recall some of its properties, see also
\cite[Section~3]{CMZeuler}.
Everything carries over with obvious modifications to the compactification
$\proj\LMS[\bfmu][\bfg,\bfn][\frakR]$ of generalized strata, see
\cite[Proposition~4.1]{CMZeuler}.
\par
Each boundary stratum of $\bP\LMS$ has its associated level graph~$\Gamma$,
a stable graph of the underlying pointed stable curve together with
a weak total order on the vertices, usually given by a a level function
normalized to have top level zero, and an enhancement~$\kappa_e \geq 0$
associated to the edges. Edges are called \emph{horizontal}, if they
start and end at the same level, and \emph{vertical otherwise}. Moreover
$\kappa_e = 0$ if and only if the edge is horizontal. We denote the closure of
the boundary stratum of points with level graph~$\Gamma$ by~$D_\Gamma^B$ and, for any level graph $\Delta$ that is a degeneration of $\Gamma$, we let $D_{\Gamma,\Delta}^{B,\circ}\subset D_\Gamma^B$ be the open subset parametrising multi-scale differentials compatible with an undegeneration of $\Delta$. In particular the points of  $D_{\Gamma}^{B,\circ}:=D_{\Gamma,\Gamma}^{B,\circ}$  represent multi-scale differentials with level graph exactly $\Gamma$.
These $D_\Gamma^B$ are in general not connected, and might
be empty (e.g.\ for unsuitably large~$\kappa_e$). 
\par
We let $\LG_L(B)$ be the set of all enhanced $(L+1)$-level graphs without
horizontal edges. The structure of the normal crossing
boundary of $\bP\LMS$ is encoded by {\em undegenerations}. For any
subset $I = \{i_1,\dots,i_n\}  \subseteq \{1,\dots,L\}$ %NOT L+1 !!
there are undegeneration map
\bes
\delta_{i_1,\dots,i_n} \colon \LG_L(B)\to \LG_{n}(B)\,,
\ees
that preserves the level passage given as a horizontal line just above
level~$-i$ and contracts the remaining level passages. We define
$\delta_I^\complement = \delta_{I^\complement}$.
\par
The boundary strata $D_\Gamma^B$ for $\Gamma \in \LG_L(B)$ are commensurable
to a product of generalized strata $B_\Gamma^{[i]} =
\bP \LMS[\bfmu_i][\bfg_i,\bfn_i][\frakR_i]$ defined via the following diagram.
\be \label{dia:covboundary} \begin{tikzcd}
&c_\Gamma^{-1}(D_{\Gamma,\Delta}^{B,\circ})\arrow{dr}[swap]{q_\Delta}\arrow{r}{\subset}&D_{\Gamma}^{B,s}\arrow{ddll}[swap]{p_{\Gamma}} \arrow{ddrr}{c_{\Gamma}} && \\
&& B_{\Gamma,\Delta}^s \arrow{dl}{p_\Gamma^\Delta}\arrow{dr}[swap]{c_\Gamma^\Delta} &&\\
\prod_{i=-L}^{0} B_\Gamma^{[i]} =: {B}_{\Gamma} & B_{\Gamma,\Delta} \arrow{l}[swap]{\supset} && D_{\Gamma,\Delta}^{B,\circ} \arrow{r}{\subset} & D_\Gamma^B
\end{tikzcd}
\ee
Here $\bfg_i,\bfn_i$ and $\bfmu_i$ are the tuples of the genera, marked
points and signatures of the components at level~$i$ of the level graph
and $\frakR_i$ is the global residue condition induced by the levels above.
The covering space $D_\Gamma^{B,s}$ and the moduli stack $B_{\Gamma,\Delta}^s$ of
\emph{simple multi-scale differentials compatible with an undegeneration
of $\Delta$} were constructed in \cite[Section~4.2]{CMZeuler}.

%%%%%%%%%%%%%%%%%%%%%%%%%%%%%%%%%%%%%%%%%%%%%%%%%%%%%%%%%%%%%%%%%%
\subsection{Multi-scale differentials: Prong-matchings and stack structure}
\label{sec:prongmatchings}
%%%%%%%%%%%%%%%%%%%%%%%%%%%%%%%%%%%%%%%%%%%%%%%%%%%%%%%%%%%%%%%%%%

The notion of a multi-scale differential is based on the following construction.
Given a pointed stable curve~$(X,\bfz)$, a {\em twisted differential} is
a collection of differentials $\eta_v$ on each component~$X_v$ of~$X$, that
is {\em compatible with a level structure} on the dual graph~$\Gamma$ of~$X$,
i.e.\ vanishes as prescribed by~$\mu$ at the marked points~$z$, satisfies
the matching order condition at vertical nodes, the matching residue
condition at horizontal nodes and global residue condition of \cite{BCGGM1}.
A {\em multi-scale differential of type $\mu$} on a stable curve~$(X,\bfz)$
consists of an enhanced level structure~$(\Gamma,\ell,\{\kappa_e\})$
on the dual graph~$\Gamma$ of~$X$, a twisted differential~$\bfomega$ of
type~$\mu$ compatible with the enhanced level structure, and a prong-matching for
each node of~$X$ joining components of non-equal level. Here a \emph{prong-matching}
$\bfsigma$ is an identification of the horizontal (outgoing resp.\ incoming) real tangent
vectors at a zero resp.\ a pole corresponding to each vertical edge of~$\Gamma$.
Multi-scale differentials are equivalences classes of $(X,\bfz,\Gamma,\bfsigma)$
up to the action of the level rotation torus that rescales differentials
on lower levels and rotates prong-matchings at the same time.
\par
To an enhanced two-level graph we associate the quantity
\be \label{eq:defellGamm}
\ell_\Gamma \= \lcm(\kappa_e \colon e \in E(\Gamma)) \,.
\ee
which appears in several important places of the construction of $\bP\LMS$:
\begin{itemize}
\item[i)] It is the size of the orbit of prong-matchings when rotating the
  lower level differential. Closely related:
\item[ii)] The local equations of a node are $xy = t_1^{\ell_\Gamma/\kappa_e}$,
where~$t_1$ is a local parameter (a \emph{level parameter}) transverse to
the boundary. As a consequence a family of differential forms that tends
to a generator on top level scales with $t_1^{\ell_\Gamma}$ on the bottom level
of~$\Gamma$.
\end{itemize}
For graphs with $L$ level passages we define $\ell_i = \ell_{\Gamma,i}
= \ell_{\delta_i(\Gamma)}$ to be the lcm of the edges crossing
the $i$-th level passage and $\ell_{\Gamma}=\prod_{i=1}^{L}\ell_{\Gamma,i}$.
\par
There are two sources of automorphisms of multi-scale differentials:
on the one hand, there are automorphism of pointed stable curves that
respect the additional structure (differential, prong-matching). On the
other hand, there are \emph{ghost automorphisms}, whose group we denote by
$\Gh_\Gamma = \Tw[\Gamma]/\sTw[\Gamma]$,
that stem from the toric geometry of the compactification. We emphasize that
the twist group $\Tw[\Gamma]$ and the simple twist group~$\sTw[\Gamma]$, hence also the ghost group $\Gh_\Gamma$,
depend only on the data of the enhanced level graph and will be inherited
by linear submanifolds below. The local isotropy group of $\LMS$ sits
in an exact sequence
\bes
0 \to \Gh_\Gamma \to \mathrm{Iso}(X,\bfomega) \to \Aut(X,\bfomega) \to 0
\ees
and locally near $(X,\bfz,\Gamma,\bfsigma)$ the stack of multi-scale differentials
is the quotient stack $[U/\mathrm{Iso}(X,\bfomega)]$ for some open $U \subset
\bC^{N_B}$. The same holds for $\bP\LMS$ where the automorphism group is
potentially larger since $\bfomega$ is only required to be fixed projectively.

%%%%%%%%%%%%%%%%%%%%%%%%%%%%%%%%%%%%%%%%%%%%%%%%%%%%%%%%%%%%%%%%%%
\subsection{Decomposition of the logarithmic tangent bundle}
\label{sec:decompLTB}
%%%%%%%%%%%%%%%%%%%%%%%%%%%%%%%%%%%%%%%%%%%%%%%%%%%%%%%%%%%%%%%%%%

We now define a \emph{$\Gamma$-adapted basis}, combining \cite{BDG}
and \cite{CMZeuler} with the goal of giving a decomposition of
the logarithmic tangent bundle that is inherited by a linear submanifold,
if the $\Gamma$-adapted basis is suitably chosen.
\par
We work on a neighborhood~$U\subseteq B$ of a point $p=(X,[\omega],\bfz) \in D_\Gamma^{B,\circ}$, where~$\Gamma$ is an
arbitrary level graph with~$L$ levels below zero.
We let $\alpha_j^{[i]}$ for $i=0,\ldots,-L$ be the vanishing cycles
around the horizontal nodes at level~$i$. Let $\beta_j^{[i]}$ be
a dual horizontal-crossing cycle, i.e.\ $i$ is the top level (in the sense
of \cite{BDG}) of this cycle, $\langle \alpha_j^{[i]},\beta_j^{[i]}
\rangle = 1$ and $\beta_j^{[i]}$ does not have non-zero intersection with any other horizontal vanishing cycle at level~$i$. Let $h(i)$ be the number of those horizontal vanishing cycles at level~$i$.
\par
We complement the cycles $\beta_j^{[i]}$ by a collection
of relative cycles $\gamma_j^{[i]}$ such that for any fixed level~$i$ their top
level restrictions form a basis of the cohomology at level~$i$ relative
to the poles and zeros of~$\omega$  and holes at horizontal nodes quotiented
by the subspace of global residue conditions. In particular the span
of the $\gamma_j^{[i]}$ contains the $\alpha_j^{[i]}$, and moreover the
union
\bes
\bigcup_{j=-L}^0 \bigl\{\beta_1^{[j]}, \ldots, \beta_{h(j)}^{[j]},\gamma_1^{[j]},
\ldots, \gamma_{s(j)}^{[j]} \bigr\} \quad
\text{is a basis of} \quad H_1(X \setminus P, Z,\bC).
\ees
Next, we define the $\omega$-periods of these cycles and exponentiate to
kill the monodromy around the vanishing cycles. The functions
\bes
a_j^{[i]} \= \int_{\alpha_j^{[i]}} \omega\,, \quad b_j^{[i]} \= \int_{\beta_j^{[i]}}
\omega\,, \quad q_j^{[i]} \= \exp(2\pi I b_j^{[i]}/a_j^{[i]}) \,, \quad
c_j^{[i]} \= \int_{\gamma_j^{[i]}} \omega\,.
\ees
are however still not defined on~$U$ (only on sectors of the boundary
complement) due to monodromy around the vertical nodes.  
\par
Coordinates on~$U$ are given by \emph{perturbed period coordinates}
(\cite{LMS}), which are related to the periods above as follows.
For each level passage there is a \emph{level parameter~$t_i$}
that stem from the construction of the moduli space via plumbing.
On the bottom level passage~$L$ we may take $t_{L} = c_1^{[-L]}$ as a period.
For the higher level passage, the $t_i$ are closely related to the periods of
a cycle with top level~$-i$, but the latter are in general not monodromy
invariant. It will be convenient to write 
\begin{equation}\label{eq:prodtnotation}
\prodt[i] \= \prod_{j=1}^i t_j^{\ell_j}, \quad i\in \bN.
\end{equation}
There are perturbed periods $\wt{c}_j^{[-i]}$ obtained by integrating
$\omega/\prodt[i]$ against a cycle with top level~$-i$ over the part of
level~$-i$ to points nearby the nodes, cutting off the lower level part.
By construction, on each sector of the boundary complement we have
\be \label{eq:ccdiff}
\wt{c}_j^{[-i]} - c_j^{[-i]}/\prodt[i] \= \sum_{s > i}
\frac{\prodt[s]}{\prodt[i]} E_{j,i}^{[-s]} 
\ee
for some linear ('error') forms $E_{j,i}^{[-s]}$ depending on the variables
$c_j^{[-s]}$ on the lower level~$-s$.
%t_{i+1}^{\ell_{i+1}}, \ldots, t_L^{\ell_L})$.
Similarly, we can exponentiate the ratio over $a_j^{[-i]}$ of the
similarly perturbed $\wt{b}_j^{[-i]}$ and obtain  perturbed exponentiated
periods $\wt{q}_j^{[-i]}$, such that on each sector 
\be \label{eq:logqdiff}
 \log\wt{q}_j^{[-i]} - \log {q_j^{[-i]}} \= \sum_{s > i}
\frac{\prodt[s]}{\prodt[i]} E_{j,i}'^{[-s]} 
\ee
for some linear forms $E_{j,i}'^{[-s]}$.
In these coordinates the boundary is given by $\wt{q}_j^{[-i]} = 0$
and $t_i=0$. If we let
\bas
\Omega_{i,B}^{\hor}(\log) &\= \langle d\wt{q}_1^{[i]}/\wt{q}_1^{[i]}, \ldots,
d\wt{q}_{h(i)}^{[i]}/\wt{q}_{h(i)}^{[i]} \rangle, \quad
\Omega_{i,B}^{\lev}(\log) &\= \langle dt_{-i}/t_{-i} \rangle \\
\Omega_{i,B}^{\rel} &\= \langle d\wt{c}_2^{[i]}/\wt{c}_2^{[i]}, \ldots,
d\wt{c}_{N(i)-h(i)}^{[i]}/\wt{c}_{N(i)-h(i)}^{[i]} \rangle, 
\eas
with $\Omega_{0,B}^{\lev}(\log)=0$ by convention, we thus obtain a decomposition 
\be \label{eq:OmDecompB}
\Omega^1_{\ol{B}} (\log \partial B )|_U
\= \bigoplus_{i=-L}^0  \Bigl(\Omega_{i,B}^{\hor}(\log) \oplus
\Omega_{i,B}^{\lev}(\log) \oplus \Omega_{i,B}^{\rel} \Bigr)\,.
\ee

%%%%%%%%%%%%%%%%%%%%%%%%%%%%%%%%%%%%%%%%%%%%%%%%%%%%%%%%%%%%%%%%%%
\subsection{The closure of linear submanifolds}
\label{sec:closurelin}
%%%%%%%%%%%%%%%%%%%%%%%%%%%%%%%%%%%%%%%%%%%%%%%%%%%%%%%%%%%%%%%%%%

For a linear submanifold $\cH$ we denote by $\ol{\cH}$ the normalization of
the closure of the image of $\cH$ as a substack of $\LMS$. We denote by $D_\Gamma = D^\cH_\Gamma$
the preimage of the boundary divisor~$D_\Gamma^B$ in $\ol{\cH}$. 
Again, a~$\circ$ denotes the complement of more degenerate
boundary strata, i.e., $D_\Gamma^{\cH,\circ}$ is the preimage of $D_\Gamma^{B,\circ}$ in $\ol{\cH}$.
\par
We will now give several propositions that explain that~$\ol{\cH}$ is
a compactification of~$\cH$ almost as nice as the compactification $\bP\LMS$
of strata. The first statement explains the 'almost'.
\par
\begin{prop} \label{prop:bdtoricsing}
Let $\Gamma$ be a level graph with only horizontal nodes, i.e., with one level only. Then \referee{$D_\Gamma^{\cH,\circ}$ has at worst toric singularities}. 
\end{prop}
\par
More precisely, the linear submanifold is cut out by linear and
binomial equations, see~\eqref{eq:toriceqns} below.
\par
Second, the intersection with non-horizontal  boundary components is transversal
in the strong sense that each level actually causes dimension drop.
\par
\begin{prop} \label{prop:transverse}
Let $\Gamma \in \LG_L(B)$ be a level graph without horizontal
nodes. Each point in $D_\Gamma^{\cH,\circ}$ is smooth and  $D_\Gamma^{\cH,\circ}$  is
a normal crossing divisor given as the intersection of $L$ different divisors
$D^\cH_{\delta_i(\Gamma)}$.
\par
In particular,  $D^\cH_\Gamma$ has codimension~$L$ in $\ol{\cH}$.
\end{prop}
\par
Fix now an enumeration of two level graphs, i.e. a bijection between $\LG_1(B)$ and $\{1,\dots, |\LG_1(B)|\}$, and  define $D_{i_1,\ldots,i_L}\subseteq \bigcap_{j=1}^L D_{\Gamma_{i_j}}  $ to be the subset  of all $D_{\Lambda}$, with  $\Lambda\in \LG_L(B)$,  such that $\delta_j(\Lambda) = \Gamma_{i_j}$ for all $j=1,\ldots,L$. The previous proposition allows to  show, via the same argument as the
proof of \cite[Proposition~5.1]{CMZeuler}, the key result in order to
argue inductively.
\par
\begin{cor} 	\label{cor:ordering}
If $\cap_{j=1}^L D^\cH_{\Gamma_{i_j}}$ is not empty,  there is a unique
ordering $\sigma\in\Sym_L$ on the set $I=\{i_1,\dots,i_L\}$ of indices
such that
\[D_{\sigma(I)}\=\bigcap_{j=1}^L D^\cH_{\Gamma_{i_j}}\,.\]
Moreover if $i_k=i_{k'}$ for a pair of indices $k\not =k'$, then $D_{i_1,\dots,i_L}=\emptyset$.
\end{cor}
\par
The next statement is crucial to inductively apply the formulas in
this paper.   We now need to refine our analysis by looking at irreducible components of $D^\cH_\Gamma$. Recall that $\LG^+_1(\cH) $ is the set of pairs $\Gamma^+=(\Gamma,i)$ where $\Gamma\in \LG_1(\cH)$ and $i\in \{1,\dots,n_\Gamma(\ol \cH)\}$ is the index set of irreducible components of  $D_\Gamma^{\cH,\circ}$. We denote by  $D^\cH_{\Gamma^+}$ the irreducible component of $D^\cH_{\Gamma}$ corresponding to $\Gamma^+$. Recall that $p_\Gamma$ and $c_\Gamma$ are the projection and clutching morphisms of the diagram~\eqref{dia:covboundary}.
\par
\begin{prop} \label{prop:linatboundary}
There are generalized  linear submanifolds $\Omega \cH_{\Gamma^+}^{[i]} \to
\Romod[\bfmu_i][\bfg_i,\bfn_i][\frakR_i]$ of dimension $d_i$ with projectivization
$\cH_{\Gamma^+}^{[i],\circ}$, such that $$\sum_{i=-L}^0 d_i  \=  d_\cH -L$$
and such that the normalizations $\cH_{\Gamma^+}^{[i]} \to B_\Gamma^{[i]}$ of closures
of $\cH_{\Gamma^+}^{[i],\circ}$ together give a product decomposition
$\cH_{\Gamma^+} = \prod_{i=-L}^0 \cH_{\Gamma^+}^{[i]}$ of the normalization of
the~$p_\Gamma$-image of the~$c_\Gamma$-preimage of~$\mathrm{Im}(D^\cH_{\Gamma^+})
\subset \bP\LMS$. 
\end{prop}
We will call  $\cH_{\Gamma^+}^{[i]} \to B_{\Gamma^+}^{[i]}$ the \emph{$i$-th level linear
manifold}. Our ultimate goal here is to show the following decomposition.
The terminology is explained along with the definition of coordinates.
\par
\begin{prop} \label{prop:Omegadecomp}
  Let~$\Gamma$ be an arbitrary level graph with~$L$
levels below zero. In a small neighborhood~$U$ of a point in $
D_\Gamma^\cH$
there is a direct sum decomposition 
\be \label{eq:OmDecompH}
\Omega^1_{\ol{\cH}} (\log \partial\cH )|_U
\= \bigoplus_{i=-L}^0 \Bigl(\Omega_i^{\hor}(\log) \oplus
\Omega_i^{\lev}(\log) \oplus \Omega_i^{\rel} \Bigr)
\ee
for certain subsheaves such that the natural restriction map induces surjections
\bes
\Omega_{i,B}^{\hor}(\log)|_{\ol{\cH}} \twoheadrightarrow \Omega_i^{\hor}(\log), \quad
\Omega_{i,B}^{\lev}(\log)|_{\ol{\cH}}\, \simeq\, % \twoheadrightarrow
\Omega_i^{\lev}(\log) \quad
\text{and}\quad \Omega_{i,B}^{\rel}|_{\ol{\cH}} \twoheadrightarrow \Omega_i^{\rel}\,.
\ees
Moreover the statements in  items~i) and~ii) of \autoref{sec:prongmatchings}
hold verbatim for the linear submanifold with the same~$\ell_\Gamma$.
\end{prop}
\par
As a consequence we may use the symbols $\ell_\Gamma$ and $\ell_{\Gamma_i}$ ambiguously
for strata and their linear submanifolds.
\par
We summarize the relevant parts of \cite{BDG}. Equations of~$\cH$ are
interpreted as homology classes and we say that a \emph{horizontal node is
crossed by an equation}, if the corresponding vanishing cycle has non-trivial
intersection with the equation. The horizontal nodes are partitioned
into \emph{$\cH$-cross-equivalence classes} by simultaneous appearance
in equations for~$\cH$. A main observation is that $\omega$-periods
of the vanishing cycles in an $\cH$-cross-equivalence class are
proportional. Similarly, for each equation and for any level passage
the intersection numbers of the equation with the nodes crossing that
level add up to zero when weighted appropriately with the residue
times  $\ell_\Gamma/\kappa_e$ (\cite[Proposition~3.11]{BDG}).
\par
Next, in \cite{BDG} they sort the equations by level and then write them in
reduced row echelon from. One may order the periods so that the distinguished
$c_1^{[i]}$ (whose period is close to the level parameter~$t_{-i}$) is among
the pivots of the echelon form for each~$i$. The second main observation
is that each defining equation of~$\cH$ can be split into a sum of
defining equations, denoted by $F_k^{[i]}$, with the following properties.
The upper index~$i$ indicate the highest level, whose periods are involved
in the equation. Moreover, either~$F_k^{[i]}$ has non-trivial intersection
with some (vanishing cycles of a) horizontal node at level~$i$ and
then no intersection with a horizontal node at lower level, or
else no intersection with a horizontal node at all.
\par
As a result $\cH$ is cut out by two sets of equations, see
\cite[Equations~(4.2), (4.3), (4.4)]{BDG}. First, there are the equations
$G_k^{[i]}$ that are %close to linear function $F_k^{[i]}$
$\prodt[-i]$-rescalings of linear functions

\be \label{eq:lineareqns}
G_k^{[i]} \=  L_k^{[i]}\bigl(\wt{c}_{2-\delta_{i,0}^{[i]}}, \ldots, \wt{c}_{N(i)-h(i)}^{[i]}\bigr)
\ee
in the periods at level~$i$. (To get this form from the version in
\cite{BDG} absorb the terms from lower level periods into the
function $c_j^{[i]}$ where $j=j(k,i)$ is the pivot of the equation $F_k^{[i]}$.
This does not effect the truth of~\eqref{eq:ccdiff}).
\par
Second, there are multiplicative monomial equations among the exponentiated
periods, that can be written as bi-monomial equations with positive exponents
\be \label{eq:toriceqns}
H^{[i]}_k \= (\wt{\bfq}^{[i]})^{J_{1,k}} - (\wt{\bfq}^{[i]})^{J_{2,k}}
\ee
where $\wt{\bfq}^{[i]}$ is the tuple of the variables $\wt{q}_j^{[i]}$ and
$J_{1,k}, J_{2,k}$ are tuples of non-negative integers. (In the multiplicative
part \cite{BDG} already incorporated the lower level blurring into
the pivot variable.)
\par
\begin{proof}[Proof of \autoref{prop:bdtoricsing}] This  follows
directly from the form of the binomial equations~\eqref{eq:toriceqns},
see \cite[Theorem~1.6]{BDG}.
\end{proof}
\par
\begin{proof}[Proof of \autoref{prop:transverse}]
Smoothness and normal crossing is contained in \cite[Corollary~1.8]{BDG}.
The transversality claimed there contains the dimension drop claimed in
the proposition. The more precise statement in \cite[Theorem~1.5]{BDG}
says that after each intersection of $\ol{\cH}$ with a vertical boundary
divisor the result is empty or contained in the open boundary divisor
$D_\Gamma^{B,\circ}$.
\end{proof}
\par
\begin{proof}[Proof of \autoref{prop:linatboundary}] This is the
main result of \cite{Benirschke} or the restatement
in \cite[Proposition~3.3]{BDG} and this together with the 
Proposition~\ref{prop:transverse} implies the dimension statement.
\end{proof}
\par
\begin{proof}[Proof of \autoref{prop:Omegadecomp}]
Immediate from~\eqref{eq:lineareqns} and~\eqref{eq:toriceqns}, which
are equations among the respective set of generators of the
decomposition in~\eqref{eq:OmDecompB}.  The additional claim item~ii)
follows from the isomorphism of level parameters and transversality. Item~i)
is a consequence of this.
\end{proof}

%%%%%%%%%%%%%%%%%%%%%%%%%%
\subsection{Push-pull comparison for linear submanifolds}
%%%%%%%%%%%%%%%%%%%%%%%%%%

For recursive computations, we will transfer classes from $\cH_{\Gamma^+}^{[i]}$,
which were defined via \autoref{prop:linatboundary}, to $D_{\Gamma^+}^\cH$
essentially via $p_{\Gamma^+}$-pullback and $c_{\Gamma^+}$-pushforward. More precisely,
taking the normalizations into account, we have to use the maps
$c_{{\Gamma^+},\cH}$ and $p_{{\Gamma^+},\cH}$ defined on the normalization $\cH_{\Gamma^+}^s$ of the
$c_{\Gamma^+}$-preimage of the image of $D_{\Gamma^+}^\cH$ in $D_{\Gamma^+}^B$. To compute degrees we use the analog
of the inner triangle in~\eqref{dia:covboundary} and give a concrete
description of~$\cH_{\Gamma^+}^s$.
\par
Recall from the introduction that $K^\cH_{\Gamma^+}$ is  the product of the number of prong-matchings
on each edge of $\Gamma$ that are actually contained in  $D_{\Gamma^+}^\cH$.
\be \label{dia:covboundary2} \begin{tikzcd}
(\Omega\cH_{\Gamma^+}^\circ)^{\mathrm{pm}}\arrow{dd} \arrow{rrr} 
&&&\cH_{{\Gamma^+}}^{s,\circ}\arrow{ddll}[swap]{p_{{\Gamma^+},\cH}} \arrow{ddrr}{c_{{\Gamma^+},\cH}} \ar[d]
&& \\
&&&B_{{\Gamma},{\Gamma}}^s \arrow{dl}{p_{\Gamma}^{\Gamma}}\arrow{dr}[swap]{c_{\Gamma}^{\Gamma}} &&\\
\Omega\cH^\circ_{\Gamma^+} \arrow{r}  &{\cH}^\circ_{{\Gamma^+}} \arrow{r}
& B_{\Gamma,\Gamma}  && D_{\Gamma}^{B,\circ}  & D^{\cH,\circ}_{\Gamma^+} \arrow{l}
\end{tikzcd}
\ee
\par
Consider $\Omega\cH^\circ_{\Gamma^+} := \prod \Omega \cH_{\Gamma^+}^{[i]}$ as a moduli space
of differentials subject to some (linear) conditions imposed on its periods.
Consider moreover the moduli space $(\Omega\cH_{\Gamma^+}^\circ)^{\mathrm{pm}}:= (\prod \Omega
\cH_{\Gamma^+}^{[i]})^{\mathrm{pm}}$
where we add the additional datum of one of the $K_{\Gamma^+}^\cH$ prong-matchings
reachable from the interior.  The torus $(\bC^*)^{L+1}$ acts on $\Omega\cH^\circ_{\Gamma^+}$ with quotient~$\cH^\circ_{\Gamma^+}
= \prod \cH_{\Gamma^+}^{[i],\circ} $. On the
other hand, if we take the quotient of $(\Omega\cH_{\Gamma^+}^\circ)^{\mathrm{pm}}$ by
$(\bC^*)^{L+1} = (\bC^*) \times (\bC^L/\text{Tw}_\Gamma^s)$ we obtain a space~$\cH_{\Gamma^+}^{s,\circ}$
which is naturally the normalization of a subspace of~$U_{\Gamma^+}^s$, since
it covers~$D^{\cH,\circ}_{\Gamma^+}$ with marked (legs and) edges and whose generic
isotropy group does not stem from $\Gh_{\Gamma}$ (it might be non-trivial,
e.g.\ if a level of~${\Gamma^+}$ consists of a hyperelliptic stratum), while
the generic isotropy group of~$D^{\cH,\circ}_{\Gamma^+}$ is an extension of~$\Gh_{\Gamma}$ by
possibly some group of graph automorphisms and possibly isotropy groups
of the level strata. 
\par
\begin{lemma} \label{lem:ratio_degrees}
The ratio of the degrees the maps in \ref{dia:covboundary2} on $\cH_{\Gamma^+}^s$ is
\bes
\frac{\deg(p_{{\Gamma^+},\cH})}{\deg(c_{{\Gamma^+},\cH})}
\= \frac{K^\cH_{\Gamma^+} }{|\Aut_{\cH}({\Gamma^+})|\ell_{\Gamma^+}}, 
\ees
where $\Aut_{\cH}({\Gamma^+})$ is the subgroup of~$\Aut(\Gamma)$ whose induced action
on a neighborhood of $D_{\Gamma^+}^\cH$ preserves $\ol \cH$ and $\ell_{\Gamma^+}=\ell_\Gamma$.
\end{lemma}
\par
\begin{proof} We claim that $\deg(p_{{\Gamma^+},\cH}) = K_{\Gamma^+}^\cH / [R_\Gamma :  \sTw[\Gamma]]$ where $R_\Gamma \cong \bZ^L \subset \bC^L$ is the level
rotation group. In fact this follows since in the left quadrilateral in~\eqref{dia:covboundary2} the left vertical arrow has degree $K^\cH_{\Gamma^+}$ while the
bottom arrow is the quotient by $(\bC^*) \times ((\bC^*)^L / R_\Gamma)$ and the
top arrow is quotient by $(\bC^*) \times (\bC^*)^L / \sTw[\Gamma]$.
\par 
On the other side under the map~$c_{\Gamma}^{\Gamma}$ of the ambient stratum two points have the same image only if they differ by an automorphism of~$\Gamma$. However only the  subgroup $\Aut_{\cH}({\Gamma^+}) \subset \Aut(\Gamma)$ acts on $c_{{\Gamma^+},\cH}(\cH_{{\Gamma^+}}^{s,\circ})$ and its normalization and contributes to the local isotropy group of the normalization. Thus only this subgroup contributes to the degree of $c_{{\Gamma^+},\cH}$. The claimed equality now follows because $[R_\Gamma : \sTw[\Gamma]] = \ell_\Gamma$.
\end{proof}
\par
Consider $\Gamma^+=(\Gamma,j)\in \LG^+_L(\cH)$ and  $\Delta^+ \in \LG^+_1(\cH_{\Gamma^+}^{[i]})$  defining an irreducible component of a divisor
in $\cH_{\Gamma^+}^{[i]}$. We aim to compute its pullback  to $D_{\Gamma^+}^s$ and the push
forward to~$D_{\Gamma^+}$ and to~$\ol{\cH}$. For this purpose we need extend
the commensurability diagram~\eqref{dia:covboundary2} to include degenerations
of the boundary strata. This works by copying verbatim the construction
that lead in~\cite{CMZeuler} to the commensurability diagram~\eqref{dia:covboundary}.
We will indicate with subscripts~$\cH$ to the morphisms that we work in this
adapted setting. Recall from this construction that in $B_{\Gamma,\Gamma}^{s}$ (and hence in $\cH_{\Gamma^+}^{s,\circ}$) the edges of~${\Gamma}$ have been labeled once
and for all (we write $\Gamma^\dagger$ for this labeled graph) and that the
level strata~$\cH_{\Gamma^+}^{[i]}$ inherit these labels.
Consequently, there is a unique irreducible component $ D_{\widetilde{\Delta}_+^\dagger}$  associated to a level graph $\widetilde{\Delta}^\dagger$ which is a
degeneration of~$\Gamma^\dagger$ and such that the products of the levels levels~$i$ and~$i-1$ of~$D_{{\widetilde{\Delta}_+}^\dagger}$ equals~$\cH_{\Delta^+}^\circ$.
The resulting refined unlabeled graph will simply be denoted by~$\widetilde{\Delta}^+$. For a fixed labeled graph~$\Gamma^\dagger$ we denote by $J(\Gamma^\dagger,\widetilde{\Delta}^+)$ the set of $\Delta^+ \in \LG^+_1(\cH_{\Gamma^+}^{[i]})$ such that $\widetilde{\Delta}^+$ is the result of that procedure. Obviously the graphs in $J(\Gamma^\dagger,\widetilde{\Delta}^+)$ differ only by the labeling of their half-edges and the following lemma computes its cardinality.
\par
\begin{lemma} \label{le:autcancel}
The cardinality of $J(\Gamma^\dagger,\widetilde{\Delta}^+)$
is determined by
\bes
|J(\Gamma^\dagger,\widetilde{\Delta}^+)|\cdot |\Aut_\cH(\widetilde{\Delta}^+)| \=
|\Aut_{\cH^{[i]}_{\Gamma^+}}(\Delta^+)|\cdot |\Aut_\cH({\Gamma^+})|\,.
\ees\
\end{lemma}
\begin{proof}
The proof is analogous to the one of \cite[Lemma~4.6]{CMZeuler}, where one considers the kernel and cokernel of the map $\varphi:\Aut_\cH(\widetilde{\Delta}^+)\to\Aut_\cH(\Gamma^+)$ given by undegeneration. 
\end{proof}
\par
We now determine the multiplicities of the push-pull procedure. 
Recall from \autoref{sec:prongmatchings} the definition
of $\ell_{\Gamma,j} = \ell_{\delta_j(\Gamma)}$ for $j\in \bZ_{\geq 1}$.
\par
\begin{prop}
	\label{prop:pushpullcomparison}
For a fixed $\Delta^+ \in \LG_1^+(\cH_\Gamma^{[i]})$, the divisor classes of
$D^\cH_{\widetilde{\Delta}^+}$ and the clutching of $D^\cH_{\Delta^+}$ are related by
\be \label{eq:Dcomparison}
\frac{|\Aut_\cH(\widetilde{\Delta}^+)|}
{|\Aut_{\cH^{[i]}_{\Gamma^+}}(\Delta^+)| |\Aut_\cH(\Gamma^+)|}
\cdot
c_{\Gamma^+,\cH}^* [D^\cH_{\widetilde{\Delta}^+}]
\= \frac{\ell_{{\Delta}}}{\ell_{\widetilde{\Delta},-i+1}}
 \cdot p_{\Gamma^+,\cH}^{[i],*} [D^\cH_{\Delta^+}]\,.
\ee
in $\CH^1(\cH_{\Gamma^+}^s)$ and consequently by 
\be \label{eq:Dcomparison2}
\frac{|\Aut_\cH(\widetilde{\Delta}^+)|}{|\Aut_\cH(\Gamma^+)|} \cdot 
\ell_{\widetilde{\Delta},-i+1}   \cdot [D^\cH_{\widetilde{\Delta}^+}]
\=  \frac{|\Aut_{\cH^{[i]}_{\Gamma^+}}(\Delta^+)|}{\deg(c_{\Gamma^+,\cH})} \cdot  \ell_\Delta \cdot
\,c_{\Gamma^+,\cH,*} \bigl(p_{\Gamma,\cH}^{[i],*} [D^\cH_{\Delta^+}]\bigr)
\ee
in $\CH^1(D_{\Gamma^+})$.
\end{prop}
\par
Here~\eqref{eq:Dcomparison} is used later for the proofs of the main theorems
while~\eqref{eq:Dcomparison2} is implemented in \texttt{diffstrata} for the special case of $k$-differentials to compute the pull-back of tautological classes from $D_{\Delta^+}^\cH$ to $D_{\wt\Delta^+}^\cH$, see also \autoref{sec:kdiff}.
\par
\begin{proof}
The proof is similar to the one of \cite[Proposition~4.7]{CMZeuler} and works by comparing the ramification orders of the maps $c_{\Gamma^+,\cH}^{\widetilde{\Delta}^+}$ and $p_{\Gamma^+,\cH}^{\widetilde{\Delta}^+}$.
The main difference to the original proof is only that the automorphism factors appearing in the clutching morphisms are the ones fixing irreducible components of $\cH$.
\end{proof}
\par
\medskip
The final part of this section is to compare various natural vector bundles under
pullback along the maps $c_{\Gamma^+,\cH}$ and $p_{\Gamma^+,\cH}$. The first bundle we consider 
is $\cE_{\Gamma^+}^\top$, a vector bundle of rank $N_{\Gamma^+}^\top-1$ on~$D^\cH_{\Gamma^+}$  that
should be thought of
as the top level version of the logarithmic cotangent bundle. Formally,
let $U\subset D^\cH_{\Gamma^+}$ be an open set centered at a degeneration of the top level of ${\Gamma^+}$
into $k$ level passages. Then  we define 
\begin{equation}\label{eq:cEtop}
	{\cE_{\Gamma^+}^\top}_{|U} \ =\ \bigoplus_{i=-k}^0  	{\Omega_i^{\lev}(\log)}_{|U}	\oplus \Omega_i^{\hor}(\log)_{|U}\oplus {\Omega_i^{\rel}}_{|U} 	\,.
\end{equation}
Let moreover $\xi^{[i]}_{{\Gamma^+},\cH}$ be the first Chern class of the line bundle on~$D^\cH_{\Gamma^+}$  generated by the multi-scale component at level $i$ and and $\cL_{{\Gamma^+}}^{[i]}$ be the line bundle whose divisor is given by the degenerations of the $i$-th level of ${\Gamma^+}$, as  defined more formally  in~\eqref{eq:defLithlev} below.
\par
We have the following compatibilities.
\begin{lemma}\label{lem:compatibilities}
	The first Chern classes of the tautological bundles on the levels
	of a boundary divisor are related by
	\be
	c_{{\Gamma^+},\cH}^* \, \xi^{[i]}_{{\Gamma^+},\cH}  \=  p_{{\Gamma^+},\cH}^{[i],*} \xi_{\cH_{\Gamma^+}^{[i]}}
	\qquad \text{in} \quad \CH^1(\cH_{\Gamma^+}^s)\,.
	\ee
It is also true that
	\be
	p_{{\Gamma^+},\cH}^{[i]*} \cL_{\cH_{\Gamma^+}^{[i]}} \= c_{{\Gamma^+},\cH}^* \cL_{{\Gamma^+}}^{[i]}
	\quad \text{where} \quad
	\cL_{\cH_{\Gamma^+}^{[i]}} \= \cO_{\cH_{\Gamma^+}^{[i]}}
	\Bigl( \sum_{\Delta \in \LG_1(\cH_{\Gamma^+}^{[i]})} {\ell_{\Delta}} D_{\Delta}\Bigr).
\ee
Similarly for the logarithmic cotangent bundles we have
\be
p_{{\Gamma^+},\cH}^{[0],*}\, \Omega^1_{\cH_{{\Gamma^+}}^{[0]}}(\log D_{\cH_{\Gamma^+}^{[0]}})
\=  c_{{\Gamma^+},\cH}^* \, \cE_{{\Gamma^+},\cH}^\top \,.\ee
\end{lemma}
\par
\begin{proof}
The first claim is just the global compatibility of the definitions of the bundles
$\cO(-1)$ on various spaces, compare \cite[Proposition~4.9]{CMZeuler}.
\par
The second claim is a formal consequence of \autoref{le:autcancel} and \autoref{prop:pushpullcomparison}, just
as in \cite[Lemma~7.4]{CMZeuler}.
\par
The last claim follows as in \cite[Lemma~9.6]{CMZeuler}
by considering local generators, which are given in~\eqref{eq:cEtop} and
have for linear submanifolds the same shape as for strata.
\end{proof}
\par
In the final formulas we will use these compatibilities together with the
following restatement of \autoref{lem:ratio_degrees}.
\par
\begin{lemma}\label{lem:evaluation}
Suppose that $\alpha_{\Gamma^+} \in \CH_0(D^\cH_{\Gamma^+})$ is a top degree class
and that $c_{{\Gamma^+},\cH}^* \alpha_{\Gamma^+} = \prod_{i=0}^{-L({\Gamma^+})} p_{{\Gamma^+},\cH}^{[i],*} \alpha_i$
for some $\alpha_i$. Then
\bes
\int_{D^\cH_{\Gamma^+}}\alpha_{\Gamma^+} \= \frac{K^\cH_{\Gamma^+}  }{|\Aut_\cH({\Gamma^+})|\ell_{\Gamma^+}}
\prod_{i=0}^{-L({\Gamma^+})} \int_{{\cH}_{\Gamma^+}^{[i]}} \alpha_i\,.
\ees
\end{lemma}

%%%%%%%%%%%%%%%%%%%%%%%%%%%%%%%%%%%%%%%%%%%%%%%%%%%%%%%%%%%%

%%%%%%%%%%%%%%%%%%%%%%%%%%%%%%%%%%%%%%%%%%%%%%%%%%%%%%%%%%%%
%%%%%%%%%%%%%%%%%%%%%%%%%%%%%%%%%%%%%%%%%%%%%%%%%%%%%%%%%%
\section{Evaluation of tautological classes} \label{sec:nb}
%%%%%%%%%%%%%%%%%%%%%%%%%%%%%%%%%%%%%%%%%%%%%%%%%%%%%%%%%%

This section serves two purposes. First, we briefly sketch a definition of the
tautological ring of linear submanifolds and how the results of the previous
section can be used to evaluate expressions in the tautological ring, provided
the classes of the linear manifold are known. Second, we provide formulas to
compute the first Chern class of the normal bundle~$\cN^\cH_\Gamma = \cN_{D^\cH_\Gamma}$
to a boundary divisor~$D^\cH_\Gamma$ of a projectivized linear submanifold $\ol{\cH}$.
This is needed both for the evaluation algorithm and as an ingredient to prove
our main theorems.

%%%%%%%%%%%%%%%%%%%%
\subsection{Vertical tautological ring}
%%%%%%%%%%%%%%%%%%%%
We denote by $\psi_i \in \CH^1(\ol \cH)$ the pull-backs of the classes $\psi_i
\in \CH^1(\barmoduli[g,n])$ to a linear submanifold $\ol\cH$. The \emph{clutching
maps} are defined as $\cl_{{\Gamma^+},\cH} = \i_{{\Gamma^+},\cH} \circ c_{{\Gamma^+},\cH}$, where
$\i_{{\Gamma^+},\cH} : D^\cH_{\Gamma^+} \to \ol{\cH}$ is the inclusion map of an irreducible components of the boundary divisor. We define the refined \emph{(vertical) tautological ring $R_v^\bullet(\ol\cH)$
of~$\ol\cH$} to be the ring with additive generators
\be \label{eq:addgenR}
\cl_{{\Gamma^+},\cH,*} \Bigl(\,\prod_{i=0}^{-L} p_{{\Gamma^+},\cH}^{[i],*} \alpha_i\,\Bigr) 
\ee
where ${\Gamma^+}$ runs over all irreducible components of boundary components associated to level graphs without horizontal edges for all boundary
strata of~$\cH$, including the trivial graph, and where~$\alpha_i$ is a
monomial in the $\psi$-classes supported on level~$i$
of the graph~$\Gamma^+$. That this is indeed a ring follows from the
excess intersection formula \cite[Proposition~8.1]{CMZeuler} that works exactly
the same for linear submanifolds, and the normal bundle formula
\autoref{prop:generalnormalbundle}  which allows together with
\autoref{prop:Adrienrel} to rewrite products in terms of our standard
generators. We do not claim that pushforward $R_v^\bullet(\ol\cH) \to
\CH^\bullet(\barmoduli[g,n])$ maps to the tautological ring $R^\bullet(\barmoduli[g,n])$,
since the fundamental classes of linear submanifolds, e.g.\ loci of double
covers of elliptic curves, may be non-tautological in~$\barmoduli[g,n]$
(see e.g.\ \cite{GP03}).
\par
If $\alpha\in \CH_0(\ol \cH)$  is a top-degree class which is also an additive generator of the tautological ring, i.e., it is has an expression as in \eqref{eq:addgenR}, we   can apply \autoref{lem:evaluation}  to obtain
\be \label{eq:levelproduct}
\int_{\ol\cH} \alpha \=\int_{\ol\cH} \cl_{{\Gamma^+},\cH,*} \Bigl(\,\prod_{i=0}^{-L} p_{{\Gamma^+},\cH}^{[i],*} \alpha_i\,\Bigr)  =\frac{K_{\Gamma^+}^\cH}{|\Aut_\cH({\Gamma^+})| \ell_{\Gamma^+}}
\prod_{i=0}^{-L(\Gamma^+)} \int_{\cH_{\Gamma^+}^{[i]}}\alpha_i. 
\ee
To evaluate this expression, one needs to determine the fundamental classes
of the level linear submanifolds $\cH_{\Gamma^+}^{[i]}$ in their corresponding generalized
strata, which is in general a non-trivial task.
\par
In the case where $\alpha \in \CH_0(\ol \cH)$ is a special top-degree class supported on a full boundary stratum $D_\Gamma$, and not only on one of its components $D_{\Gamma^+}$, there is a possibly different way to evaluate it. Indeed, note first that 
\[\alpha=\cl_{{\Gamma},\cH,*} \left(\,\prod_{i=0}^{-L} p_{{\Gamma},\cH}^{[i],*}\left(\prod_{j=1}^{l(i)} \psi_j^{p_j}\,\right)\right)=\psi_{1}^{p_1} \cdots
\psi_{n}^{p_n} \cdot [D_{\Gamma}^\cH] \]
since the $\psi$ classes are compatible under clutchings and projections.

If one knows the class $[\ol\cH] \in
\CH_{\dim(\cH)}(\PLS(\mu))$ and this class happens to be tautological, one may evaluate
\[
\int_{\ol\cH} \alpha \= \int_{\PLS(\mu)} \psi_{1}^{p_1} \cdots \psi_{n}^{p_n}
\cdot [D_{\Gamma}] \cdot [\ol\cH]
\]
using the methods described in~\cite{CMZeuler}.  This has the advantage of not requiring the computation of the classes of all the level linear submanifolds $\cH_{\Gamma}^{[i]}$.
%%%%%%%%%%%%%%%%%%%%%%%%%%%%%%%
\subsection{Evaluation of $\xi_{\cH}$}
%%%%%%%%%%%%%%%%%%%%%%%%%%%%%%%
If we want to evaluate a top-degree class in $\CH_0(\ol \cH)$ that is not just
a product of $\psi$-classes and a boundary stratum, but also involves
the $\xi_\cH$-class, we can reduce to the previous case by applying the following
proposition.
\par
\begin{prop} \label{prop:Adrienrel}
The class $\xi_\cH$ on the closure of a projectivized linear submanifold $\ol \cH$
can be expressed as 
\be \label{eq:xirel}
\xi_\cH \= (m_{i}+1) \psi_{i} \,-\, \sum_{\Gamma \in \tensor[_{i}]{\LG}{_1}(\cH)}
\ell_\Gamma [D^\cH_\Gamma]\,
\ee
where $\tensor[_{i}]{\LG}{_1}(\cH)$ are two-level graphs with
the leg~$i$ on lower level.
\end{prop}
\begin{proof}
The formula is obtained by pulling-back the formula in
\cite[Proposition~8.1]{CMZeuler} to $\ol \cH$ and thereby using the transversality
statement from \autoref{prop:transverse}.
\end{proof}
\par
We remark here that in some cases it is possible to directly evaluate the top $\xi_\cH$-powers by using that we can represent the powers of the $\xi_\cH$-class via an explicit closed current.

Let $\bP\omoduli[g,n](\mu)$ be a \emph{holomorphic stratum}, i.e.\ a stratum of flat
surfaces of finite area or equivalently all the entries of $\mu$ are non-negative.
Then there is a canonical hermitian metric on the tautological bundle
$\cO_{\bP\omoduli[g,n](\mu)}(-1)$ given by the flat area form
\be \label{eq:defh}
h(X,\omega,\bfz) \= \area_{X}(\omega) \= \frac{i}{2} \int_{X} \omega \wedge 
\overline{\omega}
\ee
which extends to an hermitian metric of the tautological bundle on $\bP\LMS$.
If $\ol\cH\to \bP\LMS$ is the compactification of  a linear submanifold of such
a holomorphic stratum, then the area metric induces an hermitian metric, which we denote
again by~$h$, on the pull-back $\cO_{\ol\cH}(-1)$ of the tautological bundle
to~$\ol\cH$. Recall from \autoref{prop:bdtoricsing} (combined with the level-wise
decomposition in \autoref{prop:linatboundary}) that the singularities of~$\ol\cH$
are toric. Let $\ol\cH^{\tor} \to \ol\cH$ be a resolution of singularities which
is locally toric.
\par
\begin{prop}\label{prop:goodness}
Let $\ol\cH^{\tor} \to \bP\LMS$ be a resolution of a compactified linear submanifold
of a holomorphic stratum. The curvature form  $\tfrac{i}{2\pi}[F_h]$ of the pull
metric~$h$ to~$\ol\cH^{\tor}$ is a closed current that represents the
first Chern class $c_1(\cO_{\ol\cH^{\tor}}(-1))$. More generally, the $d$-th wedge power
of the curvature form represents $c_1(\cO_{\ol\cH^{\tor}}(-1))^d$ for any $d \geq 1$.
\end{prop}
\begin{proof}
In \cite[Proposition~4.3]{CoMoZa} it was shown that on the neighborhood~$U$ of
a boundary point of $\bP\LMS$ in the interior of the stratum $D_\Gamma$
the metric~$h$ has the form 
\begin{equation}\label{eq:hshapegeneral}
h(X,q)\= \sum_{i=0}^L  |\tui|^2 \left(h^{\text{tck}}_{(-i)}
+h^{\text{ver}}_{(-i)}+h^{\text{hor}}_{(-i)}\right)
\end{equation}
where $h^{\text{tck}}_{(-i)}$ (coming from the 'thick' part) are smooth positive
functions bounded away from zero and 
\begin{align}\label{eq:cfunctions}
h^{\text{ver}}_{(-i)}:=-\sum_{p=1}^{i}R^{\text{ver}}_{(-i),p} \log\left|t_p\right|,\quad h^{\text{hor}}_{(-i)}:=- \sum_{j=1}^{E_{(-i)}^h} R^{\text{hor}}_{(-i),j} \log|q^{[i]}_{j}|\,,
\end{align}
where $R^{\text{ver}}_{(-i),p}$  is a smooth non-negative function and
$R^{\text{hor}}_{(-i),j}$   is a smooth positive function bounded away from zero, both
involving only perturbed period coordinates on levels~$-i$ and below.
\par
The statement of the proposition in loc. cit. follows by formal computations from the shape
of~\eqref{eq:hshapegeneral} and the properties of its coefficients, see
\cite[Proposition~4.4 and~4.5]{CoMoZa}. We thus only need to show that in local
coordinates of a point in~$\ol\cH^{\tor}$ (mapping to the given stratum~$D_\Gamma$)
the metric has the same shape~\eqref{eq:hshapegeneral}. For this purpose,
recall that by \autoref{prop:linatboundary}, the level parameters~$t_i$ are
among the coordinates. On the other hand, a toric resolution of the toric
singularities arising from~\eqref{eq:toriceqns} is given by fan subdivision
and thus by a collection of variables $y_j^{[i]}$ for each level~$i$, each of
which is a product of integral powers of the $q_j^{[i]}$ at that level~$i$.
Conversely the map $\ol\cH^{\tor} \to \bP\LMS$ is given locally by
$q_j^{[i]} = \prod_k (y_k^{[i]})^{b_{i,j,k}}$ for some $b_{i,j,k} \in \bZ_{\geq 0}$, not
all of the $b_{i,j,k} = 0$ for fixed $(i,j)$. Plugging
this into~\eqref{eq:hshapegeneral} and~\eqref{eq:cfunctions} gives an
expression of the same shape and with coefficients satisfying the same smoothness
and positivity properties. Mimicking the proof in loc.\ cit. thus implies the claim.
\end{proof}
\par
For a linear submanifold $\cH$ consider the vector space given  in local period
coordinates  by the intersection of the tangent space of the unprojectivized linear
submanifold with the span of relative periods. We call this space the REL space
of~$\cH$ and we denote by $R_\cH$  its dimension.
\par
Using \autoref{prop:goodness} we can now generalize the result about vanishing of top $\xi$-powers on non-minimal strata of differentials to linear submanifolds with non-zero REL (see \cite[Proposition~3.3]{SauvagetMinimal} for the holomorphic Abelian strata case). 
\par
\begin{cor}\label{cor:topxivanishing}
Let $\ol\cH\to \bP\LMS$ be a linear submanifold of a holomorphic stratum. Then
\[\int_{\ol\cH} \xi_{\ol\cH}^i \alpha \=0 \quad \text{ for $i\geq d_\cH-R_\cH+1$,}\]
where $d_\cH$ is the dimension of $\cH$ and $R_\cH$ is the dimension of the REL space
and where $\alpha$ is any class of dimension~$d_\cH - i$.
\end{cor}
\begin{proof} Since the area is given by an expression in absolute periods, the
pullback of~$\xi$ to~$\ol{\cH}^{\tor}$ is represented by \autoref{prop:goodness} by
a $(1,1)$-form involving only absolute periods (see \cite[Lemma~2.1]{SauvagetMinimal}
for the explicit expression in the case of strata). Taking a wedge power that
exceeds the dimension of the space of absolute periods gives zero.
\end{proof}

%%%%%%%%%%%%%%%%%%%%%%
\subsection{Normal bundles}\label{subsec:normalbundles}
%%%%%%%%%%%%%%%%%%%%%%
Finally we state the normal bundle formula, which is necessary to evaluate
self-intersections, which is  for example needed to evaluate powers of $\xi_\cH$. 
More generally, we provide formulas for the normal bundle of an inclusion
$\frakj_{\Gamma^+,\Pi^+}\colon D^\cH_{\Gamma^+} \hookrightarrow D^\cH_{\Pi^+}$ between irreducible components of non-horizontal boundary strata of relative codimension one, say defined by the $L$-level graph~$\Pi$ and one of its $(L+1)$-level graph degenerations~$\Gamma$.
This generalization is needed for recursive evaluations. Such an inclusion
is obtained by splitting one of the levels of~$\Pi^+$, say
the level~$i\in \{0,-1,\dots,-L\}$. Here we use the structure of the equations cutting out the linear manifold in \autoref{sec:decompLTB} to observe that $j$ is a regular embedding, in fact with ideal sheaf locally generated by the parameter $t_i$, to talk about normal bundles (as opposed to merely normal sheaves). In particular these regular embeddings~$j$ and thus also their compositions~$i$ come with classes in operational Chow groups (see \cite[Section~17]{Fulton} for background and e.g.\ \cite[Section~2]{BHPSS} for the extension to stacks). This is the language that justifies all the intersection theory we need working on the (singular) stack $\ol{\cH}$. We do not reflect this in our notation of Chow groups since for the morphisms we consider, all formulas of the classical setting carry over. We define 
\be \label{eq:defLithlev}
\cL_{\Gamma^+}^{[i]} \=  \cO_{D^\cH_{\Gamma^+}}
\Bigl(\sum_{\Gamma^+ \overset{[i]}{\rightsquigarrow} %^{[i]}
	\widetilde{\Delta}^+ } \ell_{\widetilde{\Delta}^+,-i+1}D^\cH_{\widetilde{\Delta}^+} \Bigr)
\quad \text{for any}  \quad i\in \{0,-1,\dots,-L\}\,,
\ee
where the sum is over all refined graphs $\widetilde{\Delta}^+ \in \LG^+_{L+2}(\cH)$
that yield divisors in~$D^\cH_{\Gamma^+}$ by splitting the $i$-th level, which in
terms of undegenerations means $\delta_{-i+1}^\complement
(\widetilde{\Delta}^+) = \Gamma^+$.
The following result contains the formula for the normal bundle as the
special case where~$\Pi$ is the trivial graph.
\par
\begin{prop}\label{prop:generalnormalbundle}
For $\Pi^+ \overset{[i]}{\rightsquigarrow} \Gamma^+$ (or equivalently
for $\delta_{-i+1}^\complement(\Gamma^+)=\Pi^+$) the Chern class of the normal
bundle $\cN^\cH_{\Gamma^+,\Pi^+} := \cN_{D^\cH_{\Gamma^+}/D^\cH_{\Pi^+}}$  is given by
\be \label{eq:nbinPi}
c_1(\cN^\cH_{\Gamma^+,\Pi^+}) \= \frac{1}{\ell_{\Gamma,(-i+1)}} \bigl(-\xi_{\Gamma^+,\cH}^{[i]}
- c_1(\cL_{\Gamma^+,\cH}^{[i]}) + \xi_{\Gamma^+,\cH}^{[i-1]} \bigr)\quad \text{in} \quad
\CH^1(D^\cH_{\Gamma^+})\,.
\ee
\end{prop}
\begin{proof}
 We use the transversality statement \autoref{prop:transverse}
of $\cH$ with a boundary stratum $D^B_{\Gamma^+}$ in order to have that the transversal
parameter is given by~$t_i$. Then the proof is as the same as the one in the
case of Abelian strata, see \cite[Proposition~7.5]{CMZeuler}. 
\end{proof}
\par
Since in \autoref{sec:BQ} we will need to compute the normal bundle to horizontal divisors for strata of $k$-differentials, we provide here the general formula for the case of smooth horizontal degenerations of linear submanifolds.

\begin{prop}\label{prop:normalhor}
Let $D^\cH_{h}\subset D^\cH$ be a divisor in a boundary component $D^\cH$ obtained
by horizontal degeneration. Suppose that the linear submanifold is smooth along $D^\cH_{h}$
and let~$e$ be one of the new horizontal edges in the level graph of~$D^\cH_{h}$.
Then the  first Chern class of the normal bundle $\cN^\cH_{D_{h}}$ is given by
	\[c_1(\cN^\cH_{D_{h}}) \=-\psi_{e^+}-\psi_{e^-}\in \CH^1(D^\cH)\]
where  $e^+$ and $e^-$ are the half-edges associated to the two ends of $e$.
\end{prop}
\par
\begin{proof}
Similarly to the proof of \cite[Proposition~7.2]{CMZeuler}, consider the divisor~$D_e$
in $\barmoduli[g,n]$ corresponding to the single edge~$e$ and denote by~$\cN_e$ its
normal bundle. The forgetful map $f: D_{h} \to D_e$ induces an isomorphism
$\cN^\cH_{D_{h}} \to f^* \cN_{D_e}$ (compare local generators!) and the formula follows
from the well-known expression of $\cN_{D_e}$ in terms of $\psi$-classes.
\end{proof}
\par
We will need the following result about pullbacks of normal bundles to apply the
same arguments as in \cite{CMZeuler} recursively over inclusions of boundary divisors.
The proof is the same as in \cite[Corollary~7.7]{CMZeuler},
since it follows from \autoref{prop:generalnormalbundle} that we can $\frakj$-pullback
properties of $\xi$ and $\cL_\Gamma^{[i]}$ that hold on the whole stratum and
hence on linear submanifolds.
\par
\begin{lemma} \label{lemma:pullbacknormal}
Let $\Gamma^+\in\LG^+_L(\cH)$ and let $\widetilde{\Delta}^+$ be a codimension one
degeneration of the $(-i+1)$-th level of $\Gamma^+$, i.e., such that $\Gamma^+=
\delta_i^\complement(\widetilde{\Delta}^+)$, for some $i\in \{1,\dots,L+1\}$. Then  
\[\frakj_{\widetilde{\Delta}^+,\Gamma^+}^*\left(\ell_{\Gamma,j}
\c_1\bigl(\cN^\cH_{\Gamma^+/\delta_{j}^\complement(\Gamma^+)}\bigr)\right)
\= \begin{cases}
	\ell_{\widetilde{\Delta},j}\,\,\c_1\left(\cN^\cH_{\widetilde{\Delta}^+/\delta_{j}^\complement(\widetilde{\Delta}^+)}\right)
	,&  \text{ for } j< i \\
	\ell_{\widetilde{\Delta}^+,j+1}\c_1\left(\cN^\cH_{\widetilde{\Delta}^+/\delta_{(j+1)}^\complement(\widetilde{\Delta}^+)}\right)
	& \text{ otherwise.}
\end{cases}\]
\end{lemma}

%%%%%%%%%%%%%%%%%%%%%%%%%%%%%%%%%%%%%%%%%%%%%%%%%%%%%%%%%%%%

%%%%%%%%%%%%%%%%%%%%%%%%%%%%%%%%%%%%%%%%%%%%%%%%%%%%%%%%%%%%

%%%%%%%%%%%%%%%%%%%%%%%%%%%%%%%%%%%%%%%%%%%%%%%%%%%%%%%%%%%%%%%%%%
\section{Chern classes of the cotangent bundle via the Euler sequence}
\label{sec:chern}
%%%%%%%%%%%%%%%%%%%%%%%%%%%%%%%%%%%%%%%%%%%%%%%%%%%%%%%%%%%%%%%%%%

The core of the computation of the Chern classes is given by
two exact sequences that are the direct counterparts of the corresponding
theorems for Abelian strata. The proof should be read in parallel with
\cite[Section~6 and~9]{CMZeuler} and we mainly highlight the differences
and where the structure theorems of the compactification from
\autoref{sec:closurelin} are needed.
\par
\begin{theorem} \label{thm:EulerDE}
There is a vector bundle~$\cK$ on $\ol{\cH}$ that fits into an exact
sequence
\be \label{eq:EulerExt}
0\longrightarrow \cK \overset{\psi}
\longrightarrow (\Hrelbar)^\vee\otimes \cO_{\overline{\cH}}(-1)
\overset{\ev}{\longrightarrow}   \cO_{\overline{\cH}}\longrightarrow 0\,,
\ee
where $\Hrelbar$ is the Deligne extension of the local subsystem
that defines the tangent space to~$\Omega\cH$ inside the relative cohomology
$\HrelB|_{\ol{\cH}}$,
such that the restriction of $\cK$ to the interior~$\cH$ is the cotangent
bundle $\Omega^1_{\cH}$ and for~$U$ as in \autoref{prop:Omegadecomp} we have
\bes
\cK|_U \= \bigoplus_{i=-L}^0  \prodt[-i] \cdot \Bigl(\Omega_i^{\hor}(\log) \oplus
\Omega_i^{\lev}(\log) \oplus \Omega_i^{\rel} \Bigr).
\ees
\end{theorem}
\par
The definition of the evaluation map and the notion of Deligne extension
on a stack with toric singularities requires justification given in the
proof. For the next result we define the abbreviations 
\be \label{eq:ELgeneral}
\cE_\cH \= \Omega^1_{\overline{\cH}}(\log \ol\partial\cH) \quad \text{and} \quad
\cL_\cH \= \cO_{\overline{\cH}}\Bigl( \sum_{\Gamma \in \twolev} \ell_\Gamma D^\cH_\Gamma\Bigr)
\ee
that are consistent with the level-wise definitions in~\eqref{eq:cEtop}
and~\eqref{eq:defLithlev}.
\par
\begin{theorem}	\label{thm:coker}
There is a short exact sequence of quasi-coherent $\cO_{\overline{\cH}}$-modules    
\begin{equation} \label{eq:maineq}
0\longrightarrow \cE_\cH \otimes \cL_\cH^{-1} \to \cK \to \cC\longrightarrow 0
\end{equation}
where $\cC \= \bigoplus_{\Gamma \in \LG_1(\cH)} \cC_\Gamma$ is a coherent sheaf
supported on the non-horizontal boundary divisors, whose precise form
is given in \autoref{prop:coker} below.
\end{theorem}
\par
\begin{proof}[Proof of \autoref{thm:EulerDE}] We start with the
definition of the maps in the Euler sequence for the ambient stratum,
see the middle row in the commutative diagram below. It uses the 
evaluation map
\be \label{eq:defev}
\ev_B\colon(\HrelB)^\vee\otimes \cO_{\ol{B}}(-1)\to \cO_{\ol{B}}, \quad
\gamma\otimes \omega\mapsto \int_{\gamma} \omega\,,
\ee
restricted to $\ol{\cH}$. The first map in the sequence is 
\be \label{eq:eulerinbasis}
\dd c_i\mapsto \Bigl(\gamma_i-\frac{c_i}{c_k}\alpha_k\Bigr)\otimes \omega,\qquad
i=1,\dots,\hat{k},\dots,N\,,
\ee
as usual in the Euler sequence, on a chart of~$\cH$ where $c_k$ is
non-zero. The exactness of the middle row is the content of
\cite[Theorem~6.1]{CMZeuler}. 
\par
We next define the sheaf $\Eq$. In the interior,
$\Eq$ is the local system of equations cutting out~$\Omega\cH$, and thus
the quotient $(\cH_{\rel}^1)^\vee = (\cH_{\rel,B}^1)^\vee / \Eq$ is the
relative homology local system, by definition of a linear manifold.
The proof in \cite[Section~6.1]{CMZeuler} concerning the restriction of
the sequence to the interior~$\cH$ uses that $\cH$ has
a linear structure with tangent space modeled on the local system
$\bH^1_{\mathrm{rel}}$. In particular it gives the claim about $\cK|_\cH$.
\par
As an interlude, we introduce notation for the Deligne extension of
$(\cH_{\rel,B}^1)^\vee$. For each $\gamma_j^{[i]}$ we let $\wh{\gamma}_j^{[i]}$
be it extension, the sum of the original cycles and vanishing cycles
times logarithms of the coordinates of the boundary divisors to kill
monodromies. The functions
\bes
\wh{c}_j^{[i]} \= \frac1{\prodt[-i]} \int_{\wh{\gamma}_j^{[i]}} \omega
\ees
%$\wh{c}_j^{[i]} \=  \int_{\wh{\gamma}_j^{[i]}} \omega /{\prodt[i]}$
are called \emph{log periods} in \cite{BDG}.
\par
We now \emph{define} $\Eq $ at the boundary, say
locally near a point $p \in D_\Gamma$, to be the subsheaf of $(\HrelB)^\vee$ generated by the defining equations~$F_k^{[i]}$  constructed in \autoref{sec:closurelin}, but
with each variable replaced by its Deligne extension. It requires justification
that this definition near the boundary agrees with the previous
definition in the interior. We can verify this for the distinguished
basis consisting of the $F_k^{[i]}$.
Equations that do not intersect horizontal nodes agree with their Deligne
extension. This cancellation of the compensation terms is \cite[Proposition~3.11]{BDG} (
see also the expression for $F_k^{[i]}$ after \cite[Proposition~4.1]{BDG}) which
displays the $\omega$-integrals of the terms to be compared. For
equations $F_k^{[i]}$ that do intersect horizontal nodes (thus only at
level~$i$ by construction) the difference $F_k^{[i]}(c_j^{[s]}, \text{all $(j,s)$})
- F_k^{[i]}(\wh{c}_j^{[s]}, \text{all $(j,s)$})$ vanishes thanks to the 
proportionality of the periods of horizontal nodes in an $\cH$-equivalence class
and since on~$\ol{\cH}$ the equation~$H_k^{[i]}$ holds. 
\par
By the very definition of defining equation its periods evaluate to zero,
explaining the right arrow in the top row of the following diagram and 
showing that~$\ev$ is well-defined on the quotient.
\bes
\begin{tikzcd}
0 \arrow{r} & \cK_\Eq \arrow{r} \arrow{d}& \Eq
\otimes \cO_{\overline{\cH}}(-1) \arrow{r} \arrow{d}& 0 \arrow{d} \\
0 \arrow{r} & \bigoplus \limits_{i=0}^L \prodt[-i] \cdot 
\Omega^{[i]}_{B}|_{\ol{\cH}} \arrow{r}{\psi} \arrow{d}{q_\Omega} &
(\HrelB)^\vee\otimes \cO_{\overline{\cH}}(-1) \arrow{r}{\ev_B} \arrow{d}&
\cO_{\overline{\cH}} \arrow{r} \arrow{d} & 0 \\
0 \arrow{r} & \bigoplus \limits_{i=0}^L \prodt[-i] \cdot 
\Omega^{[i]}  \arrow{r} &
(\Hrelbar)^\vee\otimes \cO_{\overline{\cH}}(-1) \arrow{r}{\ev} &
\cO_{\overline{\cH}} \arrow{r} & 0\,.
\end{tikzcd}
\ees
Here we used the abbreviations
\bes
\Omega^{[i]}_{B}  \= \Omega_{i,B}^{\hor}(\log) \oplus
\Omega_{i,B}^{\lev}(\log) \oplus \Omega_{i,B}^{\rel}, \qquad
\Omega^{[i]}  \= \Omega_i^{\hor}(\log) \oplus
\Omega_i^{\lev}(\log) \oplus \Omega_i^{\rel}\,.
\ees
\par
The surjectivity of $q_\Omega$ follows from the definition of the
summands in~\eqref{eq:OmDecompH}. It requires justification that the
image is not larger, since the derivatives of the local equations
of~$\cH$ do not respect the direct sum decomposition~\eqref{eq:OmDecompB}.
More precisely we claim that $\cK_\Eq$ is generated by two kinds of
equations. Before analyzing them, note that the log periods satisfy
by construction an estimate of the form 
\be \label{eq:chatvsctilde}
\wt{c}_j^{[-i]} - \wh{c}_j^{[-i]} \=  \sum_{s>i} \frac{\prodt[s]}{\prodt[i]}
\wh{E}_{j,i}^{[-s]} 
\ee
with some error term $\wh{E}_{j,i}^{[-k]}$ depending on the variables
$c_j^{[-s]}$ on the lower level~$-s$ as in~\eqref{eq:ccdiff}.
\par
For each of the equations~\eqref{eq:lineareqns} the corresponding linear
function $L_k^{[i]}$ in the variables~${c}_j^{[i]}$ is an element in~$\Eq$.
We use the comparisons~\eqref{eq:chatvsctilde} and~\eqref{eq:ccdiff}
to compute its $\psi$-preimage in $\cK_\Eq$ via~\eqref{eq:eulerinbasis}. It is
$\prodt[-i]$ times the corresponding expression in the~$\wh{c}_j^{[i]}$ plus
a linear combination of the terms $\prodt[-s] \wh{E}_{j,i}^{[s]}$. The quotient
by such a relation does not yield any quotient class beyond those
in $\oplus_{i=0}^i \prodt[-i] \cdot \Omega^{[i]}$.
\par
We write the other equations~\eqref{eq:toriceqns} as $(\bfq^{[i]})^{J_{1,k}-J_{2,k}}
= 1$ since we are interested in torus-invariant differential forms and
can compute on the boundary complement. Consider $d\log$ of this equation.
Under the first map~$\psi$ of the Euler sequence
\be
dq_j^{[i]}/q_j^{[i]}\=d\log(q_j^{[i]})\=d\left(2\pi I \frac{b_j^{[i]}}{a_j^{[i]}}\right)
\, \mapsto \, \frac{2\pi I}{a_j^{[i]}} \Bigl(\beta_j^{[i]} - \frac{b_j^{[i]}}
{a_j^{[i]}} \alpha_j^{[i]}\Bigr) \otimes \omega
\ee
Recall from summary of~\cite{BDG} in \autoref{sec:closurelin}
that the functions ${a_j^{[i]}}$ for all $j$ where
$(v_1,\ldots,v_{N(i)-h(i)}) := J_{1,k}-J_{2,k}$ is non-zero are rational
multiples of each other. Note  moreover that
$\beta_j^{[i]} - \tfrac{b_j^{[i]}} {a_j^{[i]}} \alpha_j^{[i]}
= \beta_j^{[i]} - \tfrac1{2\pi I}\log(q_j^{[i]}) \alpha_j^{[i]}$ is the
Deligne extension of $\beta_j^{[i]}$ across all the boundary divisors that stem
from horizontal nodes at level~$i$.
For the full Deligne extension $\wh{\beta}_j^{[i]}$ the correction terms for
the lower level nodes have to be added. Together with~\eqref{eq:logqdiff}
we deduce that the $\psi$-image of
\bes
\sum_{m=1}^{h(i)} v_m {a_m^{[i]}} \,\frac{d\wt{q}_m^{[i]}}{\wt{q}_m^{[i]}}
\= \sum_{m=1}^{h(i)} v_j c_{j(m)}^{[i]} \,\frac{d\wt{q}_m^{[i]}}{\wt{q}_m^{[i]}}
\ees
differs from the element in~$\Eq$ responsible for the equation~$H_k^{[i]}$
only by terms from lower level~$s$, which come with a factor $\prodt[-s]$.
In this equation used that $a_m^{[i]} = {c}_{j(m)}^{[i]}$ for an
appropriate~$j(m)$. Since~${c}_{j(m)}^{[i]}$ is close to $\prodt[-i] \wt{c}_{j(m)}^{[i]}$,
compare with~\eqref{eq:ccdiff} this element indeed belongs to the
kernel of~$\psi$ as claimed in the commutative diagram. 
The quotient by such a relation does not yield any quotient class beyond
those above either. Since the~\eqref{eq:toriceqns} and~\eqref{eq:lineareqns}
correspond to a basis (in fact: in reduced row echelon form) of~$\Eq$,
this completes the proof.
\end{proof}
\par
\begin{proof}[Proof of \autoref{thm:coker}]
Uses that the summands of $\cK|_U$ are, up to $t$-powers, the decomposition
of the logarithmic tangent sheaf by \autoref{prop:Omegadecomp}.
\par
\end{proof}
\par
\begin{cor} \label{cor:cKclasses}
The Chern character and the Chern polynomial of the kernel~$\cK$ of
the Euler sequence are  given by
\[\ch(\cK)\=Ne^{\xi_\cH}-1 \quad\text{ and } \quad
\c(\cK) \= \sum_{i=0}^{N-1}\binom{N}{i}\xi_\cH^i\,.\]
\end{cor}
\par
\begin{proof}
As a Deligne extension of a local system, $(\HrelB)^\vee|_{\ol{\cH}}$ has
trivial Chern classes except for $c_0$.  By construction, the pullback of the
sheaf $\Eq$  to an allowable modification (toric resolution with normal crossing
boundary, see the proof of \autoref{prop:chiviaTlog}) is the Deligne extension
of a local system. It follows that all Chern classes but $c_0$ of this pullback
vanish and by push-full this holds for $\Eq$, too. The Chern class vanishing for $(\Hrel)^\vee$
and the corollary follows.
\end{proof}
To start with the computation of~$\cC$, we will also need an infinitesimal
thickening the of the boundary divisor~$D^\cH_\Gamma$, namely we define
$D^\cH_{\Gamma,\bullet}$ to be its $\ell_\Gamma$-th thickening, the non-reduced
substack of $\ol\cH$ defined by the ideal $\cI_{D^\cH_\Gamma}^{\ell_\Gamma}$. We will
factor the above inclusion using the notation
\[\fraki_\Gamma = \fraki_{\Gamma,\bullet} \circ j_{\Gamma,\bullet} \colon
D^\cH_\Gamma \, \overset{j_{\Gamma,\bullet}}{\hookrightarrow} \, D^\cH_{\Gamma, \bullet} \,
\overset{\fraki_{\Gamma,\bullet}}{\hookrightarrow}\, \overline{\cH}\,. \]
We will denote by $\cL_{\Gamma,\bullet}^\top=(j_{\Gamma,\bullet})_*(\cL_{\Gamma}^\top)$ and  $\cE_{\Gamma,\bullet}^\top=(j_{\Gamma,\bullet})_*(\cE_{\Gamma}^\top)$ the push-forward to the thickening of the vector bundles defined in \eqref{eq:defLithlev} and \eqref{eq:cEtop}.

\begin{prop}\label{prop:coker}
		The cokernel of~\eqref{eq:maineq} is given by 
		\be \label{eq:defcalC}
		\cC \= \bigoplus_{\Gamma \in \twolev} \cC_\Gamma \quad \text{where}
		\quad \cC_\Gamma \= (\fraki_{\Gamma,\bullet})_* (\cE_{\Gamma,\bullet}^\top
		\otimes (\cL_{\Gamma,\bullet}^\top)^{-1})\,.
		\ee
Moreover, there is an equality of Chern characters
\[
\ch\Bigl((\fraki_{\Gamma,\bullet})_*(\cE_{\Gamma,\bullet}^\top \otimes
(\cL_{\Gamma,\bullet}^\top)^{-1}) \Bigr) \= \ch \Bigl((\fraki_\Gamma)_*
\bigl(\bigoplus_{j=0}^{\ell_\Gamma-1} \cN_{\Gamma}^{\otimes -j} \otimes
\cE_{\Gamma}^\top \otimes (\cL_{\Gamma}^\top)^{-1}\bigl)\Bigr)\,.
\]
\end{prop}
\begin{proof}
The second part of the statement is justified by the original argument in \cite[Lemma 9.3]{CoMoZa}.

The first part of the statement follows since, from \autoref{thm:EulerDE} we know that 
\bes
\cK|_U \= \bigoplus_{i=-L}^0  \prod_{j=1}^{-i} t_j^{\ell_j} \cdot \Bigl(\Omega_i^{\hor}(\log) \oplus
\Omega_i^{\lev}(\log) \oplus \Omega_i^{\rel} \Bigr)
\ees
 and from \autoref{prop:Omegadecomp} we also know that
 \be
(\cE_\cH \otimes \cL_\cH^{-1})|_U
 \= \bigoplus_{i=-L}^0 \prod_{j=1}^L t_j^{\ell_j} \cdot \Bigl(\Omega_i^{\hor}(\log) \oplus
 \Omega_i^{\lev}(\log) \oplus \Omega_i^{\rel} \Bigr)
 \ee
where  $\Gamma$ is an arbitrary level graph with~$L$
levels below zero and $U$ is a small neighborhood of a point in $D^{\cH,\circ}_\Gamma$.
\end{proof}
We can finally compute 
\begin{prop} \label{prop:prerecursion}
	The Chern character of the twisted logarithmic cotangent bundle~$\cE_\cH \otimes \cL_\cH^{-1}$ can be
	expressed in terms of the twisted logarithmic cotangent bundles of the top levels of non-horizontal divisors as
	\bas
	\ch(\cE_\cH \otimes \cL_\cH^{-1}) \=  N e^{\xi} -1  \,-\,  \sum_{\Gamma \in \twolev}
	{\fraki_\Gamma}_* \left(
	\ch(\cE_{\Gamma}^\top) \cdot \ch(\cL_{\Gamma}^\top)^{-1}
	\cdot  \frac{(1-e^{-\ell_\Gamma\c_1(\cN_{\Gamma})})}{\c_1(\cN_{\Gamma})}\right)\,.
	\eas
\end{prop}
\begin{proof}
The proof \cite[Prop. 9.5]{CoMoZa} works in the same way, since the only tool that was used is  the Grothendieck-Riemann-Roch Theorem applied to the
map $f=\fraki_\Gamma$, which is still a regular embedding. 
\end{proof}
\par
\begin{proof}[Proof of Theorem~\ref{thm:c1cor} and Theorem~\ref{intro:Chern}]
The final formulas of the full twisted Chern character, Chern polynomials
and Euler characteristic follow from the arguments used for Abelian strata in
\cite[Section~9]{CoMoZa}, since they were purely formal starting from the
previous proposition. Here we need a more refined sum distinguishing irreducible components, but this works since the relevant inputs needed are the compatibility
statement of \autoref{lem:compatibilities}, the formula for pulling back
normal bundles given in \autoref{lemma:pullbacknormal} and
\autoref{cor:ordering} which work in this more refined setting.
\end{proof}
\par
\begin{proof}[Proof of Theorem~\ref{intro:ECformula}]
 A formal consequence of Theorem~\ref{intro:Chern} and the rewriting in 
\cite[Theorem~9.10]{CMZeuler} (with the reference to 
\cite[Proposition~4.9]{CMZeuler} replaced by Lemma~\ref{lem:compatibilities})
is 
\begin{equation}
	\label{eq:eulerH}
	\chi(\cH) =
	(-1)^d
	\sum_{L=0}^d \sum_{\Gamma^+ \in \LG^+_L(\cH)} N_{\Gamma^+}^\top \cdot \ell_{\Gamma^+} \cdot
	\int_{D_{\Gamma^+}^\cH} \prod_{i=-L}^{0}(\xi_{\Gamma^+,\cH}^{[i]})^{d_{\Gamma^+}^{[i]}},
\end{equation}
We now use 
\autoref{lem:evaluation} to convert integrals on a boundary component
into the product of integrals of its the level strata.
\end{proof}

%%%%%%%%%%%%%%%%%%%%%%%%%%%%%%%%%%%%%%%%%%%%%%%%%%%%%%%%%%%%

%%%%%%%%%%%%%%%%%%%%%%%%%%%%%%%%%%%%%%%%%%%%%%%%%%%%%%%%%%%%
%%%%%%%%%%%%%%%%%%%%%%%%%%%%%%%%%%%%%%%%%%%%%%%%%%%%%%%%%%
\section{Example: Euler characteristic of the eigenform locus}
\label{sec:examples}
%%%%%%%%%%%%%%%%%%%%%%%%%%%%%%%%%%%%%%%%%%%%%%%%%%%%%%%%%%

For a non-square $D \in \bN$ with $D \equiv 0$ or $1 \pmod{4}$ let
\bes
\Omega E_D(1,1) \subseteq \omoduli[2,2](1,1) \qquad \text{and} \qquad
\Omega W_D \subseteq \omoduli[2,1](2)
\ees
be the eigenform loci for real multiplication by $\cO_D$ in the given
stratum, see \cite{McMBi}, \cite{Calta}, \cite{McMFol} for the first proofs that these loci
are linear submanifolds and some background. We define $E_D := \bP \Omega E_D(1,1)$ and \emph{Weierstrass curve} $W_D:=\bP \Omega W_D $
as the projectivized eigenform loci. Associating with the curve its Jacobian,
the projectivized eigenform loci map to the \emph{Hilbert modular surface}
\bes
X_D \= \bH \times \bH / \SL(\cO_D \oplus \cO_D^\vee)\,.
\ees
Inside $X_D$ let $P_D \subseteq X_D$ denote the \emph{product locus}, i.e.\ the
curve consisting of those surfaces which are polarized products of elliptic curves.
The images of $E_D$ and $W_D$ are  contained in the complement $X_D \setminus P_D$.
\par
The goal of this section is to provide a new short proof of \autoref{thm:ECWDintro}.
\begin{proof}[Proof of \autoref{thm:ECWDintro}]
The Hilbert modular surface~$X_D$ is
the disjoint union of the symmetrization of the eigenform locus $E_D \subset
\omoduli[2,1](1,1)$, the product locus~$P_D$ of reducible Jacobians
and the Teichm\"uller curve~$W_D$.  
This gives
\begin{equation}\label{eq:chiXD}
\chi(P_D) + \chi(W_D) + \frac12 \chi(E_D) \= \chi(X_D)\,.
\end{equation}
The numerical input is
\begin{equation}\label{eq:numinputs}
\chi(X_D) = 2 \zeta(-1) \quad \text{and} \qquad
\chi(P_D) = - \frac{5}{2} \chi(X_D) = -5 \zeta(-1),
\end{equation}
where $\zeta = \zeta_{\bQ(\sqrt D)}$ is the Dedekind zeta function. The first formula
is due to Siegel \cite{Siegel}, see also \cite[Theorem~IV.1.1]{vdG}, the
second is given in \cite[Theorem~2.22]{Bai07} viewing $P_D$ as
the vanishing locus of the product of odd theta functions .

We are left to compute $\chi(E_D)$, which we will do using the formula for the Euler characteristic provided in Theorem~\ref{intro:ECformula}. We first  need to list the vertical boundary strata of the linear submanifold $\overline{E_D}\subset \bP \Xi\cM_{2,2}(1,1)$. This list consists of two divisorial strata only, given in Figure~\ref{fig:Eigenformlocus}.
\begin{figure}[ht]
	\def\eq{=}
\newlength\leveldist
\setlength{\leveldist}{1.3cm}
\newlength\hordist
\setlength{\hordist}{1cm}

\begin{tikzpicture}[
		baseline={([yshift=-.5ex]current bounding box.center)},
		scale=2,very thick,
		node distance=\HoG, 
		bend angle=30,
		every loop/.style={very thick},
     		comp/.style={circle,fill,black,,inner sep=0pt,minimum size=5pt},	
		order top left/.style={pos=\PotlI,left,font=\tiny},
		order top right/.style={pos=\PotrI,right,font=\tiny},		
		bottom right with distance/.style={below right,text height=10pt},
		bottom left with distance/.style={below left,text height=10pt}]

%%%%%%%%%%%%%%%%%%%%%%%%%%
%%%%%%%%% LEVEL 1 %%%%%%%%%%%%
%%%%%%%%%%%%%%%%%%%%%%%%%%
\begin{scope}[local bounding box = l]
\node[circled number] (T1) [] {$1$}; 
\node[circled number] (T2) [right=of T1] {$1$}; 
\coordinate (mid) at ($(T1.south)!0.5!(T2.south)$);
\node[comp] (B) [below=of mid] {}
	edge
		node [order bottom left] {$-2$} 
		node [order top left, yshift=-1] {$0$} (T1)
	edge
		node [order bottom right] {$-2$} 
		node [order top right, yshift=-1] {$0$} (T2);
\node [bottom right with distance] (B-1) at (B.south east) {$1$};
\node [bottom left with distance] (B-2) at (B.south west) {$1$};
\path (B) edge [shorten >=4pt] (B-1.center);
\path (B) edge [shorten >=4pt] (B-2.center);
\end{scope}

\begin{scope}[shift={($(l.base) + (2\hordist,0)$)}, local bounding box = ct]
\node[circled number] (T) [] {$2$}; 
\node[comp] (B) [below=of T] {}
	edge 
		node [order bottom left] {$-4$} 
		node [order top left] {$2$} (T);
\node [bottom right with distance] (B-1) at (B.south east) {$1$};
\node [bottom left with distance] (B-2) at (B.south west) {$1$};
\path (B) edge [shorten >=4pt] (B-1.center);
\path (B) edge [shorten >=4pt] (B-2.center);
\end{scope}

%%% Labels %%%
\node [xshift=-5, anchor=north east] at (l.north west) {$\Gamma_P$};
\node [xshift=-5, anchor=north east] at (ct.north west) {$\Gamma_W$};

\end{tikzpicture}
	\caption{The boundary divisors of the eigenform locus $E$.} \label{fig:Eigenformlocus}
\end{figure}

Hence in this situation the formula  of \autoref{intro:ECformula} gives
\begin{align*}
	\chi(E_D)&=3\int_{E_D}\xi_{E_D}^2 +2\frac{K^{E_D}_{\Gamma_P}}{|\Aut_{E_D}(\Gamma_P)|}\int_{D_{\Gamma_P}^\top}\xi_{D_{\Gamma_P}^\top} \int_{D_{\Gamma_P}^\bot} 1\\
	&\quad \quad \quad + 2\frac{K^{E_D}_{\Gamma_W}}{|\Aut_{E_D}(\Gamma_W)|}\int_{D_{\Gamma_W}^\top}\xi_{D_{\Gamma_W}^\top} \int_{D_{\Gamma_W}^\bot} 1.
\end{align*}

Firstly note that the top-$\xi$-integral on $E_D$ vanishes by \autoref{cor:topxivanishing}, since $E_D$ is a linear submanifold with REL non-zero.

Secondly, the full automorphism groups of the graphs $\Gamma_P$ and $\Gamma_W$ are trivial and all three prong-matchings
	for~$\Gamma_W$ are reachable since they belong to one orbit of the prong rotation
	group. Hence we obtain
	\[\frac{K^{E_D}_{\Gamma_P}}{|\Aut_{E_D}(\Gamma_P)|}=1,\quad \frac{K^{E_D}_{\Gamma_W}}{|\Aut_{E_D}(\Gamma_W)|}=3.\]
	
	Thirdly, we can identify the  top levels $D_{\Gamma_P}^\top$ and $D_{\Gamma_W}^\top$ with $P_D$ and $W_D$ respectively. Hence, again by applying \autoref{intro:ECformula}  to the top level strata, we get
\[2\int_{D_{\Gamma_P}^\top}\xi_{D_{\Gamma_P}^\top}=-\chi(P_D),\quad 2\int_{D_{\Gamma_W}^\top}\xi_{D_{\Gamma_W}^\top}=-\chi(W_D).\] 

Finally, is it clear that
\[ \int_{D_{\Gamma_P}^\bot} 1 = 1, 
\quad \text{and} \quad \int_{D_{\Gamma_W}^\bot} 1 = 1\]
since there is unique differential up to scale of
type $(1,1,-2,-2)$ on a~$\bP^1$ with vanishing residues and $D_{\Gamma_W}^\bot\cong \cM_{0,3}$. 

From the previous computations we hence obtain that
\[\chi(E_D) \= -\chi(P_D) - 3\chi(W_D).\]
This, together with \eqref{eq:chiXD} and the numerical inputs \eqref{eq:numinputs}, yields the desired result
\[\chi(W_D)=-2\chi(X_D)+\chi(P_D)=-9\zeta(-1).\]
\end{proof}

%%%%%%%%%%%%%%%%%%%%%%%%%%%%%%%%%%%%%%%%%%%%%%%%%%%%%%%%%%%%

%%%%%%%%%%%%%%%%%%%%%%%%%%%%%%%%%%%%%%%%%%%%%%%%%%%%%%%%%%%%
%%%%%%%%%%%%%%%%%%%%%%%%%%%%%%%%%%
\section{Strata of $k$-differentials}
\label{sec:kdiff}
%%%%%%%%%%%%%%%%%%%%%%%%%%%%%%%%%%%%%%

Our goal here is to prove \autoref{cor:kstrata} that gives a
formula for the Euler characteristic of strata $\bP \Omega^k \cM_{g,n}(\mu)$
of $k$-differentials. Those strata can be viewed as linear submanifolds
of strata of Abelian differentials $\bP \omoduli[\wh{g},\wh{n}](\wh\mu)$ via
the canonical covering construction and thus \autoref{intro:ECformula} applies.
This is however of little practical use as we do not know the classes of
$k$-differential strata in $\bP \omoduli[\wh{g},\wh{n}](\wh\mu)$. However, we do
know their classes in $\ol\cM_{g,n}$  via Pixton's formulas for the DR-cycle
(\cite{HolmesSchmitt}, \cite{BHPSS}).
As a consequence the formula in \autoref{cor:kstrata} can
be implemented, and the \texttt{diffstrata} package does provide such
an implementation. In this section we thus recall the basic
definitions of the compactification and collect all the statements
to perform evaluation of expressions in the tautological rings on
strata of $k$-differentials.
\par
%%%%%%%%%%%%%%%%%%%%%%%
\subsection{Compactification of strata of $k$-differentials}
%%%%%%%%%%%%%%%%%%%%%%%
We want to work on the multi-scale compactification 
$\ol{\cQ} :=  \ol{\cQ}_k := \bP \Xi^k \ol \cM_{g,n}(\mu)$
of the space of $k$-differentials.  As topological
space this compactification was given in \cite{CoMoZa}, reviewing the
plumbing construction from \cite{LMS}, but without giving the stack
structure. Here we consider a priori the compactification of
\autoref{sec:closurelinsection}. We give some details, describing auxiliary stacks
usually by giving $\bC$-valued points and morphisms, from which the
reader can easily deduce the notion of families following the
procedure in \cite{LMS}. From this description it should become clear
that the two compactifications,  the one of \autoref{sec:closurelinsection} and
\cite{CoMoZa}, agree up to explicit isotropy groups (see \autoref{lem:deg_d}). In
particular the compactification $\ol{\cQ}_k$ is smooth. This follows also directly
from the definition of \autoref{sec:closurelinsection}, since the only potential
singularities are at the horizontal nodes. There however the local
equations~\eqref{eq:toriceqns}
simply compare monomials (with exponent one), the various $q$-parameters
of the $k$~preimages of a horizontal node. 
\par
We start by recalling notation for  the canonical $k$-cover \emph{in the
primitive case}. Let $X$ be a Riemann surface of genus $g$ and let~$q$ be a \emph{primitive}
meromorphic $k$-differential of type~$\mu=(m_1,\dots,m_n)$, i.e.\ not the $d$-th power of
a $k/d$-differential for any $d>1$. This datum defines (see e.g.\
\cite[Section~2.1]{kdiff}) a connected $k$-fold cover $\pi \colon \wh{X} \to X$
such that $\pi^* q = \omega^k$ is the $k$-power of an Abelian differential.
This differential~$\omega$ is of
type 
$$
\whmu \,:=\, \Bigl(\underbrace{\wh m_1, \ldots, \wh m_1}_{g_1:=\gcd(k,m_{1})},\,
\underbrace{\wh m_2,  \ldots, \wh m_2}_{g_2:=\gcd(k,m_{2})} ,\ldots,\,
  \underbrace{\wh m_n, \ldots, \wh m_n}_{g_n:= \gcd(k,m_{n})} \Bigr)\,,
  $$
where $\wh m_i := \tfrac{k+m_{i}}{\gcd(k,m_{i})}-1$. (Here and throughout
marked points of order zero may occur.) We let
$\wh{g} = g(\wh{X})$ and $\wh{n} = \sum_i \gcd(k,m_{i})$. The type
of the covering determines a natural subgroup $S_{\whmu} \subset S_{\wh{n}}$
of the symmetric group that allows only the permutations of each the
$\gcd(k,m_i)$ points corresponding to a preimage of the $i$-th point.
In the group $S_{\whmu}$ we fix the element
\be \label{eq:deftau}
\tau_0 \= \Bigl(12\cdots g_1\Bigr)\Bigl(g_1+1\ \, g_1+2\cdots g_1+g_2\Bigr)
\cdots \Bigl(1+\sum_{i=1}^{n-1}g_i \,\cdots \,\sum_{i=1}^{n}g_n \Bigr)\,,
\ee
i.e.\ the product of cycles shifting  the $g_i$ points in the
$\pi$-preimage of each point in~$\bfz$. We fix a primitive $k$-th root of unity
$\zeta_k$ throughout.
\par
We consider the stack $\Omega \cH_k := \Omega \cH_k(\wh{\mu})$
whose points  are
\be \label{eq:point_Hk}
\{(\wh{X},\wh{\bfz},\omega, \tau) \,:\, \tau \in \Aut(\wh{X}),
\quad \ord(\tau)\=k, \quad
\tau^* \omega \= \zeta_k \omega, \quad \tau|_{\wh{\bfz}} = \tau_0\} \,.
\ee
Families are defined in the obvious way. Morphisms are morphisms
of the underlying pointed curves that commute with~$\tau$. Since the
marked points determine the differential up to scale, the differentials
are identified by the pullback of morphisms up to scale. Commuting with~$\tau$
guarantees that morphisms descend to the quotient curves by $\langle \tau \rangle$
(for a morphism~$f$ to descend, a priori $f \tau f^{-1} = \tau^a$ for
some~$a$ would be sufficient, but the action on~$\omega$ implies that in fact
$a=1$). It will be convenient to label the tuple of points~$\wh{\bfz}$ by
tuples $(i,j)$ with $i=1,\ldots,n$ and $j = 1,\ldots,\gcd(k,m_i)$. 
There is a natural forgetful map $\Omega\cH_k \to \omoduli[\wh{g},\wh{n}]$
and period coordinates (say, after providing both sides locally with
a Teichm\"uller marking) show that this map is the normalization of its image
and the image is cut out by linear equations, i.e.\ that~$\Omega\cH_k$ is a
linear submanifold as defined in \autoref{sec:linsubmfd}.
\par
The subgroup
\be \label{eq:defG}
G \=  \Bigl\langle \Bigl(12\cdots g_1\Bigr), \Bigl(g_1+1\ \, g_1+2\cdots g_1+g_2\Bigr),
\cdots, \Bigl(1+\sum_{i=1}^{n-1}g_i \,\cdots \,\sum_{i=1}^{n}g_n \Bigr)
\Bigr \rangle \, \subset S_{\wh{\mu}}
\ee
generated by the cycles that~$\tau_0$
is made from  acts on $\Omega \cH_k$ and on the
projectivization $\cH_k$. We denote the quotient of the latter by
$\cH_k^\pmarked := \cH_k /G$,
where the upper index is an abbreviation of \emph{marked (only) partially}.
\par
Since $\tau$ has $\omega$ as eigendifferential, its $k$-th power
naturally descends to (projectivized) $k$-differential~$[q]$ on 
the quotient $X = \wh{X}/\langle \tau \rangle$, which is decorated
by the marked points~$\bfz$, the images of $\wh{\bfz}$.
\par
We denote by $\cQ$ the stack given by the rigidification of $\cH_k^\pmarked$ by the action of $\langle\tau\rangle$, i.e., the stack with the same underlying set as $\cH_k^\pmarked$,
but where morphisms are given by the morphisms of $(X/\langle \tau \rangle,
\bfz, [q])$ in $\bP \komoduli[g,n](\mu)$. Written out on curves, a morphism
in $\cQ$ is a map $f:\wh{X}/\langle \tau \rangle \to \wh{X'}/\langle \tau'
\rangle$, such that there exists a commutative diagram
\be \label{eq:Qmorphism}
\begin{tikzcd}
\wh{X} \ar[r, dashed, "\widetilde{f}"] \ar[d] & \wh{X'} \ar[d] \\
X = \wh{X} / \langle \tau \rangle \ar[r, "f"] &
X' = \wh{X'} / \langle \tau' \rangle \,.
\end{tikzcd}
\ee
If two such maps $\widetilde{f}$ exist, they differ by pre- or postcomposition with
an automorphism of~$\wh{X}$ resp.~$\wh{X}'$. Via the canonical cover construction,
the stack $\cQ$ is isomorphic to $\bP \komoduli[g,n](\mu)$. The non-uniqueness
of~$\widetilde{f}$ exhibits $\cH_k^\pmarked = \cQ / \langle \tau \rangle$ as the quotient
stack by a group of order~$k$, acting trivially.
\par
As in \autoref{sec:closurelinsection}, we denote by $\ol{\Omega \cH}_k := \ol{\Omega \cH_k}(\mu)$ the normalization
of the closure of $\Omega \cH_k$ in $ \Xi \ol \cM_{\wh{g},\wh{n}}(\mu)$ an let
$\ol{\cH}_k := \ol{\cH_k}(\mu)$ be the corresponding projectivizations.
We next describe the boundary strata of $\ol{\cH}_k$. These are indexed by
enhanced level graphs $\wh \Gamma$ together with an $\langle \tau \rangle$-action
on them. We will leave the group action implicit in our notation. The following
lemma describes the objects parametrized by the boundary components
$D_{\wh \Gamma}^{\cH_k}$ (using the notation from \autoref{sec:closurelinsection})
of the compactification $\ol{\cH}_k$.
\par
\begin{lemma} \label{le:bdHk}
A point in the interior of the boundary stratum $D_{\wh \Gamma}^{\cH_k}$ is
given by a tuple
\[
\{(\wh{X}, \wh \Gamma, \wh \bfz, [\bfomega], \bfsigma, \tau ) \,:\, \tau
\in \Aut(\wh{X}), \quad \ord(\tau)\=k, \quad
\tau^* \bfomega = \zeta_k\bfomega, {\quad \tau|_{\wh{\bfz}} = \tau_0} \}
\]
where $(\wh{X}, \wh \Gamma, \wh \bfz, [\bfomega], \bfsigma) \in
\bP\LMS[\whmu][\wh g,\wh n]$
is a multi-scale differential and where moreover the prong-matching $\bfsigma$ is
equivariant with respect to the action of $\langle \tau \rangle$.
\end{lemma}
\par
The equivariance of prong-matching requires an explanation: Suppose $x_i$
and~$y_i$ are standard coordinates near the node corresponding to an edge~$e$
of~$\Gamma$, so that the prong-matching at~$e$ is given by $\sigma_e =
\tfrac{\partial}{\partial x_i} \otimes -\tfrac{\partial}{\partial y_i}$
(compare \cite[Section~5]{LMS} for the relevant definitions). Then $\tau^*x_i$
and $\tau^* y_i$ are standard coordinates near $\tau(e)$. We say that
a global prong-matching $\bfsigma = \{\sigma_e\}_{e \in E(\wh{\Gamma})}$ is
\emph{equivariant} if $\sigma_{\tau(e)} =
\tfrac{\partial}{\partial \tau^*x_i} \otimes - \tfrac{\partial}{\partial \tau^*y_i}$
for each edge~$e$.
\par
\begin{proof}
The necessity of the conditions on the boundary points is obvious from
the definition in~\eqref{eq:point_Hk}, except for the prong-matching
equivariance. This follows from the construction of the induced prong-matching
in a degenerating family in \cite[Proposition~8.4]{LMS} and applying~$\tau$ to it.
\par
Conversely, given $(\wh{X}, \wh \Gamma, \wh \bfz, [\bfomega], \bfsigma,
\langle \tau \rangle)$ as above with equivariant prong-matchings, we need to
show that it is in the boundary of $\cH_k$. This is achieved precisely by the
equivariant plumbing construction given in \cite{kdiff}.
\end{proof}
\par
The group~${G}$ still acts on the
compactification $\ol{\Omega \cH_k}$ and on
its projectivization $\ol \cH_k$. As above we denote the quotient by
$\ol{\cH}^\pmarked_k = \ol \cH_k/G$ to indicate that the points~$\wh{\bfz}$
are now marked only partially. By \autoref{le:bdHk} we may construct~$\ol \cQ$
just as in the uncompactified case.
\par
The map $\ol{\cH}^\pmarked_k \to \ol\cQ$ is in general non-representable due to
the existence of additional automorphisms of objects in $\ol{\cH}_k^\pmarked$. This resembles
the situation common for Hurwitz spaces, where the target map is in general
non-representable, too.
We denote by $s : \ol\cH_k \to \ol{\cH}^\pmarked_k \to \cQ$ the composition
of the maps.
\par

%%%%%%%%%%%%%%%%%%%%%%%
\subsection{Generalized strata of $k$-differentials}
%%%%%%%%%%%%%%%%%%%%%%%

Our notion of generalized strata is designed for recursion purposes so that the
extraction of levels of a boundary stratum of $\ol\cQ$ is an instance of
a generalized stratum (of $k$-differentials). This involves incorporating
disconnected strata, differentials that are non-primitive on some components,
and residue conditions. Moreover, we aim for a definition of a space of
$k$-fold covers on which the group~$G$ acts, to match with the previous setup.
The key is to record which of the marked points is adjacent to which component of the canonical cover, an information that is obviously trivial in the case of primitive $k$-differentials.
\par
A map $\cA: \wh\bfz \to \pi_0(\wh{X})$ that records which marked point is
adjacent to which component of $\wh{X}$ is called an \emph{adjacency datum}.
Such an adjacency datum is equivalent to specifying a one-level graph of
a generalized stratum, which is indeed the information we get when we extract
level strata.  Note in particular that from the adjacency datum it is possible to reconstruct  the unique $u$ such that the Abelian differentials on $(\wh{X},\wh\bfz)$ are $u$-th power of  primitive $k/u$ differentials.
\par
More abstractly, an adjacency datum is given by a set $\pi_0$ with a transitive action of $\bZ/k\bZ$ together with a map $\cA : \wh\bfz\to \pi_0$ that is equivariant with respect to the action of $\bZ/k\bZ$. We say that $(\wh{X},  \wh\bfz)$ has adjacency $\cA$ if there is a $\bZ/k\bZ$ -equivariant bijection $\pi_0 \cong \pi_0(\wh{X})$ such that $\cA$ records the adjacency of the markings $ \wh\bfz$ in the components of $\wh{X}$.
The subgroup~$G$ from~\eqref{eq:defG} acts on the triples $(\wh X, \wh \bfz,
\cA)$ of pointed stable curves with adjacency map by acting simultaneously
on $\wh \bfz$ and on $\cA$ by precomposition. For a fixed adjacency datum~$\cA$
we consider the stack $\Omega \wt \cH_k(\wh{\mu}, \cA)$
whose points are
\begin{align*}
\{(\wh{X},\wh{\bfz},\omega, \tau) \,:\, \, \text{$(\wh{X},\wh{\bfz})$ have
adjacency~$\cA$}, \,\,  \tau \in \Aut(\wh{X}), \\
\quad \ord(\tau)\=k, \quad
\tau^* \omega \= &\zeta_k \omega, \quad \tau|_{\wh{\bfz}} = \tau_0, %\\
%& \cA \text{ is an association map}
\} \,.
\end{align*}
We denote by $\Omega \cH_k(\wh{\mu}, [\cA]) := G \cdot \Omega \wt \cH_k(\wh \mu, \cA)$ the $G$-orbit of this space.  
\par
A \emph{residue condition} is given by a $\tau$-invariant partition~$\lambda_\frakR$ of a subset of
the set $H_p \subseteq \{1, \dots, \wh n\}$ of marked points such
that $\widehat{m}_i < -1$. We often also call the associated linear subspace
\[
	\frakR := \left\{(r_i)_{i \in H_p} \in \bC^{H_p} \;:\; \sum_{i \in\lambda} r_i = 0 \text{ for all } \lambda \in \lambda_\frakR\right\}.
\]
the residue condition. The linear subspace $\frakR$ is obviously  not $G$-invariant in general.
\par
We denote by $\Omega \cH_k^{\frakR}(\wh \mu, \cA) \subseteq
\Omega \cH_k(\wh \mu, \cA)$ the subset where for each $R \in \frakR$ the residues
of $\wh \omega$ at all the points $z_i \in R$ add up to zero. If $(\wh X, \wh \bfz,
\omega, \tau)$ is contained in $\Omega \cH_k^{\frakR}(\wh \mu, \cA)$, then
$g \cdot (\wh X, \wh \bfz, \omega, \tau)$ is contained in $\Omega
\cH_k^{g \cdot \frakR}(\wh \mu, g \cdot \cA)$ for any $g \in G$. That is, the $G$-action
simultaneously changes the residue condition and the adjacency datum. We
denote by $[\frakR,\cA]$ the $G$-orbit of this pair and use the abbreviation
\be \label{eq:defOmHRA}
\Omega \cH_k^{[\frakR,\cA]} %:= \Omega \cH_k^\frakR(\wh \mu, \cA)
\,:= \,G \cdot \Omega  \cH_k^{\frakR}(\wh \mu, \cA)
\ee
for the $G$-orbit of the spaces, $\wh\mu$ being tacitly fixed throughout.
\par
As above, we denote by $\cH_k^{[\frakR,\cA]}$ the projectivization of $\Omega
\cH_k^{[\frakR,\cA]}$ and by $\cH_k^{\frakR,\pmarked} := \cH_k^{[\frakR,\cA]} / G$ the $G$-quotient,
dropping the information about adjacency and the connected components to ease
notation. Finally, we denote by $\cQ^\frakR$  the stack with the same underlying set as
$\cH_k^{\frakR,\pmarked}$ and with morphisms defined in the same way as above for~$\cQ$.
Recall that the curves in $\cQ^\frakR$ may be disconnected. We call such a stratum
with possibly disconnected curves and residue conditions a \emph{generalized
stratum of $k$-differentials}. Since $\cH_k^{[\frakR,\cA]}$ is a linear
submanifold, we can still compactify them as before and a version
of \autoref{le:bdHk} with adjacency data still holds.
\par
We will now compute the degree of the map $s$ from the linear submanifolds to the strata of $k$-differential. Our definition of generalized strata of $k$-differentials makes the degree of this map the same in the usual and in the generalized case.
  
\begin{lemma} \label{lem:deg_d}
	The map $s: \ol \cH^{[\frakR,\cA]}_k \to \ol \cQ^\frakR$ is proper, quasi-finite, unramified and of degree
	\[
	\deg(s) \= \frac{1}{k} \prod_{m_i \in \mu} \gcd(m_i, k)\,.
	\]
\end{lemma}
\par
\begin{proof}
The composition of the map $s$ with the quotient map $\ol \cQ^\frakR  \to \ol \cH^{[\frakR,\cA]}_k$ by the trivial action of $\bZ/k\bZ$ is  the quotient by a group of order $|G| = \prod_{m_i \in \mu} \gcd(m_i, k)$. Since degrees are multiplicative under compositions, the claimed formula for $\deg(s)$ follows.
\par
{The map is unramified as both quotient maps are unramified.}
\end{proof}
\par

%%%%%%%%%%%%%%%%%%%%%%%
\subsection{Decomposing boundary strata }
%%%%%%%%%%%%%%%%%%%%%%%

Having constructed strata of $k$-differentials, we now want to decompose their
boundary strata again as a product of generalized strata of $k$-differentials and
argue recursively.  In fact, the initial stratum should be  a generalized stratum
$\ol \cQ^\frakR$, thus coming with its own residue condition, but we suppress
this in our notation, focusing on the new residue condition that arise
when decomposing boundary strata. 
Here 'decomposition' of the boundary strata
should be read as a construction of a space finitely covering both of
them, as given by the following diagram,
\be \label{dia:k-comens}
\begin{tikzcd}[column sep=small]
&& D_{\pi}^{\circ,\cH_k,s} \ar[dl,
"p_{\pi}"'] \ar[dr, "c_{\pi}"] & \\
\cH_{k}(\pi):=\prod_{i=0}^{-L} \cH_{k}(\pi_{[i]})
\ar[d, "\bs_{\pi}"']  & \mathrm{Im}(p_\pi) \ar[l,"\supseteq"]&&
D_{\pi}^{\circ,\cH_k} \ar[d, "d_{\pi}"] \\
\cQ(\pi):=\prod_{i=0}^{-L} \cQ(\pi_{[i]})
&&& D_{\pi}^{\circ,\cQ}\,,
\end{tikzcd}
\ee
whose notation we now start to explain. Note that the diagram is
for the open boundary strata throughout, since we mainly need
the degree all these maps as in \autoref{lem:ratio_degrees} (the
existence of a similar diagram over the completions follows as
at the beginning of \autoref{sec:boundarycomb}).
\par
We denote by $\wh{\Gamma}$ the level graphs indexing the boundary strata of
$\bP\LMS[\whmu][\wh g,\wh n]$ and thus of $\ol{\cH}_k$.
Following our general convention for strata their legs are labeled, but
not the edges. In $\ol\cH_k^\pmarked$ the leg-marking is only
well-defined up to the action of~$G$. A graph with such a marking
is said to be  \emph{marked (only) partially} and denoted by $\wh \Gamma_\pmarked$.
Even though curves in $\ol{\cH}_k$ are marked (and not only marked up to the
action of $G$), the boundary strata of $\ol{\cH}_k$ are naturally indexed by
partially marked graphs as well: If $\wh \Gamma$ is the dual graph of one stable
curve in the boundary of $\ol{\cH}_k$, then for all $g \in G$ the graph
$g \cdot \wh \Gamma$ is the dual graph of another stable curve in the boundary
of $\ol{\cH}_k$. The existence of~$\tau$ implies that level graphs $\wh{\Gamma}$ at
the boundary of~$\ol{\cH}_k$ come with the quotient map by this action.
To each boundary stratum of~$\ol\cQ$ we may thus associate a $k$-cyclic
covering of graphs $\pi : \wh \Gamma_\pmarked \to \Gamma$ (see
\cite[Section~2]{CoMoZa} for the definitions of such covers). We denote the
corresponding (open) boundary strata by $D_\pi^{\circ,\cQ} \subset \ol\cQ$
and the (open) boundary strata corresponding to such a $G$-orbit of graphs
by $D_{\pi}^{\circ,\cH_k} \subset \ol\cH_k$. The map $d_\pi : D_{\pi}^{\circ,\cH_k} \to D_\pi^{\circ,\cQ}$ is
the restriction of the map $s : \ol\cH_k \to \ol\cQ$.
\par
Next we construct the commensurability roof just as in~\eqref{dia:covboundary2},
though for each $\wh{\Gamma}$ in the $G$-orbit separately, so that $
D_{\pi}^{\circ,\cH_k,s}$ is the disjoint union of a $G$-orbit of the roofs
in~\eqref{dia:covboundary2}.
\par
Next we define the spaces $\cH_{k}(\pi_{[i]})$. Consider the linear submanifolds of
generalized strata of $k$-differentials with signature and adjacency datum given
by the $i$-th level of one marked representative $\wh \Gamma$ of $\wh
\Gamma_{\pmarked}$ (the resulting strata are independent of the choice of a
representative). Their product defines the image $\Im(p_\pi)$. For every level~$i$, consider the orbit under $G(\cH_k(\pi_{[i]}))$, where $G(\cH_k(\pi_{[i]}))$ is the group as in \eqref{eq:defG} for the $i$-th level, of the linear submanifolds we extracted from the levels. We define  $\cH_{k}(\pi_{[i]})$ to be these orbits, which in particular are then linear submanifolds associated to generalized strata of $k$-differentials as we defined them above. We can hence consider, for every level, the morphism given by the quotient by $G(\cH_k(\pi_{[i]}))$ composed  with the rigidification  by the action of  $\langle\tau\rangle$ at each level and denote by $\cQ(\pi_{[i]})$ its image, which is called the  generalized stratum of $k$-differentials at level~$i$.
The map $\bs_\pi$ in diagram~\ref{dia:k-comens} is just a product of maps
like the map~$s$ above, thus \autoref{lem:deg_d} immediately implies:
\par
\begin{lemma} \label{lem:deg_q}
The degree of the map $\bs_\pi$ in the above diagram~\eqref{dia:k-comens} is
\[ \deg(\bs_\pi) \= \frac{1}{k^{L+1}} \prod_{i=1}^n \gcd(m_i, k)
\prod_{e \in E(\Gamma)} \gcd(\kappa_e, k)^2
\]
where $\kappa_e$ is the $k$-enhancement of the edge~$e$.
\end{lemma}
\par
\par
We recall \autoref{lem:ratio_degrees} and compute explicitly the
coefficients appearing in our setting here. Note that the factor
$|\Aut_{\cH}(\Gamma)|$ there should be called $|\Aut_{\cH_k}(\wh{\Gamma})|$ in
the notation used in this section.
\par
\begin{lemma} \label{le:newratdeg}
The ratio of the degrees of the topmost maps in~\eqref{dia:k-comens} is
\bes
\frac{\deg(p_{\pi})}
{\deg(c_{\pi})}
\= \frac{K^{\cH_k}_{\wh \Gamma} } {|\Aut_{\cH_k}(\wh\Gamma)|  \cdot {\ell_{\wh \Gamma}}}
\ees
where the number of reachable prong-matchings is given by
\[
K_{\wh \Gamma}^{\cH_k} \= \prod_{e \in E(\Gamma)} \frac{\kappa_e}{\gcd(\kappa_e, k)}
\] and $\Aut_{\cH_k}(\wh\Gamma)$ is the subgroup of automorphisms of $\wh\Gamma$
commuting with $\tau$.
\end{lemma}
\par
We remark that the quantity $\ell_{\wh \Gamma}$ is intrinsic to $\Gamma$, for a two-level graph it is given by
 $\ell_{\wh \Gamma} = \lcm\bigl(\frac{\kappa_e}
{\gcd(\kappa_e,k)} \text{ for } e \in E(\Gamma)\bigr)$.
\begin{proof}
The first statement is exactly the one of \autoref{lem:ratio_degrees} since the
topmost maps in~\eqref{dia:k-comens} are given by a disjoint union of the topmost
maps in~\eqref{dia:covboundary2}. 
\par
For the second statement, consider an edge $e \in E(\Gamma)$. The edge $e$ has
$\gcd(\kappa_e, k)$ preimages, each with an enhancement
$\frac{\kappa_e}{\gcd(\kappa_e,k)}$. The prong-matching at one of the preimages
determines the prong-matching at the other preimages by \autoref{le:bdHk},
as they are related by the action of the automorphism.
\par
For the third statement, we need to prove that the subgroup of $\Aut(\wh\Gamma)$
fixing setwise the linear subvariety~$\ol\cH_k$ is precisely the subgroup commuting
with~$\tau$. If $\rho \in \Aut(\wh\Gamma)$ commutes with~$\tau$, then it descends
to a graph automorphism of~$\Gamma$ and gives an automorphism of families of
admissible covers of stable curves, thus preserving ~$\ol\cH_k$. Conversely, if
$\rho$ fixes~$\ol\cH_k$, it induces an automorphism of families of
admissible covers of stable curves, thus of coverings of graphs. A priori this
implies only that~$\rho$ normalizes the subgroup generated by~$\tau$. Note however
that on~$\ol\cH_k$ the automorphism~$\tau$ acts by a fixed root of unity~$\zeta_k$.
If $\rho \tau \rho^{-1}$ is a non-trivial power of~$\tau$, this leads to another
(though isomorphic) linear subvariety. We conclude that~$\rho$ indeed commutes
with~$\tau$.
\end{proof}
\par
The aim of the following paragraphs is to rewrite the evaluation
\autoref{lem:evaluation} in our context in order to find the shape of the formula
in \autoref{cor:kstrata}. We elaborate on basic definitions to distinguish
notions of isomorphisms and automorphisms.
The underlying graph of an enhanced (k-)level graph can be written as a tuple
$\Gamma = (V, H, L, a : H \cup L \to V, i: H \to H)$, where $V$, $H$ and $L$ are the
sets of vertices, half-edges and legs, $a$ is the attachment map and $i$ is the
fixpoint free involution that specifies the edges.
An isomorphism of graphs $\sigma: \Gamma \to \Gamma'$ is a pair of bijections
$\sigma = (\sigma_V : V \to V', \sigma_H: H \to H')$  that preserve the attachment
of the half-edges and legs and the the identification of the half-edges to edges,
i.e. the diagrams
\be \label{dia:graph_auto}
\begin{tikzcd}
	H \cup L \ar[r, "a"] \ar[d, "\sigma_H \cup \id_L"] & V \ar[d, "\sigma_V"] &[2cm] H \ar[r, "i"] \ar[d, "\sigma_H"] & H \ar[d, "\sigma_H"] \\
	H' \cup L \ar[r, "a'"] & V' & H' \ar[r, "i'"] & H'
\end{tikzcd}
\ee
commute. If the graph is an enhanced level graph, we additionally ask that $\sigma$ preserves the enhancements and level structure. In the presence of a deck transformation $\tau$, we moreover ask that $\sigma$ commutes with $\tau$.
\par
In the sequel we will encounter isomorphisms of graphs with the same underlying sets of vertices and half-edges. We emphasize that in this case an isomorphism~$\sigma$ is an \emph{automorphism} if and only if it preserves the maps $a$ and $i$, i.e.\ if
\be \label{eq:graph_auto}
\sigma_V^{-1} \circ a \circ (\sigma_H \cup \id_L) = a \qquad \text{and} \qquad \sigma_H^{-1} \circ i \circ \sigma_H = i.
\ee
\par
We now define the group of level-wise half-edge permutations compatible
with the cycles of~$\tau$, i.e., we let
\[
\bfG \,:=\ \bfG_{\pi}
\= \prod_{i=0}^{-L} G(\cH_k(\pi_{[i]})),
\]
where $G(\cH_k(\pi_{[i]}))$ is the group $G$ from~\eqref{eq:defG} applied
to the $i$-th level stratum. An element of the group $\bfG$ is a permutation $g : H \cup L \to H \cup L$ and acts on a graph~$\wh \Gamma$ via $g \cdot \wh \Gamma
= (V, H, L, a \circ g, i)$.
\par
There is a natural action of the group $\bfG$
on the set of all (possibly disconnected) graphs with the same set of
underlying vertices as $\wh\Gamma_\pmarked$.
We denote by 
\begin{equation}\label{def:stab}
\Stab_\bfG(\wh \Gamma) := \{ g \in \bfG\colon g \wh\Gamma \cong \wh\Gamma\}
\end{equation}
the stabilizer. Note that this is in general not a group, as it is not the stabilizer of an element but of an isomorphism class. We also denote by 
$\Stab_{\bf G}(\cH(\pi))$ the set of elements of $\bf G$ which fix the adjacency data (or equivalently the $1$-level graphs) of the level-wise linear manifolds $\cH(\pi_{[i]})$, i.e., elements which  permute vertices with the same signature and permute  legs of the same order on the same vertex.
\par
\begin{lemma} \label{le:stab}
	We have
	$$|\Aut_{\cH_k}(\wh \Gamma)| \cdot |\Stab_\bfG(\wh \Gamma)| \=  |\Aut(\Gamma)| \prod_{e \in E(\Gamma)}
	\gcd(\kappa_e,k)\cdot |\Stab_{\bf G}(\cH(\pi)) | $$
\end{lemma}
\par
\begin{proof}
Fix a cover $\wh \Gamma \to \Gamma$. We may assume that the vertices of $\Gamma$ are $\{1, \dots, v_\Gamma\}$, the legs are $\{1, \dots, n\}$ and the half-edges are $\{1^\pm, \dots, h_\Gamma^\pm\}$ with the convention that $i(h^\pm) = h^\mp$. For $\wh \Gamma$, we may assume that the preimages of vertex $v$ are $(v, 1), \dots, (v, p_v)$ such that $\tau((v, q)) = (v, q+1)$, where equality in the second entry is to be read mod $p_v$.
	Similarly, we index the legs of $\wh \Gamma$ by tuples $(m,1),\dots,(m,p_{m})$ for $m=1,\dots,n$, and the half-edges by tuples $(h^\pm,1),\dots,(h^\pm,p_{h^\pm})$ for $h^\pm=1,\dots,h^\pm_\Gamma$, again such that $(h^+,q)$ and $(h^-,q)$ form an edge.
	
	We consider the group $\cP$ of pairs of permutations $\sigma = (\sigma_V, \sigma_H)$ of the vertices and half-edges of $\wh \Gamma$ that are of the following form:
	There exists a $\gamma = (\gamma_V, \gamma_H) \in \Aut(\Gamma)$, integers $\lambda_v\in \bZ/p_v\bZ$ for any $v\in V(\Gamma)$ and integers $\mu_{h^\pm}\in \bZ/p_{h^\pm}\bZ$ for any $h^\pm\in E(\Gamma)$ such that
	\bes
	\sigma_V = \{(v, q) \mapsto (\gamma_V(v), q + \lambda_v)\} \qquad \text{and} \qquad
	\sigma_H = \{(h^\pm, q) \mapsto (\gamma_H(h^\pm), q + \mu_{h^\pm})\}.
	\ees
	We let this group act on $\wh \Gamma$ via $\sigma\cdot \wh \Gamma = (V, H, L, \sigma_V^{-1} \circ a \circ (\sigma_H \cup \id_L), i)$.
	An element $\sigma \in \cP$ acts always as an isomorphism since the
        diagrams~\eqref{dia:graph_auto} commute.
	If we denote by $e$ the edge given by $h^\pm$, we have $p_{h^\pm}= \gcd(\kappa_e, k)$. Hence  the group $\cP$ has cardinality
	\bes
	|\cP| \= |\Aut(\Gamma)| \cdot \prod_{e \in E(\Gamma)} \gcd(\kappa_e, k) \cdot \prod_{v \in V(\Gamma)} p_v.
	\ees
\par	
	Recall that the group $\bfG$ is a product cyclic groups and thus Abelian.
	The stabilizer $\Stab_\bfG(\cH_k(\pi))$ has a subgroup $\Stab^f$ where only half-edges and legs attached to the same vertex are permuted (the superscript $f$ is for \emph{fixed}), i.e.\ the elements $g \in \Stab^f$ are exactly those for which $a \circ g = a$.
	The quotient $\Stab^p := \Stab_\bfG(\cH_k(\pi)) / \Stab^f$ can be identified with those elements of $\bfG$ that permute legs and half-edges in such a way that whenever a leg or half-edge attached to a vertex $v_1$ is moved to another vertex $v_2$, then all the legs and half-edges attached to $v_1$ are moved to $v_2$.
	So we may alternatively identify $\Stab^p$ with $\tau$-invariant permutations of the vertices of $\wh \Gamma$ (hence the superscript $p$ for \emph{permutation}).
	This yields $|\Stab^p| = \prod_{v \in V(\Gamma)} p_v$.
\par	
	The group $\cP$ comes with a commutative triangle
	\bes
	\begin{tikzcd}
		\Aut_\cH(\wh \Gamma) \ar[r, hookrightarrow] \ar[rd] & \cP \ar[d, two heads] \\
		& \Aut(\Gamma)
	\end{tikzcd}
	\ees
	where the vertical map is the forgetful map, the diagonal map is the quotient by $G$-map and the horizontal map is natural injection. 
	Since we computed above $|\cP|$, we know that the kernel of the surjective map $\cP \to \Aut(\Gamma)$ has cardinality $\prod_{e \in E(\Gamma)} \gcd(\kappa_e, k) \cdot \prod_{v \in V(\Gamma)} p_v$.
	
	Note now that the group $\Stab^f$ acts  on the set $\Stab_\bfG(\wh\Gamma)$ and we denote by $\Stab_\bfG(\wh\Gamma)/\Stab^f$ the space of orbits.
	We are done if we can identify elements of $\Stab_\bfG(\wh\Gamma)/\Stab^f$ with elements of the cosets in $\cP / \Aut_\cH(\wh \Gamma)$.
	
	For this identification, first consider $g \in \Stab_\bfG(\wh\Gamma)$.
	By definition, there exists an isomorphism $\sigma(g) : g \cdot \wh \Gamma \to \wh \Gamma$ such that $g \cdot \wh\Gamma = \sigma(g)(\wh \Gamma)$. This induces a map $\sigma:\Stab_\bfG(\wh\Gamma)\to \cP$.
	Note that $\Stab^f$ is a subgroup of $\Aut_\cH(\wh \Gamma)$.
	If we had chosen a different representative $g'$ in the orbit $ g\cdot \Stab^f$, the resulting element $\sigma(g') \in \cP$ would differ by an element of $\Aut_\cH(\wh\Gamma)$. Hence $\sigma$ induces a well-defined map $\Stab_\bfG(\wh\Gamma)/\Stab^f\to\cP / \Aut_\cH(\wh \Gamma)$.
	We now construct an inverse map for $\sigma$. For any $\rho \in \cP$,
	we need to find an element $g \in \bfG$ such that $\sigma(g) = \rho$, i.e.\ such that $g \cdot \wh \Gamma = \rho(\wh\Gamma)$. 
	This implies that $g$ must satisfy the equation
	\[
	a \circ g \= \rho_V^{-1} \circ a \circ (\rho_H \cup \id_L),
	\]
	which determines the element~$g$ up to the action of $\Stab^f$.
	The resulting $g$ does not depend on the choice of a representative of the coset $\rho / \Aut_\cH(\wh\Gamma)$ because of~\eqref{eq:graph_auto}.
\end{proof}
We let now
\be \label{eq:defSpi}
S(\pi) \= \frac{|G|}{|\bfG|} \cdot\frac{|\Stab_\bfG(\wh \Gamma)|}{|\Stab_G(\wh \Gamma)|}
\=\frac{|\Stab_{\bfG/G}(\wh \Gamma)|}{\prod_{e} \gcd(\kappa_e, k)^2}
\ee
where the stabilizers are defined in a way analogous to \eqref{def:stab}.
\par
\begin{rem} \label{rem:Spi1}
We have the ratio $S(\pi) = 1$ for many coverings of graphs $\pi: \wh\Gamma \to \Gamma$, e.g. when all vertices of $\Gamma$ have exactly one preimage in $\wh\Gamma$. In this case $\bfG/G$ only permutes half-edges adjacent to one vertex, and this always stabilizes the graph. Thus $S(\pi) = 1$, as $|\bfG/G| = \prod_e\gcd(\kappa_e,k)^2$. More generally $S(\pi) =1$ if each edge of $\Gamma$ is adjacent to at least one vertex which has exactly one preimage in~$\wh\Gamma$. In this case it is straightforward to verify that the obvious generators of $\bfG / G$ are stabilizing the graph.
\par
If there are vertices of $\Gamma$ with more than one pre-image in $\wh\Gamma$,
then $S(\pi)$ is in general non-trivial. Consider for example the covering of graphs $\pi$ depicted in Figure~\ref{cap:non-triv_Spi1}, for
which $S(\pi) = \tfrac{1}{2}$.
\end{rem}
\begin{figure}
\bes
\begin{tikzpicture}[
baseline={([yshift=-.5ex]current bounding box.center)},
scale=2,very thick,
bend angle=30,
every loop/.style={very thick},
%comp/.style={circle,fill,black,,inner sep=0pt,minimum size=5pt},
comp/.style={circle,black,draw,thin,inner sep=0pt,minimum size=5pt,font=\tiny},
order bottom left/.style={pos=.05,left,font=\tiny},
order top left/.style={pos=.9,left,font=\tiny},
order bottom right/.style={pos=.05,right,font=\tiny},
order top right/.style={pos=.9,right,font=\tiny},
order node dis/.style={text width=.75cm}]
\node[circle, draw, inner sep=0pt, minimum size=12pt] (T1) []{$1$};
\node[circle, draw, inner sep=0pt, minimum size=12pt] (T2) [right=1cm of T1]{$1$};
\node[circle, draw, inner sep=0pt, minimum size=12pt] (B1) [below=1cm of T1]{$1$} 
edge 
node [order bottom left] {$-2$}
node [order top left] {$0$} (T1)
edge 
node [order bottom right] {$-2$}
node [order top left] {$0$} (T2); 
\node[circle, draw, inner sep=0pt, minimum size=12pt] (B2) [below=1cm of T2]{$1$}
edge 
node [order bottom left] {$-2$}
node [order top right] {$0$} (T1)
edge 
node [order bottom right] {$-2$}
node [order top right] {$0$} (T2); 
\node [minimum width=18pt,below left] (B-1) at (B1.south west) {$4$};
\path (B1) edge [shorten >=5pt] (B-1.center);
\node [minimum width=18pt,below right] (B-2) at (B2.south east) {$4$};
\path (B2) edge [shorten >=5pt] (B-2.center);
\end{tikzpicture}
\xrightarrow{\quad\pi\quad}
\begin{tikzpicture}[
baseline={([yshift=-.5ex]current bounding box.center)},
scale=2,very thick,
bend angle=30,
every loop/.style={very thick},
%comp/.style={circle,fill,black,,inner sep=0pt,minimum size=5pt},
comp/.style={circle,black,draw,thin,inner sep=0pt,minimum size=5pt,font=\tiny},
order bottom left/.style={pos=.05,left,font=\tiny},
order top left/.style={pos=.9,left,font=\tiny},
order bottom right/.style={pos=.05,right,font=\tiny},
order top right/.style={pos=.9,right,font=\tiny},
order node dis/.style={text width=.75cm}]
\node[circle, draw, inner sep=0pt, minimum size=12pt] (T1) []{$1$};
\node[circle, draw, inner sep=0pt, minimum size=12pt] (B1) [below=1cm of T1]{$1$} 
edge [bend left]
node [order bottom left] {$-4$}
node [order top left] {$0$} (T1)
edge [bend right]
node [order bottom right] {$-4$}
node [order top right] {$0$} (T1); 
\node [minimum width=18pt,below left] (B-1) at (B1.south west) {$8$};
\path (B1) edge [shorten >=5pt] (B-1.center);
\end{tikzpicture}
\ees
\caption{A covering of graphs $\pi : \wh\Gamma \to \Gamma$ in~$\kLMS[8][3,1][2]$ with non-trivial~$S(\pi)$.}
\label{cap:non-triv_Spi1}
\end{figure}
\par
As a consequence of the degree computation in \autoref{le:newratdeg} and \autoref{le:stab}, we can write an evaluation lemma for $k$-differentials analogous to \autoref{lem:evaluation}. We give two versions, for~$\cH_k$ and~$\cQ$ respectively.
\par
\begin{lemma} \label{le:top_degree_comp_new2}
Let $(\pi:\wh\Gamma_\pmarked \to \Gamma) \in \LG_L(\cH^\pmarked_k)$ and $\wh\Gamma$ a marked version of $\wh\Gamma_\pmarked$. Suppose that $\alpha_{\pi} \in \CH_0(D_{\pi}^{\cH_{k}})$ and $\beta_{\pi} \in \CH_0(D_{\pi}^\cQ)$ are top degree classes and that
\[
c_{\pi}^* \alpha_{\pi} \= p_{\pi}^* 
\prod_{i=0}^{-L} \alpha_{i}
\qquad \text{and} \qquad
c_{\pi}^*d_\pi^* \beta_{\pi} \= p_{\pi}^*\bs_\pi^* 
\prod_{i=0}^{-L} \beta_{i}	\]
for some $\alpha_i$ and $\beta_i$. Then
\[
\int_{D_{\pi}^{\cH_{k}}} \alpha_{\pi} \= S(\pi)\cdot \frac{\prod_{e \in E(\Gamma)} \kappa_e}
{|\Aut(\Gamma)| \cdot \prod_{e \in E(\Gamma)} \gcd(\kappa_e,k)^2
	\cdot {\ell_{\wh \Gamma}}}  \cdot \prod_{i=0}^{-L}
\int_{\cH_{k}(\pi_{[i]})} \alpha_i
\]
and
\[
\int_{D_{\pi}^{\cQ}} \beta_{\pi} \= S(\pi)\cdot \frac{\prod_{e \in E(\Gamma)} \kappa_e}
{k^L \cdot |\Aut(\Gamma)| \cdot {\ell_{\wh \Gamma}}}  \cdot \prod_{i=0}^{-L}
\int_{\cQ(\pi_{[i]})} \beta_i.
\]
\end{lemma}
\begin{proof}
In order to show the first statement, we first apply  \autoref{le:newratdeg} and note that the map $p_{\pi}$ is not surjective in general. It is now enough to check that the number of  
 of adjacency data appearing in $\cH_{k}(\pi)$ is $ |\mathbf{G}|/|\Stab_{\bf G}\big(\cH_k(\pi)\big) |$, while the one appearing in the image of $p_{\pi}$ is  $| G|/|\Stab_{ G}\wh \Gamma |$.  We finally use \autoref{le:stab} to rewrite the prefactor.
For the second statement, we additionally apply \autoref{lem:deg_d} and \autoref{lem:deg_q}.
\end{proof}
\par
We are finally ready to prove \autoref{cor:kstrata}.
\par
\begin{proof}[Proof of \autoref{cor:kstrata}]
The orbifold Euler characteristics of $\cQ = \bP \Omega^k \cM_{g,n}(\mu)$
and $\cH_k$ are related by
\bes
\chi(\bP \Omega^k \cM_{g,n}(\mu)) \= \frac{1}{\deg(s)} \cdot \chi(\cH_k).
\ees
We apply the general Euler characteristic formula in the form~\eqref{eq:eulerH} to
$\cH_k$ and group the level graphs $\wh{\Gamma} \in \LG_L(\cH_k)$ by those with
the same  graph~$\wh{\Gamma}_{\pmarked}$ that is marked partially. Since the integrals
do not depend on the marking, we obtain 
\bes
\chi(\cQ) \= \frac{k}{|G|} (-1)^d
\sum_{L=0}^d \sum_{(\pi:\wh\Gamma_\pmarked \to \Gamma)  \in \LG_L(\cH_k^{\pmarked})} N_{\pi}^\top \cdot \ell_{\wh\Gamma} \cdot
\int_{D_{\pi}^{\cH_k}} \prod_{i=-L}^{0}(\xi_{\wh\Gamma,\cH_k}^{[i]})^{d_\Gamma^{[i]}}
\ees
where we used the notation that $\wh\Gamma$ is a fully marked representative of $\wh\Gamma_\pmarked$.
Thanks to \autoref{lem:compatibilities} we can apply \autoref{le:top_degree_comp_new2}
and convert the integral over $D_{\pi}^{\cH_k}$ into a $\xi$-integral over
the product of $\cH_k(\pi_{[i]})$. We hence obtain
\bas
&\chi(\bP \okmoduli[g,n](\mu))\\
%%%
&\;=\frac{k}{|G|} \cdot
\left(-1\right)^d \sum_{L=0}^d
\sum_{(\pi:\wh\Gamma_\pmarked \to \Gamma) \in \LG_L(\cH^\pmarked_k)}
S(\pi)\frac{\prod_{e \in E(\Gamma)} \kappa_e \cdot N_{\pi}^\top}
{|\Aut(\Gamma)| \cdot \prod_{e} \gcd(\kappa_e, k)^2}
\cdot \prod_{i=0}^{-L}
\int_{\cH_{k}(\pi_{[i]})} \xi^{d_{\pi}^{[i]}} \\
%%%
&\;=\left(\frac{-1}{k}\right)^d \sum_{L=0}^d
\sum_{(\pi:\wh\Gamma_\pmarked \to \Gamma) \in \LG_L(\cQ)} S(\pi)\cdot
\frac{\prod_{e \in E(\Gamma)} \kappa_e \cdot N_{\pi}^\top}
{|\Aut(\Gamma)| } \cdot \prod_{i=0}^{-L}
\int_{\cQ(\pi_{[i]})} \zeta^{d_{\pi}^{[i]}}.
\eas
For the second equality, we used that
\be \label{eq:xi_zeta_comparision}
s^* \zeta \= k\xi\,, \quad \text{and hence} \quad
d_*\xi \= \frac{\deg(s)}{k} \zeta
\ee
for any level stratum, together with the dimension statement of
\autoref{prop:linatboundary}. The final result is what we claimed in
\autoref{cor:kstrata}.
\end{proof}
\par

%%%%%%%%%%%%%%%%%%%%%%%
\subsection{Evaluating tautological classes}
%%%%%%%%%%%%%%%%%%%%%%%
In this section we explain how to evaluate any top degree class of the form
\begin{equation}\label{eq:beta}
\beta \,:=\, \zeta^{p_0} \psi_1^{p_1} \cdots \psi_n^{p_n}
\cdots [D_{\pi_1}^\cQ] \cdots [D_{\pi_w}^\cQ] \in \CH_0(\ol\cQ)
\end{equation}
for any generalized stratum $\ol\cQ$ of $k$-differentials.
First, we show how to transform the previous class into the form
\[
\beta \= \sum_{i} \psi_1^{q_{i,1}} \cdots \psi_1^{q_{i,n}} [D_{\sigma_i}^\cQ].
\]
Then by \autoref{le:top_degree_comp_new2}, we can write every summand of $\beta$ as a product of $\psi$-classes evaluated on generalized strata of $k$-differentials. We finally will explain how to evaluate such classes.
\par
Let us start with the first task. The relations in the Chow ring of a general linear submanifold we obtained in \autoref{sec:nb} immediately apply to the covering $\ol\cH_k$ and we want to restate them in the Chow ring of the generalized stratum $\ol\cQ$ of $k$-differentials.
Let $i$ be the index of a marked point in $\ol\cQ$ and $(i,j)$ be the index of a preimage of this point in $\ol\cH_k$.
Moreover, let $m_i$ denote the order of the $k$-differential at the $i$-th marked point, and let $\wh m_{i,j}$ denote the order of the Abelian covering at the $(i,j)$-th marked point.
Then the relation
\be \label{eq:psicomp}
\psi_{i,j} \= \frac{\gcd(m_i,k)}{k} \cdot d^* \psi_i
\ee
holds, see for example~\cite[Lemma~3.9]{SvZ}.
Using the relation $$\wh m_{i,j} + 1= (m_i+k) / \gcd(m_i,k)$$ and applying push-pull we obtain
\be \label{eq:psi_comparision}
(\wh m_{i,j} + 1)d_*\psi_{i,j} \= \frac{\deg(d)}{k} (m_i + k) \psi_i.
\ee
\par
We can now write the analogue of \autoref{prop:Adrienrel} for the first Chern class $\zeta\in \CH^1(\ol{Q})$ of the tautological line bundle on the stratum of $k$-differentials.
\begin{cor} \label{cor:k-adrien}
	The class $\zeta$ can be expressed as
	\begin{align*}
		\zeta &= (m_i + k) \psi_i
		- \sum_{(\pi : \wh \Gamma_\pmarked \to \Gamma) \in \prescript{}{i}\LG_1(\ol\cQ)} k \ell_{\wh \Gamma_\pmarked}
		[D^\cQ_\pi] \\
		&= (m_i + k) \psi_i
		- \sum_{(\pi : \wh \Gamma_\pmarked \to \Gamma) \in \prescript{}{i}\LG_1(\ol\cQ)} S(\pi)\frac{\prod_{e \in E(\Gamma)} \kappa_e}{|\Aut(\Gamma)|}
		\cl_{\pi,*}p_\pi^*\bs_{\pi}^*[\cQ(\pi)]
	\end{align*}
	where $\prescript{}{i}\LG_1(\ol\cQ)$ are covers of two-level graphs with the
	leg~$i$ on lower level and $\cl_{\pi}=\i_\pi\circ d_\pi\circ c_\pi $ is the clutching morphism analogous to \eqref{eq:addgenR}.
\end{cor}
\begin{proof}
The first equation is obtained by pushing forward the equation in
\autoref{prop:Adrienrel} along $d$ and using the
relations~\eqref{eq:xi_zeta_comparision} and~\eqref{eq:psi_comparision}.
The second equation is obtained from the first by \autoref{le:top_degree_comp_new2}.
\end{proof}
\begin{rem}
The expression given by the second line of \autoref{cor:k-adrien} reproves the
formula of \cite[Theorem 3.12]{sauvagetVolumes} and computes explicitly the
coefficients appearing in loc.cit., which were computed only for special
two-level graphs.
\end{rem}
\par
To state the formula for the normal bundle, let
\[
\cL_\pi^\top = \cO_{D^\cQ_\pi}\Big(\sum_{\substack{(\sigma:\wh\Delta_\pmarked\to\Delta) \in \LG_2(\ol\cQ)\\\delta_2(\sigma) = \pi}} \ell_{\wh\Delta,1}D_{\sigma}^\cH\Big)
\]
denote the top level correction bundle.
\begin{cor} \label{cor:normalbundle}
Suppose that $D_\pi$ is a divisor in $\ol \cQ$ corresponding to a covering of graphs $(\pi:\wh\Gamma_\pmarked\to\Gamma) \in \LG_1(\ol\cQ)$. Then the first Chern class of the normal bundle is given by
\[
	c_1(\cN_\pi) \= \frac{1}{\ell_{\wh\Gamma}}\Big(-\frac{1}{k}\zeta_\pi^\top - c_1(\cL_\pi^\top) + \frac{1}{k}\zeta_\pi^\bot\Big) \in \CH^1(D_\pi^\cQ),
	\]
	where  $\zeta_\pi^\top$, resp. $\zeta_\pi^\bot$, is the first Chern class of the line bundle generated by the top, resp. bottom, level multi-scale component.
\end{cor}
\begin{proof}
We can pull-back the right and left hand sides of the relation via $d$. Using the expression \eqref{eq:xi_zeta_comparision}, we see that the pulled-back relation holds since it agrees with the one of \autoref{prop:generalnormalbundle}. Since $d$ is a quasi-finite proper unramified map, we are done. The same argument, together with \autoref{prop:normalhor}, works for the second statement about horizontal divisors.
\end{proof}
\par
Using the same arguments as \cite[Proposition~8.1]{CMZeuler}, it is possible to show an excess intersection formula in this context of $k$-differentials. We will not explicitly do this here since the methods and the result are exactly parallel to the original ones for Abelian differentials. Using the previous ingredients we can then reduce the computation of the class $\beta$ in \eqref{eq:beta} to the computation of a top-degree product of $\psi$-classes 
	\[\alpha := \psi_{1}^{p_1} \cdots \psi_{n}^{p_n} \in \CH_0(\ol \cQ) \]
	on a generalized stratum.
If we can describe the class of a generalized stratum in its corresponding moduli space of pointed curves, then we are done since it is possible to compute top-degree tautological classes on the moduli space of curves, e.g. with the SageMath package {\it admcycles}, see \cite{DSvZ}.
\par
One of the advantages in comparison to the situation with general linear submanifolds (as explained in \autoref{sec:nb}) is that the fundamental classes of strata of primitive $k$-differentials $\bP\kLMS$ are known in $\ol\cM_{g,n}$, see \cite{BHPSS}.
\par
More generally, if $\cQ$ parameterizes $k$-differentials, on a curve with connected $\tau$-quotient, which are $d$-th powers of primitive $k':=k/d$-differentials, we can compare $\psi$-classes on $\ol\cQ$ to $\psi$-classes on the stratum of primitive $k'$ differentials $\bP\kLMS[\mu/d][g,n][k']$ via the diagram
\[
	\begin{tikzcd}
		\cH_k^\pmarked(\mu) \ar[r, "\phi"] \ar[d, "d_1"] & \cH_{k'}^\pmarked(\mu/d) \ar[d, "d_2"] \\
		\cQ & \bP\kLMS[\mu/d][g,n][k']
	\end{tikzcd}
\]
where the map $\phi$ sends the disconnected curve $(\bigcup_{i=1}^d \wh X_i, \bigcup_{i=1}^d \wh \bfz_i, \bigcup_{i=1}^d \omega_i, \tau)$ to $(\wh X_1, \bfz_1, \omega_1, \tau^d|_{\wh X_1})$.
The map $\phi$ has degree $\deg(\phi) = d^{n-1}$, since up to the action of~$\tau$ there are such many ways to distribute the marked points $\wh \bfz$ onto the connected components of $\wh X$.
Using $\deg(d_1) = \frac{1}{k}$ and $\deg(d_2) = \frac{1}{k'}$ we can evaluate~$\alpha$ as
\[
	\int_\cQ \alpha \= d^{n} \int_{\bP\kLMS[\mu/d][g,n][k']} \psi_{1}^{p_1} \cdots \psi_{n}^{p_n}.
\]
\par
If $\cQ$ parameterizes primitive differentials on disconnected curves, then $\int_\cQ \alpha = 0$ since we go down in dimension by looking at the image of the projection to the moduli spaces of curves.
\par
It remains to explain how to evaluate intersection numbers in the presence
of residue conditions. In addition to the space~$\frakR$ defined starting
from a $\tau$-invariant partition~$\lambda_\frakR$ we consider the linear
subspace
\[
R \,:=\, \left\{(r_i)_{i \in H_p} \in \bC^{H_p} \;:\; 
	\begin{array}{c}
		\sum_{i \in \cA^{-1}(\wh X')} r_i = 0 \,\,\text{ for all } \wh X' \in \pi_0(\wh X) \\[2mm]
		r_i = \zeta_k^{-1} r_{\tau(i)} \,\, \text{ for all } i \in H_p
	\end{array}
\right\}
\]
cut out by the residue theorem on each component and the deck transformation.
Recall that $\lambda_\frakR$ is $\tau$-invariant. Let $\lambda_{\frakR_0}$ denote a
subset of $\lambda_\frakR$ obtained by removing one element,
and let $\frakR_0$ denote the new set of residue conditions. For ease of notation
let for now $H_k^\frakR := \bP\Omega \cH_k^{[\frakR,\cA]}$ and $H_k^{\frakR_0}
:= \bP\Omega \cH_k^{[\frakR_0,\cA]}$. If $R \cap \frakR = R \cap \frakR_0$
then $\cH_k^{\frakR} = \cH_k^{\frakR_0}$. So assume that
$R \cap \frakR \neq R \cap \frakR_0$, in which case $\cH_k^{\frakR} \subsetneq
\cH_k^{\frakR_0}$ is a divisor since removing one element from $\lambda_\frakR$
forces to remove its $\tau$-orbit.
For a divisor $D_{\pi}^{\cH_k^{\frakR}} \subseteq \ol\cH_k^{\frakR}$, we denote by $\frakR^\top$ the residue conditions induced by $\frakR$ on the top-level stratum $\cH_k(\pi_{[0]})$. It can be simply computed by discarding from the parts of $\lambda_\frakR$ all indices of legs that go to lower level in $D_{\pi}^{\cH_k^\frakR}$.
Moreover, we denote be $R^\top$ the linear subspace belonging to the top-level stratum of $\pi$ that is cut out by the residue theorem and the deck transformation.
\par
\begin{prop} \label{prop:resolving_residues}
The class of $\ol\cH_k^{\frakR}$ compares inside the Chow ring of $\ol\cH_k^{\frakR_0}$
to the class $\xi$ by the formula
\[
[\ol\cH_k^\frakR] \= -\xi - \sum_{(\pi:\wh\Gamma_\pmarked \to \Gamma) \in \LG^\frakR_1(\ol\cH_k^{\frakR_0})}
\ell_{\wh\Gamma} [D^{\cH_k^{\frakR_0}}_\pi]
- \sum_{(\pi:\wh\Gamma_\pmarked \to \Gamma) \in \LG_{1,\frakR}(\ol\cH_k^{\frakR_0})} \ell_{\wh\Gamma} [D^{\cH_k^{\frakR_0}}_\pi],
\]
where $\LG^\frakR_1(\ol\cH_k^{\frakR_0})$ are the two-level graphs with
$R^\top \cap \frakR^\top = R^\top \cap \frakR_0^\top$, i.e., where the GRC
on top level induced by $\frakR$ does no longer introduce an extra condition,
and where $\LG_{1,\frakR}(\ol\cH_k^{\frakR_0})$ are the two-level graphs where all
the legs involved in the condition forming $\frakR \setminus \frakR_0$ go
to lower level.
\end{prop}
\par
\begin{proof}
The formula is obtained by intersecting the formula in
\cite[Proposition~8.3]{CMZeuler} with $\ol \cH_k^{\frakR_0}$ and thereby using
the transversality statement from \autoref{prop:transverse}. 
\end{proof}
\par
By pushing down this relation along $d$ and applying relation~\eqref{eq:xi_zeta_comparision} we obtain a similar relation for a generalized stratum of $k$-differentials $\cQ^\frakR$ with residue conditions $\frakR$.
\begin{cor}
	The class of $\ol\cQ^{\frakR}$ compares inside the Chow ring of $\ol\cQ^{\frakR_0}$ to the class $\zeta$ by the formula
	\[
		[\ol\cQ^\frakR] \= -\frac{1}{k}\zeta - \sum_{(\pi:\wh\Gamma_\pmarked \to \Gamma) \in \LG^\frakR_1(\ol\cQ^{\frakR_0})} \ell_{\wh\Gamma} [D^{\cQ^{\frakR_0}}_\pi] - \sum_{(\pi:\wh\Gamma_\pmarked \to \Gamma) \in \LG_{1,\frakR}(\ol\cQ^{\frakR_0})} \ell_{\wh\Gamma} [D^{\cQ^{\frakR_0}}_\pi],
	\]
	where $\LG^\frakR_1(\ol\cQ^{\frakR_0})$ are the two-level graphs with $R^\top \cap \frakR^\top = R^\top \cap \frakR_0^\top$, i.e.~where the GRC
	on top level induced by $\frakR$ does no longer introduce an extra condition and where $\LG_{1,\frakR}(\ol\cQ^{\frakR_0})$ are the two-level graphs where all the legs involved in the condition forming $\frakR \setminus \frakR_0$ go to lower level.
\end{cor}
\par
The last expression allows us, in the presence of residue conditions, to reduce
to the previous situations without residue conditions when we want to
evaluate $\alpha$.

%%%%%%%%%%%%%%%%%%%%%%%
\subsection{Values and cross-checks} \label{sec:values}
%%%%%%%%%%%%%%%%%%%%%%%

In this section we provide in Table~\ref{cap:EulerMero2} and Table~\ref{cap:EulerHolo3} some Euler characteristics
for strata of $k$-differentials. We abbreviate
\[\chi_k(\mu) := \chi(\bP \Omega^k_{\mathrm{pr}}\cM_{g,n}(\mu))\]
for the orbifold Euler characteristic of strata of primitive $k$-differentials. Moreover we provide several cross-checks for our values.

\begin{figure}[h]
$$ \begin{array}{|c|c|c|c|c|c|}
\hline  &&&&& \\ [-\halfbls] 
\mu  & (2,2) & (2,1^2) & (1^4) & (5,-1) & (4,1,-1) \\
[-\halfbls] &&&&& \\ 
\hline &&&&& \\ [-\halfbls]
\chi_2(\mu) & -\tfrac{1}{8} & \tfrac{1}{5} & -1 & -\tfrac{7}{15} & \tfrac{6}{5} 
 \\
[-\halfbls] &&&&& \\
\hline
 &&&&& \\ [-\halfbls] 
\mu  & (3,2,-1) & (3,1^2,-1) & (2^2,1,-1) & (2,1^3,-1) & (1^5,-1) \\
[-\halfbls] &&&&& \\ 
\hline &&&&& \\ [-\halfbls]
\chi_2(\mu) & \tfrac{5}{3} & -5 & -6 & 26 & -147
 \\
[-\halfbls] &&&&& \\
\hline
\end{array}
$$
\captionof{table}[foo2]{Euler characteristics of the strata of primitive quadratic differentials in genus $2$ with at most one simple pole.}
\label{cap:EulerMero2}
\end{figure}

The second power of the projectivized Hodge bundle over $\cM_2$ is the union of the strata of quadratic differentials of type $(4)$, $(2,2)$, $(2,1^2)$ and $(1^4)$, if all of them are taken with unmarked zeros.
(Note that there are no quadratic differentials of type~$(3,1)$.)
All quadratic differentials of type $(4)$ are second powers of Abelian differentials of type $(2)$. The stratum $(2,2)$ contains both primitive quadratic differentials and second powers of Abelian differentials of type $(1,1)$.
From Table~\ref{cap:EulerMero2} and \cite[Table~1]{CMZeuler} we read off that
\[
    \chi_1(2) + \frac{1}{2}\chi_2(2,2) + \frac{1}{2}\chi_1(1,1) + \frac{1}{2} \chi_2(2,1^2) + \frac{1}{4!} \chi_2(1^4) = -\frac{1}{80} = \chi(\bP^2) \chi(\cM_2).
\]
Similarly, one checks for the third power of the projectivized Hodge bundle over $\cM_2$ that the numbers in provided in Table~\ref{cap:EulerHolo3} add up to $-\tfrac{1}{48} = \chi(\bP^4)\chi(\cM_2)$.  In the above checks  we have used that $\chi(\cM_2)=-\tfrac{1}{240}$ by \cite{HaZa}.

\begin{figure}[h]
$$ \begin{array}{|c|c|c|c|c|c|c|}
\hline  &&&&&& \\ [-\halfbls] 
\mu  & (6) & (5,1) & (4,2) & (3,3) & (4,1^2) & (3,2,1) \\
[-\halfbls] &&&&&& \\ 
\hline &&&&&& \\ [-\halfbls]
\chi_3(\mu) %\,(\oamoduli(1^{2g-2})) 
& \tfrac{1}{3} & -\tfrac{4}{5} & -\frac{9}{8} & -\frac{4}{3} & \frac{16}{5} & 4 \\
[-\halfbls] &&&&&& \\
\hline  &&&&&& \\ [-\halfbls] 
\mu  & (2^3) & (3,1^3) & (2^2,1^2) & (2,1^4) & (1^6) & \\
[-\halfbls] &&&&&& \\ 
\hline &&&&&& \\ [-\halfbls]
\chi_3(\mu) %\,(\oamoduli(1^{2g-2})) 
& \tfrac{41}{10} & -16 & -\tfrac{52}{3} & 90 & -567 & \\
[-\halfbls] &&&&&& \\
\hline
\end{array}
$$
\captionof{table}[foo2]{Euler characteristics of the strata of primitive holomorphic $3$-differentials in genus $2$.}
\label{cap:EulerHolo3}
\end{figure}

Now consider the second power of the projectivized Hodge bundle twisted by the universal section over $\cM_{2,1}$. It decomposes into the unordered strata $(4)$, $(5,-1)$, $(4,1,-1)$, $(3,2,-1)$, $(2,1^2)$, $(3,1^2,-1)$, $(2^2,1,-1)$, $(2,1^3,-1)$, $(1^5,-1)$, $(4,0)$, $(2^2,0)$, $(2,1^2,0)$, $(1^4,0)$, the ordered stratum $(2^2)$, $(2,1^2)$ (since the zero at the unique marked point is distinguished) and the partially ordered stratum $(1^4)$.
The stratum $(2,1^2)$ appears two times in the list: the first time the unique marked point is the zero of order $2$, the second time it is one of the simple zeros.
On the stratum $(1^4)$ one of the simple zeros is distinguished, while the others may be interchanged.
Note that $\chi_k(m_1, \dots, m_n, 0) = (2-2g-n) \chi_k(m_1, \dots, m_n)$.
The contributions in Table~\ref{cap:EulerMero2} and \cite[Table~1]{CMZeuler} add up to $\tfrac{1}{30} = \chi(\bP^3) \chi(\cM_{2,1})$, where we have used that $\chi(\cM_{2,1})=\tfrac{1}{120}$ by \cite{HaZa}.
\par
We present some further cross-checks suggested by the referee.
\par
The stratum $\bP \Omega^2_{\mathrm{pr}}\cM_{2,3}(2,1,1)$ is isomorphic to the space of $3$-marked curves where the markings are at a Weierstrass point and at two hyperelliptic  conjugate points (see \cite[Thm. 1.2]{Lanneau}). The latter space is isomorphic to $\cM_{0,7} / S_5$, where the symmetric group $S_5$ permutes the first five markings, while the last two markings correspond to the three marked points of the genus two curve under the hyperelliptic map. Then indeed we have $\chi_2(2,1,1)=\chi(\cM_{0,7} / S_5)=1/5$.
\par
Similarly the stratum $\bP \Omega^2_{\mathrm{pr}}\cM_{2,2}(2,2)$ is isomorphic to the space of $2$-marked curves where the markings are at the Weierstrass points. This space is a $(\bZ/2\bZ)$-gerbe over $\cM_{0,6}/S_4$, where the $\bZ/2\bZ$ comes from the hyperelliptic involution. Also in this case we get the correct number $\chi_2(2,2)=\chi(\cM_{0,6} / S_4)/2=-1/8$.
\par
Finally, the stratum $\bP \Omega^2_{\mathrm{pr}}\cM_{2,4}(1,1,1,1)$ decomposes as the disjoint union of three copies of the space of curves with two marked pairs of Weierstrass points (the three possibilities arise as the way of grouping the four markings into two pairs). Each of these copies is a double cover of  $\cM_{0,8}/S_6$, leading to the correct number $\chi_2(1,1,1,1)=6\cdot \chi(\cM_0,8)/(6!)=-1$.

%%%%%%%%%%%%%%%%%%%%%%%%%%%%%%%%%%%%%%%%%%%%%%%%%%%%%%%%%%%%

%%%%%%%%%%%%%%%%%%%%%%%%%%%%%%%%%%%%%%%%%%%%%%%%%%%%%%%%%%%%
%%%%%%%%%%%%%%%%%%%%%%%%%%%%%%%
\section{Ball quotients} \label{sec:BQ}
%%%%%%%%%%%%%%%%%%%%%%%%%%%%%%%

The goal of this section is to prove Theorem~\ref{intro:BQcertificate},
which gives an independent proof of the Deligne-Mostow-Thurston construction
(\cite{DeligneMostow86}, \cite{thurstonshapes}) of ball quotients via cyclic
coverings. For this proof of concept we consider the special case of surfaces,
i.e.\ lattices in $\mathrm{PU}(1,2)$.
\par
We first prove a criterion for showing that a two dimensional smooth Deligne-Mumford
stack is a ball quotient via the Bogomolov-Miyaoka-Yau equality. Such
a criterion exists in many contexts, typically for pairs of a variety and a $\bQ$-divisor
with various hypothesis on the singularities a priori allowed, see for example
\cite{GKPT,GT22}.  We anyway found no criterion for stacks in the literature. Only
the inequality was proven in \cite{MYstacky} and only in the compact case.
\par
We then investigate the special two dimensional strata of $k$-differentials of genus
zero considered in Deligne-Mostow-Thurston, compute all the relevant intersection
numbers and construct, via a contraction of some specific divisor, the smooth
surface stack which we finally show to be a ball quotient.
\par
\subsection{Ball quotient criterion}  We provide a version of the
Bogomolov-Miyaoka-Yau inequality for stacks in the surface case,
based on \cite{KNS}. Singularity terminology and basics about the minimal model program
can be found e.g.\ in \cite{KoMori}.
\par
\begin{prop} \label{prop:BQcrit} Suppose that $\ol{\frakB}$ is a smooth
Deligne-Mumford stack of dimension 2 with trivial isotropy group at
the generic point and let $\cD_1$ be a normal crossing divisor. 
Moreover, suppose that $K_{\ol{\frakB}}(\log \cD_1)^2>0$ and  that
$K_{\ol{\frakB}}(\log \cD_1)$ intersects positively any curve not contained
in $\cD_1$. Then the Miyaoka-Yau inequality
\be\label{eq:MY}
c_1^2(K_{\ol{\frakB}}(\log \cD_1)) \leq  3 c_2(K_{\ol{\frakB}}(\log \cD_1))
\ee
holds, with equality if and only if  $\frakB = \ol{\frakB} \setminus \cD_1$ is
a ball quotient, i.e.\ there is a cofinite lattice $\Gamma \in \PU(1,n)$ such
that $\frakB= [\bB^2 /\Gamma]$ as quotient stack, where
$\bB^2= \{(z_1,z_2) \in \bC^2: |z_1|^2 + |z_2|^2 < 1\}$ is the $2$-ball.
\end{prop}
\par
\begin{proof}
Let $\cD$ be the divisor defined as~$\cD_1$ together with the sum $\cD_2$ of the
divisors $\cD_{2}^i$  with non-trivial isotropy groups of order $b_i$.
Let $\pi:\ol{\frakB}\to \ol B$ be the map to the coarse space and let $D_1=\pi(\cD_1)$,
$D_2=\sum (1-1/b_i)\pi(\cD_{2}^i)$ and $D=D_1+D_2$. 
\par	
We start by assuming that the pair $(\ol B,D)$ is log-canonical and the pair
$(\ol B,D_2)$ is log-terminal. We will show that this assumptions holds in our
situation at the end of the proof.
\par	
Let $\ol B'$ be a log-minimal model given by contracting all the log-exceptional
curves in $D_1$, i.e., contracting all irreducible curves $C\subseteq D_1$ with the
properties $C^2<0$ and $(c_1(K_{\ol B})+[D_1]+[D_2])\cdot C\leq 0$, and
let $D_i'$ be the image of $D_i$, for $i=1,2$. Then 
	\[K_{\ol{B}}(\log D_1)+D_2 \=\pi^*(K_{\ol B'}(\log D'_1)+D_2').\]
Moreover the log-canonical bundle satisfies
	\begin{equation}\label{eq:RH}
		K_{\ol{\frakB}}(\log \cD_1) \= \pi^*(K_{\ol B}(\log D_1)+D_2)\,.
\end{equation}
The fact that the support of the log-exceptional curves is in $\cD_1$,  together
with \eqref{eq:RH}, implies that $K_{\ol B'}+ D_1'+D_2'$ is numerically ample.
By the assumption above on the singularities  we know that $(\ol B,D)$ is log-canonical.
Hence  we are in the situation of applying \cite[Theorem~12]{KNS}.
\par
As a consequence of~\eqref{eq:RH} we know that $c_1^2(K_{\ol{\frakB}}(\log \cD_1))$
coincides with the left hand side of the Miyaoka-Yau inequality of
\cite[Theorem~12]{KNS} applied to $\ol B'$ with boundary divisor $D_1'+D_2'$. 
\par
Moreover, by the Gauss-Bonnet theorem for DM-stacks (see e.g. \cite[Proposition~2.1]{CMZeuler}) we can also identify $c_2(K_{\ol{\frakB}}(\log \cD_1))$ with the right hand side of the inequality of \cite[Theorem~12]{KNS} applied to $\ol B'$ with boundary divisor $D_1'+D_2'$, up to non-log-terminal singularities (similarly as it was done in \cite[Section~3.2]{MYstacky}). By the  assumption above,  the pair $(\ol B,D_2)$ is log-terminal and so the previous identification of the right hand side of \cite[Theorem~12]{KNS} with $c_2(K_{\ol{\frakB}}(\log \cD))$ is true without corrections.
\par
This shows inequality \eqref{eq:MY} and that in the case of equality  $\ol B'\setminus D_1'\cong \ol B\setminus D_1$ is a ball quotient, i.e. $\ol B\setminus D_1\cong \bB^2/\Gamma$. Moreover, in this case, the divisors $D_2^i$ are the branch loci
of~$\pi$ with branch indices $b_i$.
\par
Since $\ol B\setminus D_1$ is the coarse space associated both to $\ol{\frakB}\setminus \cD_1$ and to $[\bB^2/\Gamma]$, this implies that these two DM stacks have to differ by a composition of root constructions along divisors (see e.g. \cite[Section~3.1]{MYstacky}). But since the branch indices of $D_2^i$ can be identified with the isotropy groups of the corresponding divisors in $[\bB^2/\Gamma]$, and since they coincide with the isotropy groups of the corresponding divisor $\ol B\setminus D_1$, we can identify $\ol B\setminus D_1$ with $[\bB^2/\Gamma]$, as non-trivial  root constructions would have changed the size of such isotropy groups.
\par
We are finally left to show the assumption on the singularities. 
First, there exists a resolution  $\widetilde{\frakB}$ of $\ol{\frakB}$  where the proper transform $\widetilde{\cD}$ of $\cD$  is a normal crossing divisor and the exceptional divisors $\cE_i$ are log-exceptional, i.e. $\cE_i^2<0$ and $(c_1(K_{\widetilde{\frakB}})+[\widetilde\cD_1])\cdot \cE_i\leq 0$. Indeed such a resolution can be obtained by blowing-up smooth points of the DM stack, where the numerical conditions can be checked on an étale chart
just as for the usual blow-up of a smooth point of a variety.
\par 
In this situation the corresponding exceptional divisors $E_i$ for the coarse space resolution $\widetilde{B}$ of $\ol{B}$ are also log-exceptional, i.e.,  $(c_1(K_{\widetilde{B}})+[\widetilde D_1]+[\widetilde D_2])\cdot E_i\leq 0$ and $E_i^2\leq 0$. Since contracting log-exceptional divisors does not change the singularity type, this implies that to show that $(\ol B,D_1+D_2)$ is log-canonical  and $(\ol B,D_2)$ is log-terminal, it is enough to show that $(\widetilde{B},\widetilde{D_1}+\widetilde{D_2})$ is log-canonical  and $(\widetilde{B},\widetilde{D}_2)$ is log-terminal. 
\par 
In order to do this, we observe that in general since $(\widetilde{\frakB},
\widetilde{\cD})$ is a smooth DM stack with normal crossing divisor, then
$(\widetilde{B},\widetilde{D_1}+\sum_i \widetilde{D}_2^i)$ is log-canonical. Details
are given in \cite[Theorem~5.1]{CCM}, using  \cite[Proposition A.13]{HHlogcan} . Then we
can use that $\widetilde{B}$ has at worst klt singularities (since it is a surface
with quotient singularities and by \cite[Prop. 4.18]{KoMori}). It is easy to show
that this implies that $(\widetilde{B},\widetilde{D_1}+\sum_i t_i\widetilde{D}_2^i)$
has log-canonical singularities and $(\widetilde{B},\sum_i t_i\widetilde{D}_2^i)$ has
log-terminal singularities, for any $0\leq t_i<1$. The desired statement follows then
by setting $t_i=1-1/b_i$.
\end{proof}

%%%%%%%%%%%%%%%%%%%%
\subsection{Strata of genus zero satisfying (INT)}
%%%%%%%%%%%%%%%%%%%%
Let $(a_1,\ldots,a_5)$ be positive integers
such that $\gcd(a_1,\dots,a_5,k) = 1$ with
\bes
\sum_{i=1}^5 a_i  = 2k,  \quad \text{and for all $i \neq j$} \quad
\Bigl(1-\frac{a_i}k - \frac{a_j}k\Bigr)^{-1} \in \bZ \quad \text{if $a_i + a_j < k$.}
\ees
The first condition states that $\mu=(-a_1,\ldots,-a_5)$ is a
type of a stratum of $k$-differentials on $5$-pointed rational lines and that
the intersection form on eigenspace giving period coordinates has the desired
signature~$(1,2)$. Imposing the gcd-condition lets us work without loss of generality
with primitive $k$-differentials. The last condition is the condition~(INT)
of \cite{DeligneMostow86}. For Deligne-Mostow this condition is key to ensure
that the period map extends as an \'etale map over all boundary divisors.
Thurston \cite{thurstonshapes} uses this condition to show that his cone manifolds
are indeed orbifolds. Mostow completed in \cite{MostowDisc} the $g=0$ picture by showing
that up to the variant $\Sigma\mathrm{INT}$ from \cite{Mostow} these are the
only ball quotient surfaces uniformized by the VHS of a cyclic cover of
$5$-punctured projective line. We recall from \cite[Section~14]{DeligneMostow86}
that there are exactly $27$ five-tuples satisfying INT, and all of them
satisfy in fact the integrality condition INT for all $i \neq j$ with $a_i + a_k
\neq k$.
\par
For us the condition INT has the most important consequence that the
enhancements $\widehat\kappa_e$ of the Abelian covers of the level graphs are
all one. This implies that ghost groups of all strata in this section are trivial.
However the condition INT also enters at other places of the following
computations of automorphism groups and intersection numbers. 
\par
In the sequel we will use the notation $\cQ = \Omega^k \cM_{0,5}(a_1,\dots,a_5)$.
We now list the boundary divisors without horizontal edges. A short case inspection
shows that the only possibilities are the level graphs $\Gamma = \Gamma_{ij}$,
see Figure~\ref{cap:B10bddiv}~left, and $\mathrm{L} = \mathrm{L}_{ij}$,
see Figure~\ref{cap:B10bddiv}~middle, that yield the 'dumbbell' divisors
with two or three legs on bottom level under the condition that
that the $a_i$'s on lower level add up to less than~$k$,
and the level graphs $\Lambda = \cherry$
that yield 'cherry' divisors, see Figure~\ref{cap:B10bddiv}~right ($V$-shaped graphs are
ruled out by $\sum a_i = 2k$).
\begin{figure}
\bes
D_{\Gamma_{45}}=\left[
\begin{tikzpicture}[
baseline={([yshift=-.5ex]current bounding box.center)},
scale=2,very thick,
bend angle=30,
every loop/.style={very thick},
%comp/.style={circle,fill,black,,inner sep=0pt,minimum size=5pt},
comp/.style={circle,black,draw,thin,inner sep=0pt,minimum size=5pt,font=\tiny},
order bottom left/.style={pos=.05,left,font=\tiny},
order top left/.style={pos=.9,left,font=\tiny},
order bottom right/.style={pos=.05,right,font=\tiny},
order top right/.style={pos=1.9,right,font=\tiny},
order node dis/.style={text width=.75cm}]
\node[comp,fill] (T) [] {};
\node [order node dis,above left] (T-1) at (T.north east) {$-a_1$};
\node [minimum width=18pt,below right] (T-2) at (T.south east) {$-a_3$};
\path (T) edge [shorten >=7pt] (T-1.center);
\path (T) edge [shorten >=7pt] (T-2.center);
\node [minimum width=18pt,above right] (T-2) at (T.north east) {$-a_2$};
\path (T) edge [shorten >=7pt] (T-2.center);
\node[comp,fill] (B) [below=of T] {}
edge 
node [order bottom left] {$ $} 
node [order top left] {$ $} (T);
\node [minimum width=18pt,below right] (B-2) at (B.south east) {$-a_5$};
\path (B) edge [shorten >=7pt] (B-2.center);
\node [minimum width=18pt,below left] (B-2) at (B.south west) {$-a_4$};
\path (B) edge [shorten >=7pt] (B-2.center);
\end{tikzpicture}\right]
%\qquad
D_{\mathrm{L}_{12}}=\left[
\begin{tikzpicture}[
baseline={([yshift=-.5ex]current bounding box.center)},
scale=2,very thick,
bend angle=30,
every loop/.style={very thick},
%comp/.style={circle,fill,black,,inner sep=0pt,minimum size=5pt},
comp/.style={circle,black,draw,thin,inner sep=0pt,minimum size=5pt,font=\tiny},
order bottom left/.style={pos=.05,left,font=\tiny},
order top left/.style={pos=.9,left,font=\tiny},
order bottom right/.style={pos=.05,right,font=\tiny},
order top right/.style={pos=.9,right,font=\tiny},
order node dis/.style={text width=.75cm}]
\node[comp,fill] (T) [] {};
\node [order node dis,above left] (T-1) at (T.north east) {$-a_1$};
%\node [minimum width=18pt,below right] (T-2) at (T.south east) {$-a_3$};
\path (T) edge [shorten >=7pt] (T-1.center);
\path (T) edge [shorten >=7pt] (T-2.center);
\node [minimum width=18pt,above right] (T-2) at (T.north east) {$-a_2$};
\path (T) edge [shorten >=7pt] (T-2.center);
\node[comp,fill] (B) [below=of T] {}
edge 
node [order bottom left] {$ $} 
node [order top left] {$ $} (T);
\node [minimum width=18pt,below right] (B-2) at (B.south east) {$-a_5$};
\path (B) edge [shorten >=7pt] (B-2.center);
\node [minimum width=18pt,below left] (B-2) at (B.south west) {$-a_4$};
\path (B) edge [shorten >=7pt] (B-2.center);
\node [minimum width=18pt,above left] (B-2) at (B.north west) {$-a_3$};
\path (B) edge [shorten >=7pt] (B-2.center);

\end{tikzpicture}\right]
%\qquad
%\qquad
D_{\cherry[12][45]} = \left[
\begin{tikzpicture}[
baseline={([yshift=-.5ex]current bounding box.center)},
scale=2,very thick,
bend angle=30,
every loop/.style={very thick},
%comp/.style={circle,fill,black,,inner sep=0pt,minimum size=5pt},
comp/.style={circle,black,draw,thin,inner sep=0pt,minimum size=5pt,font=\tiny},
order bottom left/.style={pos=.05,left,font=\tiny},
order top left/.style={pos=.9,left,font=\tiny},
order bottom right/.style={pos=.05,right,font=\tiny},
order top right/.style={pos=.9,right,font=\tiny},
order node dis/.style={text width=.75cm}]
\node[comp,fill] (T) [] {};
\path[draw] (T) edge [shorten >=7pt] (T-1.center);
\node [order node dis,above left] (T-1) at (T.north east) {$-a_3$};
\node[comp,fill] (B) [below left= 1cm and 0.5cm of T] {}
edge 
node [order bottom left] {$ $}
node [order top left] {$ $} (T); 
\node[comp,fill] (C) [below right= 1cm and 0.5cm of T] {}
edge 
node [order bottom right] {$ $}
node [order top right] {$ $} (T);
\node [minimum width=18pt,below right] (B-2) at (B.south west) {$\!\!\!-a_2$};
\path (B) edge [shorten >=7pt] (B-2.center);
\node [minimum width=18pt,below left] (B-2) at (B.south west) {$-a_1$};
\path (B) edge [shorten >=7pt] (B-2.center);
\node [minimum width=18pt,below right] (C-2) at (C.south east) {$-a_5$};
\path (C) edge [shorten >=7pt] (C-2.center);
\node [minimum width=18pt,below left] (C-2) at (C.south east) {$\,\;-a_4$};
\path (C) edge [shorten >=7pt] (C-2.center);
\end{tikzpicture}\right],
\ees
\caption{Level graphs of boundary divisors for strata
$\omoduli[0,5](a_1,\ldots,a_5)$}
\label{cap:B10bddiv}
\end{figure}
We
define $\kappa_{i,j} := k-(a_i+a_j)$, which is both the $k$-enhancement of the
single edge of $\Gamma_{i,j}$ and the negative of the $k$-enhancement of the single edge of $\mathrm{L}_{i,j}$.
\par
\begin{lemma} \label{le:Spi}
Each of the graphs $\Gamma_{i,j}$, $L_{i,j}$ and $\cherry$ is the codomain of a unique
covering of graphs $\pi \in \LG_1(\ol \cQ)$ and for each such covering $S(\pi) = 1$.
\end{lemma}
\par
\begin{proof}
We will give the argument for $\Gamma_{1,2}$, the argument for the other graphs is similar.
The number of preimages of the vertices of $\Gamma_{1,2}$ is $\gcd(k,a_1,a_2)$ for the
bottom level and $\gcd(k,a_3,a_4,a_5)$ for the top level, while the edge has
$\kappa_{1,2}$ preimages.
\par
We claim that for any cover of graphs $\pi : \wh \Gamma_{\pmarked} \to \Gamma_{1,2}$ the
domain is connected. In fact, suppose there are $k'$ components. This subdivides
the top level and the bottom level into subset of equal size. This implies
$k' \mid \gcd(k,a_1,a_2)$ and $k' \mid \gcd(k,a_3,a_4,a_5)$, and hence $k'=1$
because of $\gcd(k,a_1,\dots,a_5)=1$.
\par
To construct such a cover of graphs it suffices to prescribe one edge
of $\wh \Gamma_{\pmarked}$, the other edges are then forced, since $\tau$-acts
transitively on edges. Since the vertices on top and bottom level are indistinguishable
(forming each one orbit $\tau$-orbit) the resulting graph is independent of the
choice of the first edge. In particular $\wh \Gamma_{\pmarked}$ is unique and
$S(\pi) = 1$.
\end{proof}
\par
Next we compute (self)-intersection numbers of boundary divisors.
\par
\begin{lemma} \label{le:self_intersection_bq}
The self-intersection numbers of the boundary divisors of $\ol\cQ$ are
    \bas \
      [D^\cQ_\Gamma]^2 &\= -\frac{\kappa_{i,j}^2}{k^2} - \sum_{\substack{p < q,\,
          a_p+a_q < k \\ p,q \notin \{i,j\}}}
\frac{\kappa_{i,j}\kappa_{p,q}}{k^2}, \\
      [D^\cQ_L]^2 &\= -\frac{\kappa_{i,j}^2}{k^2}
      \quad \text{and} \quad
	  [D^\cQ_\Lambda]^2 \= - \frac{\kappa_{i,j} \kappa_{p,q}}{k^2}.
\eas
The mutual intersection numbers are
	\bas \
[D^\cQ_\Gamma] \cdot [D^\cQ_\mathrm{L}] &\=
\begin{dcases*}
	\frac{|\kappa_{i,j}\kappa_{p,q}|}{k^2} & if $\Gamma \cap \mathrm{L} \neq \emptyset$ \\
0 & otherwise
\end{dcases*}	\\
[D^\cQ_\Gamma] \cdot [D^\cQ_\Lambda] &\=
\begin{dcases*}
\frac{\kappa_{i,j}\kappa_{p,q}}{k^2} & if $\Gamma \cap \Lambda \neq \emptyset$ \\
0 & otherwise.
\end{dcases*}
\eas
\end{lemma}
\par
\begin{proof}
For the self-intersection numbers consider the formula in
Corollary~\ref{cor:normalbundle}. As remarked above, the condition (INT) implies that all
enhancements of the Abelian coverings are $1$ and hence the same is true for
the $\hat{\ell}$-factor in the corollary.
Let $\Delta_{i,j}^{p,q}$ denote the slanted cherry with points $i,j$ on bottom level and points $p,q$ on middle level.
Together with Corollary~\ref{cor:k-adrien} and Corollary~\ref{cor:normalbundle}
we obtain
\[
[D^\cQ_{\Gamma_i,j}]^2 \= \frac{-1}{k} \zeta^\top - c_1(\cL^\top)
\= -\frac{\kappa_{i,j}^2}{k^2} \int_{\ol\cM_{0,4}} \psi_{1} - \sum_{\substack{p < q,\,
a_p+a_q < k \\ p,q \notin \{i,j\}}} [D_{\Delta_{i,j}^{p,q}}^\cQ].
\]
The degree of the slanted cherry is
\begin{equation} \label{eq:stackfactor}
\int_{\ol\cQ} [D_{\Delta_{i,j}^{p,q}}^\cQ] \= \frac{\kappa_{i,j} \kappa_{p,q}}{k^2}
\end{equation}
by applying the second formula in \autoref{le:top_degree_comp_new2} and \autoref{le:Spi}.
The other numbers are obtained similarly.
\end{proof}
%\par
%\begin{proof}[Proof of Corollary~\ref{cor:xi05}] We treat the case
%that $a_i+a_j <k$ for all $(i,j)$, leaving the details of the other case
%to the reader.
%To evaluate the expression in Corollary~\ref{cor:kstrata}, we need
%to compute top-$\zeta$-integrals over various strata. The top-dimensional
%term gives for any fixed~$i$ by applying Corollary~\ref{cor:k-adrien}
%\bas
%\int_{\overline \cQ} \zeta^2 
%&\= \Bigl( (k-a_i) \psi_i - \sum_{j \neq i} k [D_{\Gamma_{i,j}}] - \sum_{\substack{j,p,q \neq i \\ p < q}} k [D_{\cherry}] \Bigr)^2 \\
%%		&\= (k-a_i)^2 \psi_i^2 + \sum_{j \neq i} k^2 [D_{\Gamma_{i,j}}]^2 + \sum_{\substack{j,p,q \neq i \\ p < q}} k^2 [D_{\cherry}]^2 \\
%%		&\qquad + 2\sum_{\substack{j,p,q \neq i \\ p < q}} k^2 [D_{\Gamma_{i,j}}] [D_{\cherry}] \\
%& \= (k-a_{i})^2 - \sum_{j \neq i} \kappa_{i,j}^2
%\eas
%and by using the intersection number computations above.
%%
%Similarly,
%\bes
%\int_{D_{\Gamma_{ij}}^\top} \zeta \cdot \int_{D_{\Gamma_{ij}}^\bot} 1 \=  \kappa_{i,j}, \quad 
%\int_{D_{\cherry}^\top} 1 \cdot \int_{D_{\cherry}^\bot} \zeta \= - 1. 
%\ees
%and the degree of the slanted cherries has been computed above. 
%To convert to the final formula we used the identity
%\[
%\frac{3}{k^2} \int_{\ol \cQ} \zeta^2 \= 3\Bigl((1-\mu_{i_0})^2 - \sum_{j \neq i_0}
%(1-(\mu_{i_0}+\mu_j))^2\Bigr)
%\=
%\sum_{\substack{i<j \\ i < p < q \\ j \notin \{p,q\}}} (1-(\mu_i+\mu_j))(1-(\mu_p+\mu_q))
%\]
%for any $i_0 = 1, \dots, 5$ that can be checked directly, using
%$\sum \mu_i  =2$.
%\end{proof}
\par
\subsection{The contracted spaces}
We want to construct the compactified ball quotient candidate $\ol\frakB$ from $\ol\cQ$ by
contracting the all the divisors $D_{\mathrm{L}}^\cQ$ and $D_{\Lambda}^\cQ$.
This is in fact possible:
\par
\begin{lemma} \label{le:contractible}
The divisors $D_{\mathrm{L}}^\cQ$ and $D_{\Lambda}^\cQ$ of $\ol\cQ$ are contractible.
The DM-stack $\ol\frakB$ obtained from $\ol\cQ$ by contracting those divisors is smooth.
If  $D_{\wt{\mathrm{L}}}^\frakB$ and $D_{\wt\Lambda}^\frakB$ denote
the points in $\frakB$ obtained by contracting the corresponding divisors in~$\cQ$
then
\[	\int_{\ol\frakB} [D_{\wt{\mathrm{L}}}^{\frakB}] \= \frac{\kappa_{i,j}^2}{k^2} \quad
\text{and} \quad
\int_{\ol\frakB} [D_{\wt\Lambda}^{\frakB}] = \frac{\kappa_{i,j}\kappa_{p,q}}{k^2}.
\]
\end{lemma}
\begin{proof}
For each of the two types of boundary divisors $D_\mathrm{L}^\cQ$ and $D_\Lambda^\cQ$,
we will write a neighborhood~$U$ as quotient stack $[\wt{U}/G]$ with $\wt{U}$ smooth,
and show that
the preimage of the boundary divisor in~$\wt{U}$ is a~$\bP^1$ with self-intersection
number~$-1$. Castelnuovo's criterion then implies that this curve is smoothly
contractible. The order of~$G$ will be $\tfrac{k^2}{\kappa_{i,j}^2}$ for $D_\mathrm{L}^\cQ$
and $\tfrac{k^2}{\kappa_{i,j}\kappa_{p,q}}$ for $D_\Lambda^\cQ$. 	After contracting
the covering $\bP^1$, the quotient is a point with isotropy group~$G$ and the claim
on the degrees follows.
\par
We first consider a cherry divisor $D_\Lambda^\cQ$. Let $D_\Lambda^{\cH_k^\pmarked}$ denote
its preimage in $\cH_k^\pmarked$. As all the Abelian enhancements of the cover
of $\cherry$ are one, the divisor $D_\Lambda^{\cH_k^\pmarked}$ is irreducible, in fact
isomorphic to~$\bP^1$ with coordinates the scales of the differential forms on
the cherries.
\par
We compute the order of the automorphism group of any point $(\wh X, \wh \omega)$
in $D_\Lambda^{\cH_k^\pmarked}$. Suppose first that $(\wh X, \wh \omega)$ is generic.
The irreducible components of $\wh X$ group into three $\tau$-orbits: The components $\wh X^\top$ corresponding to the top-level vertex of $\cherry$, the components $\wh X^\bot_{i,j}$ corresponding to the vertex with marked points $i,j$, and the components $\wh X^\bot_{p,q}$ corresponding to the vertex with marked points $p,q$.
Observe that there are $\kappa_{i,j}$ edges between $\wh X^\top$ and $\wh X^\bot_{i,j}$ and $\kappa_{p,q}$ edges between $\wh X^\top$ and $\wh X^\bot_{p,q}$.
The restriction of $\tau$ to each of the three (not necessarily connected) curves $\wh X^\top$, $\wh X^\bot_{i,j}$, $\wh X^\bot_{p,q}$ has order $k$.
Given an automorphism of the complete curve $\wh X$ its restrictions to~$\wh X^\top$
and $\wh X^\bot_{i,j}$ need to agree on the $\kappa_{i,j}$ nodes, and the analogue
argument applies to $\wh X^\bot_{p,q}$.
Hence after fixing the automorphism on the top-level curve $\wh X^\top$, there are $\tfrac{k^2}{\kappa_{i,j}\kappa_{p,q}}$ possible choices for the automorphism on the two bottom-level curves left.
Together with the $k$ choices for the top-level automorphism, we obtain
\[
	|\Aut(\wh X, \wh \omega)| \= \frac{k^3}{\kappa_{i,j}\kappa_{p,q}}.
\]
As the non-representable map $\cH_k^\pmarked \to \cQ$ has degree $\tfrac{1}{k}$, this
yields that the generic point of $D_\Lambda^\cQ$ has an isotropy group of size
$r := \frac{k^2}{\kappa_{i,j}\kappa_{p,q}}$. Exactly the same argument also applies to the
two boundary points of $D_\Lambda^\cQ$ corresponding to the slanted cherries.
\par
The automorphism group is thus generated by multiplying the transversal $t$-parameter
(compare \autoref{sec:decompLTB}) by an $r$-th root of unity in local charts
covering all of~$\cherry$. We may thus take for~$U$ any tubular neighborhood
of~$D_\Lambda^\cQ$ and take
a global cover~$\wt{U}$ of degree $\frac{k^2}{\kappa_{i,j}\kappa_{p,q}}$. Comparing with
the degree of the normal bundle in \autoref{le:self_intersection_bq} shows that
preimage of~$D_\Lambda^\cQ$ in~$\wt{U}$ is a $(-1)$-curve. 
\par
We now consider a dumbbell divisor $D_\mathrm{L}^\cQ$. As above one checks that the
isotropy group at the generic point of $D_\mathrm{L}^\cQ$ is of
order~$\tfrac{k}{|\kappa_{i,j}|}$ and that the isotropy groups of the boundary points
of the divisor have a quotient group of that order. Consider a tubular neighborhood
of~$D_\mathrm{L}^\cQ$ and a degree $\tfrac{k}{|\kappa_{i,j}|}$ cover that trivializes
the isotropy group at the generic point. Let $\wt D_\mathrm{L}^\cQ$ be the preimage
of the boundary divisor in this cover.
\par
Let $p,q,r$ denote the three marked points on the bottom level of a point in
$\mathrm{L}_{i,j}$. By applying the above line of arguments again, the three boundary
points of $\wt D_\mathrm{L}^\cQ$ have cyclic isotropy groups of sizes
$\tfrac{k}{\kappa_{p,q}}$, $\tfrac{k}{\kappa_{p,r}}$ and $\tfrac{k}{\kappa_{q,r}}$
respectively. The triangle group $T = T(\tfrac{k}{\kappa_{p,q}}, \tfrac{k}{\kappa_{p,r}},
\tfrac{k}{\kappa_{q,r}})$ is always spherical, because $a_i + a_j > k$ implies
$a_p + a_q + a_r < k$ and hence
\[
2-(1-\frac{\kappa_{p,q}}{k}) - (1-\frac{\kappa_{p,r}}{k}) - (1-\frac{\kappa_{q,r}}{k}) = 2-2\frac{a_p + a_q + a_r}{k} > 0.
\]
This implies that the $T$-cover  of $\wt D_\mathrm{L}^\cQ$ ramified to order
${k}/{\kappa_{p,q}}$ along the divisor where~$\{p,q\}$ have come together etc,
trivializes the isotropy groups on the boundary divisor~$\wt D_\mathrm{L}^\cQ$ and
the preimage of $\wt D_\mathrm{L}^\cQ$ is a~$\bP^1$. More precisely, the isotropy
groups of order ${k}/{\kappa_{p,q}}$ do not fix isolated points on the boundary
divisor but have one-dimensional stabilizer, the boundary divisors
intersecting~$\wt D_\mathrm{L}^\cQ$. This implies that the above $T$-cover actually 
provides a chart of a full tubular neighborhood.
\par
It remains to show that $|T| = {k}/{|\kappa_{i,j}|}$ in order to conclude with 
the normal bundle degree from \autoref{le:self_intersection_bq} that this $\bP^1$
is a $(-1)$-curve. To show this, recall that as $T$ is spherical, there are only
the cases $(\tfrac{k}{\kappa_{p,q}}, \tfrac{k}{\kappa_{p,r}},\tfrac{k}{\kappa_{q,r}}) = (2,2,n)$ for $n \in \bN_{\geq 2}$ and $(\tfrac{k}{\kappa_{p,q}}, \tfrac{k}{\kappa_{p,r}},\tfrac{k}{\kappa_{q,r}}) = (2,3,n)$ for $n \in \{3,4,5\}$ to consider.
In the first case the order of $T(2,2,n)$ is $2n$, and assuming that $\tfrac{k}{\kappa_{p,q}} = \tfrac{k}{\kappa_{p,r}} = 2$, one easily checks that $2\tfrac{k}{\kappa_{q,r}} = \tfrac{k}{|\kappa_{i,j}|}$ by using $\sum_i a_i = 2k$.
In the second case the order of $T(2,3,n)$ is $2\lcm(6,n)$, and the claimed equality follows with a similar argument.
\end{proof}
\par
We will now compute the Chern classes of~$\ol\frakB$. Let $c : \ol \cQ \to
\ol\frakB$ denote the contraction map. Let
\[
	\mathbf{\Gamma} := \{(i,j) \;:\; i<j, a_i+a_j<k\} \quad \text{and} \quad
	\mathbf{L} := \{(i,j) \;:\; i<j, a_i+a_j>k\}
\]
be the pairs of integers appearing as indices of the $\Gamma_{i,j}$ and $L_{i,j}$.
Let $\mathrm{I} = \mathrm{I}^{pq}_{ij}$ denote the common degeneration of $\Gamma_{ij}$
and $\mathrm{L}_{pq}$, i.e.~the three-level graph with points $p, q$ on bottom level,
$i,j$ on top level and the remaining point on the middle level. Accordingly, we write
\begin{align*}
	\mathbf{\Lambda} &:= \{(i,j,p,q) \;:\; i<j, i<p<q, j\notin\{p,q\}, a_i+a_j<k, a_p+a_q<k\} \quad \text{and} \\
	\mathbf{I} &:= \{(i,j,p,q) \;:\; i<j, i<p<q, j\notin\{p,q\}, a_i+a_j>k, a_p+a_q<k\}
\end{align*}
for the quadruples of possible indices. Recall that $D_{\hor}$ is the union of all
boundary divisors $D_{H_{ij}}$ whose level graph has a horizontal edge, i.e.\
corresponding to pairs~$(i,j)$ with $a_i + a_j = k$.
We write
\[
    \mathbf{H} := \{(i,j) \;:\; i<j, a_i+a_j=k\}.
\]
\par
We summarize the intersections of the boundary divisors: The cherry~$D^\cQ_{\cherry}$
intersects precisely $D^\cQ_{\Gamma_{ij}}$ and $\Gamma^\cQ_{pq}$. The divisor
$D_{L_{ij}}$ intersects precisely the three divisors $D^\cQ_{\Gamma_{ab}}$ for any
pair $(a,b)$ disjoint from $\{i,j\}$. For the divisor $D^\cQ_{\Gamma_{ij}}$ consider
any pair $(p,q)$ of the three remaining points as $\{p,q,r\}$. This gives an
intersection with a cherry if $a_p + a_q < k$, with a horizontal divisor if $a_p + a_q = k$
and with an $L$-divisor if $a_p + a_q > k$.
Consequently, the divisor
$D^\cQ_{H_{ij}}$ intersects precisely the three divisors $D^\cQ_{\Gamma_{ab}}$ for any
pair $(a,b)$ disjoint from $\{i,j\}$.
\par
\begin{lemma} \label{le:self_intersection_bq_downstairs}
The self-intersection numbers of the boundary divisors of $\ol\frakB$ are
\bes
	[D^\frakB_{\Gamma_{i,j}}]^2 \= -\frac{\kappa_{i,j}^2}{k^2} + \sum_{\substack{p < q,\,
	a_p+a_q > k \\ p,q \notin \{i,j\}}}
	\frac{\kappa_{i,j}^2}{k^2}
	\qquad \text{and} \qquad
	[D^\frakB_{H_{i,j}}]^2 \= -1.
\ees
The mutual intersection numbers are for $\{i,j\}\cap\{p,q\} = \emptyset$ given by
\bes
[D^\frakB_{\Gamma_{i,j}}] \cdot [D^\frakB_{\Gamma_{p,q}}] \= \frac{\kappa_{i,j}\kappa_{p,q}}{k^2}
\qquad \text{and} \qquad
[D^\frakB_{\Gamma_{i,j}}] \cdot [D^\frakB_{H_{p,q}}] \= \frac{\kappa_{i,j}}{k}
\ees
and for $|\{i,j,p\}| = 3$ by
\bas \  
[D^\frakB_{\Gamma_{i,j}}] \cdot [D^\frakB_{\Gamma_{i,p}}] &\=
\begin{dcases*}
	\frac{\kappa_{i,j}\kappa_{i,p}}{k^2} & if $a_i+a_j+a_p < k$ \\
	0 & otherwise.
\end{dcases*}
\eas
\end{lemma}
\begin{proof}
We claim that the pull back of $[D_{\Gamma_{i,j}}^\frakB]$ is given by
\bes
c^* [D_{\Gamma_{i,j}}^\frakB] \= [D_{\Gamma_{i,j}}^\cQ] + \sum_{\substack{p<q,\,a_p+a_q>k\\p,q\notin\{i,j\}}} \frac{\kappa_{i,j}}{|\kappa_{p,q}|} [D_{L_{p,q}}^\cQ] + \sum_{\substack{p<q,\,a_p+a_q<k\\p,q\notin\{i,j\}}} [D_{\cherry}^\cQ].
\ees
To determine the coefficients in the above expression, one may intersect the equation $c^* [D_{\Gamma_{i,j}}^\frakB] = [D_{\Gamma_{i,j}}^\cQ] + \sum_{p,q} l_{p,q} [D_{L_{p,q}}^\cQ] + \sum_{p,q} \lambda_{p,q} [D_{\cherry}^\cQ]$ with unknown coefficients with each of the divisors $[D_{L_{p,q}}^\cQ]$ and $[D_{\cherry}^\cQ]$ in turn.
The left hand side vanishes by push-pull, and the intersection numbers on the right hand side are given by \autoref{le:self_intersection_bq}.
The claimed intersection numbers involving only $\Gamma$-divisors follow again by \autoref{le:self_intersection_bq}.
\par
The pull back of the horizontal divisor is given by $c^*[D_{H_{i,j}}^\frakB] = [D_{H_{i,j}}^\cQ]$.
The intersection number $[D^\frakB_{\Gamma_{i,j}}] \cdot [D^\frakB_{H_{p,q}}]= [D^\cQ_{\Gamma_{i,j}}] \cdot [D^\cQ_{H_{p,q}}]$ follows from \autoref{le:top_degree_comp_new2} and \autoref{le:Spi}.
Finally by \autoref{prop:normalhor} and~\eqref{eq:psicomp}, the normal bundle
of $[D_{H_{i,j}}^\cQ]$ is given by $-\psi_e$ in $\CH(D_{H_{i,j}}^\cQ)$, where
$\psi_e$ is the $\psi$-class supported on the half edge of $H_{i,j}$ that
is adjacent to the vertex with three adjacent marked points.
\end{proof}
\par
\begin{prop} \label{prop:allonBQ} The log canonical bundle on~$\ol\frakB$ has
first Chern class
\be \label{eq:c1onB}
c_1(\Omega^1_{\ol\frakB}(\log D_{\hor}))
\= \sum_{i,j \in \mathbf{\Gamma}} (\frac{k}{2\kappa_{i,j}} - 1)
[D_{\Gamma_{i,j}}^{\frakB}] + \frac{1}{2} [D_{\hor}^{\frakB}] \qquad \text{in} \CH_1(\frakB)
\ee
\par
Its square and the second Chern class are given
by
\be \label{eq:c1squareonB}
c_1(\Omega^1_{\ol\frakB}(\log D_{\hor}))^2 \= 6 - 3\sum_{i,j \in \mathbf{\Gamma}} \frac{\kappa_{i,j}}{k}
+ 3\sum_{i,j \in \mathbf{L}} \frac{\kappa_{i,j}^2}{k^2}
+ 3\sum_{i,j,p,q \in \mathbf{\Lambda}} \frac{\kappa_{i,j}\kappa_{p,q}}{k^2}
\ee
and
\be \label{eq:c2onB}
c_2(\Omega^1_{\ol\frakB}(\log D_{\hor})) \= 2 - \sum_{i,j \in \mathbf{\Gamma}} \frac{\kappa_{i,j}}{k}
+ \sum_{i,j \in \mathbf{L}} \frac{\kappa_{i,j}^2}{k^2}
+ \sum_{i,j,p,q \in \mathbf{\Lambda}} \frac{\kappa_{i,j}\kappa_{p,q}}{k^2}.
\ee
respectively.
\end{prop}
\par
\begin{proof} To derive~\eqref{eq:c1onB} from \autoref{thm:c1cor} we insert into
\bes
c_1(\Omega^1_\cQ(\log D_{\hor})) \= \frac{3}{k} \cdot \zeta +
\sum_{\mathrm{L}} [D_\mathrm{L}^\cQ] + \sum_{\Lambda} [D_\Lambda^{\cQ}]
\ees
that $5\xi - \sum (m_i+k) \psi_i$ is a sum of boundary terms by the
relation~\eqref{cor:k-adrien}. Consider Keel's relation
\bes
\psi_i \= \frac16 \sum_{c <d \atop i \not\in\{c,d\}}  \Delta_{cd} + \frac12
\sum_{a \neq i} \Delta_{ia}\,,
\ees
where $\Delta_{ij}$ is the boundary divisor in $\barmoduli[0,5]$ where the
points $(i,j)$ have come together. We pull back this relation via the forgetful map
$\pi: \bP\kLMS[\mu][0,5] \to \barmoduli[0,5]$. Since this map is a root-stack
construction  and the isotropy groups of the divisors were computed in th
proof of \autoref{le:contractible}, we obtain
\bes
\pi^* \Delta_{ab} \=   \begin{cases}
\frac1{|\kappa_{ab}|} [D_{\mathrm{L}_{ab}}^\cQ] & \text{if $a+b < -k$} \\
[D_{\mathrm{H}_{ab}}]  & \text{if $a+b = -k$} \\
\frac1{\kappa_{ab}} [D_{\mathrm{\Gamma}_{ab}}^\cQ] + \sum_{\substack{i<j,\,a_i+a_j < k\\i,j\notin\{a,b\}}} \frac1{\kappa_{ab}}
[D_{\cherry[i,j][a,b]}^\cQ] & \text{if $a+b > -k$.} \\
 \end{cases}
\ees
Putting everything together we find in $\CH_1(\cQ)$ that
\ba \label{eq:c1onQ} 
c_1(\Omega^1_{\cQ}(\log D_{\hor})) &\= \sum_{i,j \in \mathbf{\Gamma}} (\frac{k}{2\kappa_{i,j}} - 1) [D_{\Gamma_{i,j}}^{\cQ}] + \sum_{i,j \in \mathbf{L}} (\frac{k}{2|\kappa_{i,j}|} - 1) [D_{\mathrm{L}_{i,j}}^{\cQ}] \\
& \phantom{\=}
+ \sum_{i,j,p,q \in \mathbf{\Lambda}} (\frac{k}{2\kappa_{i,j}} + \frac{k}{2\kappa_{p,q}} - 1) [D_{\cherry}^{\cQ}] + \frac{1}{2} [D_{\hor}^\cQ] 
\ea
and since the divisors $D_{\mathrm{L}_{i,j}}^{\cQ}$ and $D_{\cherry}^{\cQ}$ are smoothly
contractible we deduce~\eqref{eq:c1onB}.

To derive~\eqref{eq:c1squareonB} we first note that
$-\tfrac{1}{4}|\mathbf{\Gamma}| + \tfrac{1}{2}|\mathbf{\Lambda}| + \tfrac{5}{4}|\mathbf{H}| + \tfrac{5}{4}|\mathbf{L}| = 5$
and that for $(i,j) \in \mathbf{L}$ the relation
\bes
	1 + \sum_{\substack{p \in \{1,\dots,5\}\setminus \{i,j\} \\ \{q,r\} = \{1,\dots,5\}\setminus\{i,j,p\}}} \left( -\frac{\kappa_{p,q}+\kappa_{p,r}}{k} + 2\frac{\kappa_{p,q}\kappa_{p,r}}{k^2} + \frac{\kappa_{q,r}^2}{k^2} \right) = 4\frac{\kappa_{i,j}^2}{k^2}
\ees
holds because of $\sum_i a_i = 2k$.
Using those relations and the intersection numbers in~\autoref{le:self_intersection_bq_downstairs} squaring~\eqref{eq:c1onB} yields
\bes
	c_1(\Omega^1_{\ol\frakB}(\log D_{\hor}))^2 \= 5
	-\sum_{i,j \in\mathbf{\Gamma}} \left(2\frac{\kappa_{i,j}}{k} + \frac{\kappa_{i,j}^2}{k^2}\right)
	+2\sum_{i,j,p,q \in\mathbf{\Lambda}} \frac{\kappa_{i,j}\kappa_{p,q}}{k^2} \\
	+4\sum_{i,j\in\mathbf{L}} \frac{\kappa_{i,j}^2}{k^2}
\ees
and~\eqref{eq:c1squareonB} follows because $\sum_i a_i = 2k$ implies
\be \label{eq:zero_relation}
	1+\sum_{i,j \in\mathbf{\Gamma}} \left(-\frac{\kappa_{i,j}}{k} + \frac{\kappa_{i,j}^2}{k^2}\right)
	+\sum_{i,j,p,q \in\mathbf{\Lambda}} \frac{\kappa_{i,j}\kappa_{p,q}}{k^2} \\
	-\sum_{i,j\in\mathbf{L}} \frac{\kappa_{i,j}^2}{k^2} \= 0\,.
\ee

The second Chern class can be computed as
\[
	c_2(\Omega^1_{\ol\frakB}(\log D_{\hor})) \= \chi(\cM_{0,5}) + \sum_{i,j \in \mathbf{\Gamma}} \chi(D_{\Gamma_{i,j}}^{\frakB,\circ})
	+ \sum_{i,j \in \mathbf{L}} \chi(D_{\wt L_{i,j}}^{\frakB})
	+ \sum_{i,j,p,q \in \mathbf{\Lambda}} \chi(D_{{}_{i,j} \wt \Lambda_{p,q}}^{\frakB}),
\]
where $\chi(D_{\Gamma_{i,j}}^{\frakB,\circ}) = \chi(D_{\Gamma_{i,j}}^{\cQ,\circ}) = \frac{\kappa_{i,j}}{k}$ be \autoref{le:top_degree_comp_new2} and \autoref{le:Spi} and the Euler characteristics of the points are given in \autoref{le:contractible}.
\end{proof}
\par
\subsection{The ball quotient certificate}
We can finally put together the previous intersection numbers and use our ball quotient criterion to show that the contracted spaces are ball quotients.

\begin{proof}[Proof of \autoref{intro:BQcertificate}]
We apply \autoref{prop:BQcrit} and check that first that the only log-exceptional
curves for $c_1(\Omega^1_{\ol\frakB}(\log D_{\hor}))$ are the components of $D_{\hor}$.
In fact since the expression~\eqref{eq:c1onB} is an effective divisor and since
$\ol{\frakB} \setminus \cD \cong \moduli[0,5]$ is affine, we only have to check
positivity of $c_1^2$ and the intersection with $D_{H_{ab}}$ and
$D_{\Gamma_{i,j}}^{\frakB}$. For the $D_{\Gamma_{i,j}}^{\frakB}$-intersections this follows
from the intersection numbers in \autoref{le:self_intersection_bq_downstairs}.
In fact, the self-intersection number of~$D_{\Gamma_{i,j}}^{\frakB}$ is negative
only if $a_p + a_q \leq k$ for any pair $\{p,q\}$ disjoint from~$\{i,j\}$.
Using \autoref{le:self_intersection_bq} we compute in this case that
\bes
[D_{\Gamma_{i,j}}^{\frakB}] \cdot c_1(\Omega^1_{\ol\frakB}(\log D_{\hor})) \=
\frac{\kappa_{ij}}k \Bigl(\frac{2a_p + 2a_q + 2a_r - a_i - a_j}{k} -1\Bigr)\,,
\ees
where $\{a_1,a_2,a_3,a_4,a_5\} = \{a_i,a_j,a_p,a_q,a_q\}$. Since $a_i+a_j <k$,
this expression is positive. Moreover, one directly computes
\bes
[D_{H_{a,b}}] \cdot c_1(\Omega^1_{\ol\frakB}(\log D_{\hor})) \= 0\,.
\ees
\par
That $c_1(\Omega^1_{\ol\frakB}(\log D_{\hor}))^2 > 0$ is a consequence of the above, as $c_1(\Omega^1_{\ol\frakB}(\log D_{\hor}))$ is by Equation~\eqref{eq:c1onB} a linear combination of the divisors $D_{\Gamma_{i,j}}^\frakB$ and $D^\frakB_{\hor}$ with positive coefficients.
%Adding Equation~\eqref{eq:zero_relation} to \eqref{eq:c1squareonB} yields
%\[
	%c_1(\Omega^1_{\ol\frakB}(\log D_{\hor}))^2 \= 3 - 3\sum_{i,j \in \mathbf{\Gamma}} \frac{\kappa_{i,j}^2}{k^2}
	%+ 6\sum_{i,j \in \mathbf{L}} \frac{\kappa_{i,j}^2}{k^2}.
%\]
%We first consider the case that there are no divisors of the form $D^\frakB_L$.
%To prove the claim in this case it suffices to show that $\sum_{i,j \in \mathbf{\Gamma}}
%\frac{\kappa_{i,j}^2}{k^2} < 1$.
%Since there are no $D^\frakB_L$ divisors the standing assumption  $\sum_i a_i = 2k$
%implies $\sum_{i,j \in \mathbf{\Gamma}} \frac{\kappa_{i,j}}{k} = 2$.
%Recall that we assume that $\frac{k}{\kappa_{i,j}} \in \bZ$. Hence $\frac{\kappa_{i,j}}{k} \leq \frac{1}{2}$ for all $(i,j) \in \mathbf{\Gamma}$ and the inequality must be strict for at least one pair $(i,j)$, implying
%\[
	%\sum_{i,j \in \mathbf{\Gamma}} \frac{\kappa_{i,j}^2}{k^2} < 
	%\sum_{i,j \in \mathbf{\Gamma}} \frac{1}{2} \frac{\kappa_{i,j}}{k} \= 1.
%\]
%The cases with divisors of the form $D^\frakB_L$ can be checked case-by-case.
\end{proof}

%%%%%%%%%%%%%%%%%%%%%%%%%%%%%%%%%%%%%%%%%%%%%%%%%%%%%%%%%%%%

\printbibliography
%\newpage

\end{document}